\documentclass[reqno,11pt]{amsart}
 
\usepackage{amsmath}
\usepackage{amsthm}  
\usepackage{verbatim}
\usepackage{enumerate} 
\usepackage{mathtools}  
\usepackage{amssymb}
\usepackage{mathrsfs}
\usepackage[top=1.5in, bottom=1.5in, left=0.7in, right=0.7in]{geometry}  
\usepackage[colorlinks]{hyperref}
\hypersetup{
	colorlinks,
	citecolor=blue,
	linkcolor=red
}   
\numberwithin{equation}{section}
\usepackage{xypic}
\usepackage{bm}
\usepackage{tikz}
\usetikzlibrary{decorations.pathreplacing}
\usepackage{xcolor}
\usepackage{float}

\newtheorem{theorem}{\textbf{Theorem}}[section]
\newtheorem{theorem*}{\textbf{Theorem}}
\newtheorem{thmx}{Theorem}
 
\newtheorem{corollaryx}[thmx]{\textbf{Corollary}}
\newtheorem{definition}[theorem]{\textbf{Definition}}
\newtheorem{proposition}[theorem]{\textbf{Proposition}}
\newtheorem{lemma}[theorem]{\textbf{Lemma}}
\newtheorem{question}[theorem]{\textbf{Question}}

\newtheorem{corollary}[theorem]{\textbf{Corollary}}
\newtheorem{remark}[theorem]{\textbf{Remark}}

\newtheorem{example}[theorem]{\textbf{Example}}

\newtheorem{definition/proposition}[theorem]{\textbf{Definition/Proposition}}
\newtheorem{innercustomgeneric}{\customgenericname}
\providecommand{\customgenericname}{}
\newcommand{\newcustomtheorem}[2]{%
	\newenvironment{#1}[1]
	{%
		\renewcommand\customgenericname{#2}%
		\renewcommand\theinnercustomgeneric{##1}%
		\innercustomgeneric
	}
	{\endinnercustomgeneric}
}
\newcustomtheorem{customthm}{Theorem}
\newcustomtheorem{customprop}{Proposition}
\newcustomtheorem{customcoro}{Corollary}

\def\N{{\mathbb N}}
\def\R{\mathbb{R}}
\def\Z{{\mathbb Z}}

\def\C{{\mathbb C}}
\def\D{{\mathbb D}}

\def\Q{{\mathbb Q}}

\newcommand{\CP}{\mathbb{C}\mathbb{P}}
\newcommand{\RP}{\mathbb{R}\mathbb{P}}

\def\cA{{\mathcal A}}
\def\cB{{\mathcal B}}
\def\cC{{\mathcal C}}

\def\cH{{\mathcal H}}

\def\cJ{{\mathcal J}}

\def\cL{{\mathcal L}}
\def\cM{{\mathcal M}}
\def\cN{{\mathcal N}}
\def\cO{{\mathcal O}}
\def\cP{{\mathcal P}}

\def\cR{{\mathcal R}}
\def\cS{{\mathcal S}}
\def\cT{{\mathcal T}}
\def\cU{{\mathcal U}}
\def\cV{{\mathcal V}}

\def\bF{{\bm F}}

\def\bH{{\bm H}}

\def\rd{{\rm d}}

\def\la{\langle\,}
\def\ra{\,\rangle}
\def\st{\: \big| \:}

\DeclareMathOperator{\Ima}{im}
\DeclareMathOperator{\ind}{ind}
\DeclareMathOperator{\Hom}{Hom}
\DeclareMathOperator{\Id}{id}

\DeclareMathOperator{\sign}{sign}

\DeclareMathOperator{\supp}{supp}

\DeclareMathOperator{\coker}{coker}

\DeclareMathOperator{\rank}{rank}

\DeclareMathOperator{\End}{End}

\DeclareMathOperator{\virdim}{virdim}

\title{Symplectic fillings of asymptotically dynamically convex manifolds I}
\author{Zhengyi Zhou}

\begin{document}
	\maketitle
\begin{abstract}
We consider exact fillings with vanishing first Chern class of asymptotically dynamically convex (ADC) manifolds. We construct two structure maps on the positive symplectic cohomology and prove that they are independent of the filling for ADC manifolds. The invariance of the structure maps implies that the vanishing of symplectic cohomology and the existence of symplectic dilations are properties independent of the filling for ADC manifolds. Using them, various topological applications on symplectic fillings are obtained, including the uniqueness of diffeomorphism types of fillings for many contact manifolds. We use the structure maps to define the first symplectic obstructions to Weinstein fillability. In particular, we show that for all dimension $4k+3, k\ge 1$, there exist infinitely many contact manifolds that are exactly fillable, almost Weinstein fillable but not Weinstein fillable. The invariance of the structure maps generalizes to strong fillings with vanishing first Chern class. We show that any strong filling with vanishing first Chern class of a class of manifolds, including $(S^{2n-1},\xi_{std}), \partial(T^*L \times \C^n)$ with $L$ simply connected, must be exact and have unique diffeomorphism type.  As an application of the proof, we show that the existence of symplectic dilation implies uniruledness. In particular any affine exotic  $\C^n$ with non-negative log Kodaira dimension is a symplectic exotic $\C^{n}$.
\end{abstract}
\tableofcontents
\section{Introduction}
One natural question in symplectic topology is understanding symplectic fillings of a contact manifold. One aspect of the question is understanding the existence of symplectic fillings. Contact obstructions to the existence of symplectic fillings were first discovered by Eliashberg \cite{eliashberg1990filling}. There are various obstructions to fillings of different flavors, c.f. \cite{massot2013weak} and references therein. There are also topological obstructions to the existence of almost Weinstein fillings \cite{bowden2014topology}. Another aspect of the question is understanding the uniqueness of symplectic fillings. The first result along this line is by Gromov \cite{gromov1985pseudo} that exact fillings of the standard contact $3$-sphere are unique. In \cite{eliashberg1990filling}, Eliashberg proved fillings of $(S^3,\xi_{std})$ are diffeomorphic to blow-ups of the ball. Shortly after, McDuff \cite{mcduff1990structure} generalized the result to lens space $L(p,1)$. In dimension $3$, several uniqueness or finiteness results of symplectic fillings were obtained on $T^3$ \cite{wendl2010strongly}, lens space $L(p,q)$ \cite{lisca2008symplectic} and $S^*\Sigma_g$ \cite{sivek2017fillings}. Only the results on $S^3,L(p,1),T^3$ obtained uniqueness of symplectic fillings, while other results are about topological types of fillings. The dimension $3$ case is special since we have more tools like intersection theory of holomorphic curves and Seiberg-Witten theory. In higher dimensions,  Eliashberg-Floer-McDuff \cite{eliashberg1991on,mcduff1991symplectic} proved that any exact filling of $(S^{2n-1},\xi_{std})$ is diffeomorphic to the ball $B^{2n}$.  The Eliashberg-Floer-McDuff method was generalized by
Oancea-Viterbo \cite{oancea2012topology} to obtain homological information for symplectic aspherical fillings of simply connected subcritically fillable contact manifolds. Barth-Geiges-Zehmisch \cite{barth2016diffeomorphism} extracted refined homological information and showed that symplectic aspherical fillings of simply connected subcritically fillable contact manifolds have unique diffeomorphism types via h-cobordism. 

The symplectic aspect of the uniqueness in higher dimensions remains largely unknown. However there are some evidences. Seidel-Smith \cite{seidel2006biased} showed that any exact filling of $(S^{2n-1},\xi_{std})$ has vanishing symplectic cohomology. We \cite{zhouvanishing} showed that any exact filling with vanishing first Chern class of a simply connected flexibly fillable contact manifold has vanishing symplectic cohomology. It turns out those contact manifolds are asymptotically dynamically convex (ADC) in the sense of Lazarev \cite{lazarev2016contact}\footnote{In fact, all contact manifolds in the first paragraph are ADC, when $c_1=0$.}. The ADC condition is a generalization of the index-positive condition introduced in \cite[\S 9.5]{cieliebak2018symplectic}. A contact manifold $Y^{2n-1}$ is called index-positive, if there is a non-degenerate contact form so that every Reeb orbit $\gamma$ has positive degree, i.e. $\mu_{CZ}(\gamma)+n-3>0$. There are many natural ADC contact manifolds, e.g. boundaries of cotangent bundles of manifolds of dimension at least $4$, boundaries of flexible Weinstein domains, and links of terminal singularities (Remark \ref{rmk:ADC}). The ADC property is a condition on the Conley-Zehnder index, which is suitable for Floer theoretic study. The importance of index-positive/ADC is that positive symplectic cohomology is independent of the filling, hence a contact invariant \cite{cieliebak2018symplectic,lazarev2016contact} via neck-stretching. Combining invariance of positive symplectic cohomology and vanishing of symplectic cohomology, the tautological long exact sequence of symplectic cohomology yields that any exact filling with vanishing first Chern class of a simply connected flexibly fillable contact manifold has the same cohomology group as the flexible filling. 

The above result serves as a basic prototype of studying symplectic fillings of ADC manifolds: We first prove some invariance results on the Floer theory of fillings of ADC manifolds, then we infer invariant symplectic or topological properties from there. The key point in this paper is that the invariance is not limited to some Floer cohomology like positive symplectic cohomology, but also structure maps on those Floer cohomology. The substance of this paper is constructing two structure maps and proving their invariance w.r.t. fillings for ADC manifolds. Then we will derive various symplectic and topological applications from them on both uniqueness and existence aspects of symplectic fillings.

\subsection{Invariance of restriction and persistence of vanishing}
Let $W$ be an exact filling of $Y$, the first structure map $\delta_{\partial}$ is the composition of $\delta:SH^*_+(W)\to H^{*+1}(W)$ and the restriction $H^{*+1}(W) \to H^{*+1}(Y)$, where $\delta$ is the connecting map in the tautological long exact sequence $\ldots \to SH^*(W)\to SH^*_{+}(W) \to H^{*+1}(W)\to \ldots$. In the following, we will restrict to topologically simple fillings of ADC manifolds, i.e. those fillings $W$ such that $c_1(W)=0$ and $\pi_1(Y)\to \pi_1(W)$ is injective (when $Y$ is strongly ADC (Definition \ref{def:ADC}), we only require $c_1(W)=0$\footnote{Roughly speaking, strongly ADC requires that Reeb orbits are contractible in addition to the asymptotically dynamically convex condition.}). Our first theorem is the following.

\begin{thmx}\label{thm:A}
Let $Y$ be a (strongly) ADC contact manifold. Then  $\delta_{\partial}:SH^*_+(W)\to H^{*+1}(Y)$ is independent of topologically simple exact fillings.
\end{thmx}
\begin{remark}
	Unless specified, our coefficient can be any ring with a default setting of $\Z$.
\end{remark}

Since whether $1\in \Ima \delta$ is equivalent to whether $1$ is mapped to $0$ in the unital map $H^*(W)\to SH^*(W)$,  $1\in \Ima \delta_{\partial}$ is equivalent to $SH^*(W)=0$. This reproves the vanishing result in \cite{zhouvanishing}. Moreover, unlike the proof based on the formal properties of symplectic cohomology in \cite{zhouvanishing}, the proof here explains the background geometry to some extent by finding a persistent holomorphic curve.  Moreover, we have the following finer invariance result on the topology of the filling. 
\begin{corollaryx}\label{cor:B}
	 Let $Y$ be a (strongly) ADC contact manifold, then $SH^*(W)=0$ is a property independent of topologically simple exact fillings. In that case, $H^*(W)\to H^*(Y)$ is independent of topologically simple exact fillings,  i.e. for two topologically simple exact fillings $W,W'$, we have an isomorphism $H^*(W)\simeq H^*(W')$ such that the following commutes,
	 $$
	 \xymatrix{ H^*(W) \ar[r] \ar[d]^{\simeq} & H^*(Y) \ar[d]^{=}\\
	 	H^*(W') \ar[r] & H^*(Y)
	 }
	 $$
\end{corollaryx}

The invariance of $H^*(W)\to H^*(Y)$ is a very strong topological constraint, especially when it is injective. In particular, by the universal coefficient theorem, Corollary \ref{cor:B} implies \cite[Theorem 1.2(a)]{barth2016diffeomorphism} if the exact filling is topologically simple. Combining with Theorem \ref{thm:C} below yields a Floer theoretic proof of exact and $c_1=0$ fillings of simply connected subcritically fillable contact manifolds have unique diffeomorphism type by the h-cobordism argument in \cite[\S 5]{barth2016diffeomorphism}. Moreover, we can extract the topology condition required for the h-cobordism argument, to reach the following uniqueness result for some only Liouville fillable contact manifolds (Corollary \ref{cor:ob}) and flexible fillable contact manifolds.
\begin{thmx}\label{thm:diff}
	Exact fillings $W$ with vanishing first Chern class of the following contact manifolds $Y$ with $\dim Y\ge 5$ have unique diffeomorphism types. 
	\begin{enumerate}
		\item\label{f1} $Y=\partial(V\times \C)$ simply connected, where $V$ is a simply connected Liouville domain such that $c_1(V)=0$ and $\dim V >0$.
		\item $Y$ is the boundary of the flexible cotangent bundle $\text{Flex}(T^*Q)$ for simply connected $Q$ with $\chi(Q)=0$.
	\end{enumerate}
\end{thmx} 
\begin{remark}
	The contact boundary $Y:=\partial(V\times \C)$ was also considered in \cite[\S 2.6]{barth2016diffeomorphism}, where Barth-Geiges-Zehmisch proved that for every symplectic aspherical filling $W$ of $Y$, we must have that $H_*(Y) \to H_*(W)$ is surjective. Using Corollary \ref{cor:B}, if we assume $c_1(V)=0$ and $W$ is topologically simple, then $H^*(W)\to H^*(Y)$ is invariant and injective, see Corollary \ref{cor:iso}.
\end{remark}	
\begin{remark}
	$SH^*(W)=0$ is a restrictive condition, the major classes of examples are flexible Weinstein domains \cite{bourgeois2012effect,murphy2018subflexible}, $V\times \C$ for Liouville domains $V$ \cite{oancea2006kunneth}. These two classes provide many examples with ADC boundaries as long as the first Chern class vanishes. Recall that if $V\subset W$ is an exact subdomain, the Viterbo transfer map $SH^*(W)\to SH^*(V)$ is a unital ring map. Since the $0$ ring is the terminal object in the category of unital rings, a symplectic manifold $W$ with $SH^*(W)=0$ should be understood as the simplest symplectic manifold and the contact boundary has the best chance of having unique exact fillings.
\end{remark}
Theorem \ref{thm:A} also holds for symplectic cohomology with local systems, and Corollary \ref{cor:B} holds if the ring structure still exists. In particular, those results can be applied to cotangent bundles considered in \cite{albers2017local}, whose symplectic cohomology without local system does not vanish. In particular, we have the following invariance result for cotangent bundles.

\begin{thmx}\label{thm:local}
	Let $Q$ be an manifold such that the Hurewicz map $\pi_2(Q)\to H_2(Q)$ is nonzero and $ST^*Q$ is ADC (e.g. $\dim Q\ge 4$), then we have the following.
	\begin{enumerate}
		\item\label{l1} $H^*(W;\C) \to H^*(ST^*Q;\C)$ is independent of the topologically simple exact filling $W$ as long as $H^2(W;\Z/p)\to H^2(ST^*Q;\Z/p)$ is surjective for every prime $p$.
		\item\label{l2} If $\pi_1(Q)=0$, and $Q$ is spin, then the independence of  $H^*(W;\C) \to H^*(ST^*Q;\C)$ for topologically simple $W$ holds as long as $H^2(W)\to H^2(ST^*Q)$ is nonzero. If $Q$ is not spin, we need assume $\Ima (H^2(W;\Z/2)\to H^2(ST^*Q;\Z/2))$ contains $\pi^*w_2(Q)|_{ST^*Q}$, where $\pi:T^*Q\to Q$ is the projection and $w_2(Q)\in H^2(Q;\Z/2)$ is the second Stiefel-Whitney class of $Q$.
		\item\label{l3} If in addition, we have $\chi(Q)=0$, then the rational homotopy type of $W$ with same conditions above is $T^*Q$.
	\end{enumerate}	
\end{thmx}
We also study the symplectic cohomology of covering spaces and prove an analogous statement to Theorem \ref{thm:A}, which implies the following theorem. It can be viewed as a generalization of \cite[Theorem 1.2(b)]{barth2016diffeomorphism}.
\begin{thmx}\label{thm:C}
Assume $Y$ is an  ADC contact manifold, with a topologically simple exact filling $W$ such that $SH^*(W) = 0$ (integer coefficient without local systems) and $\pi_1(Y) \to \pi_1(W)$ is an isomorphism.  Then $\pi_1(Y)\to \pi_1(W')$ is an isomorphism for any other topologically simple exact filling $W'$. If $Y$ is strongly ADC with the same property and  $\pi_1(Y)$ is abelian, then $\pi_1(Y) \to \pi(W')$ is an isomorphism for any other topologically simple exact filing $W'$.\footnote{Note that in the strongly ADC context, topologically simple filling only requires $c_1=0$, hence it is a stronger conclusion compared to the ADC case above.}
\end{thmx}

The results above can be put under one theme: understand whether the symplectic filling is unique. It is conjectured that exact fillings of flexibly fillable contact manifolds are unique. Since Theorem \ref{thm:A}-\ref{thm:C} can be applied to a larger class of contact manifolds in addition to flexibly fillable contact manifolds, see \S \ref{s6}. They suggest that contact manifolds with unique exact fillings may go beyond flexibly fillable contact manifolds.

For $(S^{3},\xi_{std})$, Eliashberg \cite{eliashberg1990filling} and McDuff \cite{mcduff1990structure} showed that symplectic fillings of the standard contact $3$-sphere are symplectic blow-ups of the standard ball. Hence it has a unique exact filling. However the procedure of blow-up destroys both exactness and $c_1=0$. Hence in this special case, one can trade the exactness condition to $c_1=0$ condition and still has the uniqueness of fillings, in particular exactness is equivalent to $c_1=0$.   In higher dimension, the exactness/symplectic asphericity plays an important role in \cite{barth2016diffeomorphism,oancea2012topology} while $c_1=0$ is not required or used. But if we study it using Floer theory, non-exactness only adds technical difficulties and can be overcome by a dimension argument \cite{hofer1995floer} when $c_1=0$. However $c_1=0$ plays a fundamental rule, since the ADC property is an index property. Therefore, we can generalize Eliashberg-McDuff's result in the $c_1=0$ direction to higher dimensions as follows.
\begin{thmx}\label{thm:D}
	Let $(Y,\xi)$ be a tamed asymptotically dynamically convex (TADC) manifold (Definition \ref{def:TADC}) with one topologically simple exact filling $W$, such that $SH^*(W;\Q)=0$. Assume $H^2(W;\Q)\to H^2(Y;\Q)$ is injective and $H^1(W;\Q)\to H^1(Y;\Q)$ is surjective. Then any topologically simple strong filling of $Y$ is exact. 
\end{thmx}
The strategy of proving Theorem \ref{thm:D} is showing the invariance of $H^*(W;\Q)\to H^*(Y;\Q)$ like Corollary \ref{cor:B} for strong fillings. Then exactness is a consequence of such invariance. TADC maifolds are more general than index-positive manifolds, but more restricted than ADC manifolds. Examples of TADC manifolds that Theorem \ref{thm:D} can be applied are  boundaries of $\C^n (n\ge 2), T^*M\times \C, W \times \C$, where $W$ is the Milnor fiber of $\sum{x_i^{a_i}}=0,a_i\in \N$ with $\sum \frac{1}{a_i}\ge 1$, and products among them. In particular Theorem \ref{thm:D} implies that any strong filling of $(S^{2n-1},\xi_{std})$ with vanishing first Chern class must be exact\footnote{Although in this special case, proving $H^*(W;\Q)=H^*(B^{2n};\Q)$ is enough. The general case requires the invariance of $H^*(W;\Q)\to H^*(Y;\Q)$.}, hence is diffeomorphic to $B^{2n}$. Theorem \ref{thm:D} can also be applied to non-Weinstein example, e.g. $\partial (V\times \C)$, where $V$ is the exact but not Weinstein domain in \cite{massot2013weak}, since $\partial (V\times \C)$ is TADC by Theorem \ref{thm:product}. Another source of examples are cotangent bundles with local systems, we can get exactness from $c_1=0$ and similar conditions in Theorem \ref{thm:local}, see Theorem \ref{thm:final}.

\begin{remark}\label{rmk:thmD}
	Theorem \ref{thm:D} is expected to hold for ADC manifolds. In the general ADC case, we need to stretch along expanding contact hypersurfaces since non-exact filling may not contain the whole negative end of the symplectization of $Y$. Then we need more functoriality than we are able to get.  However, the expanding issue indicates using a SFT description may be a better way to prove the generalization, see Remark \ref{rmk:SFT} and Remark \ref{rmk:SFT2}.
\end{remark}
\begin{remark}\label{rmk:comp}
	It is worthwhile to compare our method with the method in \cite{barth2016diffeomorphism,mcduff1991symplectic,oancea2012topology}, where the method can be  summarized as finding a ``homological foliation" by rational curves in a partial compactification of the filling. Such method has the benefit of only assuming exactness/symplectic asphericity. Our method is from a very different perspective, which essentially only assumes $c_1=0$ as in Theorem \ref{thm:D}, while the exactness in assumptions of previous theorems are only for the simplicity of the setup. Therefore we cover different aspects of the uniqueness of filling, i.e. $c_1=0$ versus symplectic asphericity. Moreover, we are able to extract some symplectic invariance through the study of symplectic cohomology. By \cite{barth2016diffeomorphism} exact filling $W$ of subcritically fillable contact manifold $Y$ must have $c_1(W)=0$ if $c_1(Y)=0$ when $\dim W \ge 6$. Combining Theorem \ref{thm:D} and Remark \ref{rmk:thmD}, it suggests that for fillings of subcritically fillable contact manifolds, exactness is equivalent to $c_1(W)=0$. In particular, if there is any procedure of modifying a filling, exactness and $c_1(W)=0$ must be destroyed at the same time.  This is the case for blow-up.
\end{remark}
\begin{remark}
	The condition of $c_1=0$ is necessary for the results above to hold. For example, the once blow-up of the standard ball $B^{2n}$, i.e. the total space $\cO(-1)$ of the degree $-1$ bundle over $\CP^{n-1}$, has non-vanishing symplectic cohomology \cite{ritter2014floer} and $H^*(\cO(-1))\to H^*(S^{2n-1})$ is different from $H^*(B^{2n})\to H^*(S^{2n-1})$. By the Viterbo transfer map, such phenomena persist for all exact domains considered above after once blow-up.
\end{remark}

\subsection{Persistence of dilation} The second structure map is related to the symplectic dilation introduced by Seidel-Solomon \cite{seidel2012symplectic}. To explain the structure map, first recall that symplectic cohomology is equipped with a degree $-1$ BV operator $\Delta$. Then a symplectic dilation is an element $x\in SH^1(W)$ such that $\Delta(x)=1$. The existence of symplectic dilation puts strong restrictions on Lagrangians that can be embedded exactly \cite{seidel2014disjoinable,seidel2012symplectic}. On the cochain level, $\Delta$ also respects the splitting into positive and zero symplectic cohomology. Therefore we have a well-defined degree $-1$ map $\Delta_+: SH^*_+(W)\to SH^{*-1}_+(W)$. Then there is a well-defined degree $-1$ map $\Delta_{\partial}:\ker \Delta_+\to \coker \delta_{\partial}$. When $Y$ is ADC and $W$ is topologically simple exact, we have that $\Delta_+$ is independent of the filling. Moreover, by Theorem \ref{thm:A}, $\delta_{\partial}$ is independent of the filling. Then our second main result is the following. 
\begin{thmx}\label{thm:E}
	Let $Y$ be an ADC manifold, then for any two topologically simple exact fillings $W_1,W_2$, we have an isomorphism $\Gamma:SH^*_+(W_1)\to SH^*_+(W_2)$, such that 
\begin{enumerate}
	\item $\delta_{\partial}\circ \Gamma=\delta_{\partial}$,
	\item $\Delta_+\circ \Gamma=\Gamma \circ \Delta_+$,
	\item $\Delta_{\partial}\circ \Gamma=\Delta_{\partial}$.
\end{enumerate}
\end{thmx}
The property of whether $1\in \Ima \Delta_{\partial}$ is closely related to the existence of symplectic dilation. We have the following Corollary. A stronger version concerning symplectic dilation on exact fillings can be found in Corollary \ref{cor:dilation}. 
\begin{corollaryx}\label{cor:F}
	Let $Y$ be an ADC contact manifold of dimension $\ge 5$. Then the existence of symplectic dilation is independent of Weinstein fillings. 
\end{corollaryx}

The vanishing of symplectic cohomology and the existence of symplectic dilations can be understood as the first two levels of indications of the complexity of symplectic manifolds. In fact, there exists a whole hierarchy of structures after them called higher dilations, all of them have associated structure maps similar to $\delta_{\partial},\Delta_{\partial}$, which are also independent of the topologically simple exact filling for  ADC manifolds. Details of the construction will appear in the sequel paper \cite{higher}. 

\begin{remark}\label{rmk:SFT}
	Following \cite{bourgeois2009exact}, positive symplectic cohomology should be understood as the non-equivariant linearized contact homology. When $Y$ is ADC, positive symplectic cohomology can be viewed as non-equivariant cylindrical contact homology, since the augmentation from filling is trivial by degree reason.  $\delta_{\partial},\Delta_{\partial}$ should have an equivalent description using SFT on $Y$. In particular, when the analytic foundation for the full SFT is completed \cite{hofer2017application}, one should be able to strengthen results in this paper to contact manifolds admitting a Reeb flow without a degree zero orbit and its asymptotic version based on the same argument. From SFT point of view, for any linearized non-equivariant contact homology $HC_*(Y)$, one should be able to define a map $HC_*(Y)\to H^{n+1-*}(Y)$ by counting holomorphic curves (with one positive puncture and multiple negative punctures) with one marked point mapped to $Y\times \{0\}$ along with the augmentation. When the augmentation is from a filling, then the map is the composition $SH^*_+(W)\to H^{*+1}(W)\to H^{*+1}(Y)$. It will imply Theorem \ref{thm:A}, since ADC contact manifolds should have no non-trivial augmentation by degree reason. Roughly speaking, the set of fillings is closely related to the set of augmentations. The theme of this paper can be summarized as: if a contact manifold admits a unique augmentation, then many Floer theoretic properties of the filling are independent of the filling. It is possible to prove such claim by reformulising constructions in this paper using SFT. 
\end{remark}

\subsection{Obstructions to Weinstein fillings and cobordisms}
One natural question in the study of symplectic fillings is understanding the difference between exact fillability and Weinstein fillability. In dimension $3$, exact fillable but not Weinstein fillable manifold was found by Bowden \cite{bowden2012exactly}. In higher dimension, such examples were found by Bowden-Crowley-Stipsicz \cite{bowden2014topology}. Their obstruction is topological in nature and their examples are exactly fillable, but not almost Weinstein fillable. Hence the next question we could ask is whether the topological obstruction is sufficient, or whether there is a contact manifold with exact filling and almost Weinstein filling, but no Weinstein filling.  In \S \ref{s6}, we answer the question by proving the following. 
\begin{thmx}\label{thm:I}
	Let $k\ge 1$, there exist infinitely many pairwise non-contactomorphic $4k+3$ dimensional contact manifolds, such that they are exactly fillable, almost Weinstein fillable, but not Weinstein fillable.
\end{thmx}
To prove Theorem \ref{thm:I}, we need new obstructions to Weinstein fillability beyond  topological obstructions.  Using Theorem \ref{thm:A} and \ref{thm:E}, for ADC contact manifolds, $\Ima \delta_{\partial},\Ima \Delta_{\partial}$ contain a nontrivial element of grading higher than $\frac{1}{2}\dim W$ are symplectic obstructions to the existence of Weinstein fillings (Corollary \ref{cor:ob} and Corollary \ref{cor:ob2}).  Such obstructions, to our best knowledge, are the first symplectic obstructions to Weinstein fillability. This answers a Wendl's question  \cite[Question 14]{wendl} on the existence of obstructions to Weinstein fillability of contact structures in higher dimensions. With such obstructions, in addition to proving Theorem \ref{thm:I}, we give simple constructions of many exactly fillable, but not Weinstein fillable manifolds in dimension $\ge 7$. Hence they exist in abundance.

We can also use similar ideas to study symplectic cobordism. Corollary \ref{cor:B} and Theorem \ref{thm:E} show that whether $1 \in \Ima \delta_{\partial},\Ima \Delta_{\partial}$ are actually contact invariants for ADC manifolds. Hence they can be used to develop obstructions to symplectic cobordisms. In particular, we have the following.
\begin{corollaryx}\label{cor:J}
       Let $Y^{2n-1}$ be an ADC contact manifold with a Weinstein filling $W$ such that $c_1(W) = 0$ for $n\ge 3$. Let $V$ be a Weinstein domain. If one of following conditions holds, then there is no Weinstein cobordism from $\partial V$ to $Y$.
       \begin{enumerate}
       	\item If $1\in \Ima \delta_{\partial}$ for $W$, and $1\notin \Ima \delta_{\partial}$ for $V$. 
       	\item If $1\in \Ima \Delta_{\partial}$ for $W$, and $1\notin \Ima \Delta_{\partial}$ for $V$. 
       	\item If $W$ admits a symplectic dilation, and $V$ does not admit a symplectic dilation.
       \end{enumerate} 
\end{corollaryx}
A stronger version of Corollary \ref{cor:J} concerning  obstructions to exact cobordisms can be found in Theorem \ref{thm:corb}. Note that usually, an algebraic obstruction to cobordism will come from SFT type invariants \cite{latschev2011algebraic}, which is difficult to define and compute. However, our obstruction is based on symplectic cohomology, hence is relatively easy to define and compute compared to the SFT. As explained in Remark \ref{rmk:SFT}, it can be understood as an easy case of some SFT obstructions. In \S \ref{s7}, we use Corollary \ref{cor:J} to give many pairs of contact manifolds that admit almost Weinstein cobordisms but no Weinstein cobordism in every dimension $\ge 5$.

\subsection{Constructions of ADC manifolds} ADC contact manifolds exist in abundance. Moreover, subcritical and flexible surgeries preserve the ADC property by the work of Lazarev \cite{lazarev2016contact}. In order to provide examples to Theorem \ref{thm:I}, we prove two more constructions of ADC manifolds, which bear independent interests.
\begin{thmx}\label{thm:K}
	Let $V$ be an exact domain such that $c_1(V)=0,\dim V >0$, then $\partial(V\times \C)$ is ADC. If $V,W$ are two exact ADC domains (Definition \ref{def:ADCfilling}) of dimension at least $4$ with vanishing first Chern classes, then $\partial(V\times W)$ is ADC. 
\end{thmx}
When $V$ is a Weinstein domain, $V\times \C$ is a subcritical Weinstein domain, then the theorem above follows from \cite{lazarev2016contact,yau2004cylindrical}. Repeatedly using subcritical and flexible surgeries and Theorem \ref{thm:K}, we have a lot of examples of ADC contact manifolds, and many of them have either vanishing symplectic cohomology or symplectic dilation, hence Theorem \ref{thm:A} and \ref{thm:E} can be applied.

\subsection{Uniruledness}
At last, we discuss a byproduct of proofs of Theorem \ref{thm:A} and \ref{thm:E}. Uniruledness in the symplectic setting was studied by McLean \cite{mclean2014symplectic}, an exact domain $W$ is called $(k,\Lambda)$-uniruled iff for every point $p$ in the interior $W^0$ there is a proper holomorphic curve in $W^0$ passing through $p$ with area at most $\Lambda$ and the domain Riemann surface $S$ has the property that $\rank H_1(S;\Q)\le k-1$. McLean showed that the algebraic uniruledness for affine varieties is rather a symplectic property. Hence it is reasonable to look for symplectic characterization of uniruledness. 
\begin{thmx}\label{thm:G}
	Let $W$ be an exact domain and there exists a symplectic dilation, then $W$ is $(1,\Lambda)$-uniruled, for some $\Lambda \in \R_+$.
\end{thmx}
The specific value of $\Lambda$ is not relevant, as it is not an invariant w.r.t. the homotopy of Liouville forms. The key property needed in order to use results from \cite{mclean2014symplectic} is that we have such a $\Lambda$ that will bound areas of curves from above. The existence of symplectic dilation is very far from being equivalent to $(1,\Lambda)$-uniruledness. In fact, the existence of $k$-dilation in \cite{higher} will also imply uniruledness. On the other hand, if $W$ is not $1$-uniruled, then $SH^*(W) \ne 0$. By \cite{mclean2014symplectic}, the algebraic uniruledness is equivalent to symplectic uniruledness. Since the log Kodaira dimension provides obstruction to algebraic uniruledness for affine varieties, we have the following corollary, which provides a simple proof of the existence of exotic Stein $\C^n$ for $n\ge 3$ by taking  complex exotic $\C^n$ with non-negative log Kodaira dimension, which exists in abundance \cite{zaidenberg1996exotic}.
\begin{corollaryx}\label{cor:H}
	Let $V$ be an affine variety of non-negative log Kodaira dimension, then $SH^*(V)\ne 0$ and there is no symplectic dilation. In particular, any complex exotic $\C^n$ with non-negative log Kodaira dimension is a symplectic exotic $\C^n$.
\end{corollaryx}
In particular, any affine variety admitting a Calabi-Yau compactification has nonzero symplectic cohomology hence is not displaceable by \cite{kang2014symplectic}. Similar results are investigated in \cite{umut}.

\subsection*{Organization of the paper} 
In \S \ref{s2}, we review symplectic cohomology and give a different treatment of the cochain complex on the zero action part. In \S \ref{s3}, we construct an alternative description of $\delta_{\partial}$ and  prove Theorem \ref{thm:A}-\ref{thm:C} via neck-stretching. Since we will compare cochain complexes on two different fillings, naturality is very important. We carry out a detailed discussion on naturality in \S \ref{s2} and \S \ref{s3}. In \S \ref{s4}, we review the BV operator and symplectic dilation and construct $\Delta_{\partial}$. Then we prove theorem \ref{thm:E} and show that $\Delta_+, \Delta_{\partial}$ are invariants of exact domains up to exact symplectomorphisms. In \S \ref{s5}, we discuss uniruledness and prove Theorem \ref{thm:G} and its corollaries. In \S \ref{s6}, we prove Theorem \ref{thm:K} and use the symplectic obstructions introduced in \S \ref{s3} and \S \ref{s4} to prove Theorem \ref{thm:I}. In \S \ref{s7}, we discuss obstructions to symplectic cobordisms. In \S \ref{s8}, we generalize the construction to strong fillings with vanishing first Chern class and finish the proof of Theorem \ref{thm:D}. We explain our orientation conventions in Appendix \ref{app}.

\subsection*{Acknowledgements}
The author is supported  by  the National Science Foundation under Grant No. DMS-1638352.  It is a great pleasure to acknowledge the Institute for Advanced Study for its warm hospitality.
The author would like to thank Katrin Wehrheim and Kai Zehmisch for helpful comments and is grateful to Oleg Lazarev and Mark McLean for answering many questions. The author also thanks referees' helpful comments and suggestions. This paper is dedicated to the memory of Chenxue.

\section{Symplectic Cohomology}\label{s2}
In this section, we review the basics of symplectic cohomology. However, on the zero action part, we will use a Morse-Bott construction following \cite{diogo2019symplectic}. Such modification is important to the proof of our main theorem and applications to uniruledness. We will first focus on the theory of exact fillings, the situation for non-exact fillings will be discussed in \S \ref{s8}. We assume our contact manifolds and  fillings are connected throughout the paper unless specified otherwise.
\subsection{Hamiltonian-Floer cohomology using cascades}
Let $W$ be a Liouville domain with a Liouville form $\lambda$. Then we have the completion $\widehat{W}:=W\cup \partial W \times [1,\infty) $ with the completed Liouville form $\widehat{\lambda}$. Unless specified otherwise, the Reeb flow on the contact boundary $(\partial W,\lambda|_{\partial W})$ is non-degenerate throughout this paper. Let $\cS$ denote the length spectrum of the Reeb orbits, i.e. the set of periods of periodic orbits. Given a time-dependent Hamiltonian $H$ on $\widehat{W}$, the symplectic action of a \textit{contractible} loop $x:S^1\to \widehat{W}$ is defined by
\begin{equation}\label{eqn:action}
\cA_H(x):=  -\int_{S^1} x^*\widehat{\lambda}+\int_{S^1} H_t\circ x(t) \rd t.
\end{equation}
Symplectic cohomology is defined as the Floer cohomology of \eqref{eqn:action} for a Hamiltonian $H=r^2, r\gg 0$ \cite{seidel2006biased}. Alternatively, symplectic cohomology can be defined as the direct limit of Floer cohomology of $H=Dr, r\gg 0$ for $D\rightarrow \infty$ \cite{cieliebak2018symplectic,ritter2013topological}. In this paper, we will fix one special Hamiltonian, and our convention for Hamiltonian vector field is $\omega(\cdot, X_H)= \rd H$. First we consider a Hamiltonian $H=h(r)$ such that $H=0$ on $W$ and $H=h(r)$ when $r>1$ such that $h''(r)>0$ and $\lim_{r\to \infty}h'(r)=\infty$. Then the periodic orbits of $X_H$ are all points in $W$, and $S^1$ families of non-constant periodic orbits corresponding to Reeb orbits $\gamma$ on $\partial W$ shifted to the level $r$, where $r$ is given by $h'(r)=\int_{\gamma} \lambda|_{\partial W}$. The action of such orbit is then given by $-rh'(r)+h(r)$, which is a strictly decreasing function since $h''(r)>0$. Then following \cite{bourgeois2009symplectic}, we can put a small time-dependent perturbation supported near each $S^1$ family of non-constant periodic orbits, such that the 
periodic orbits of the perturbed Hamiltonian are points in $W$ and pairs of non-degenerate non-constant orbits near each $S^1$ family of non-constant periodic orbits of $H$. We will fix one such perturbed Hamiltonian and call it $\bH$. Let $\cP^*(\bH)$ denote the set of non-constant \textit{contractible} periodic orbits of $\bH$, then $\bH$ has the following properties.
\begin{enumerate}
	\item $\bH=0$ on $W$.
	\item $\bH = h(r)$, such that $h''(r)>0$ on $\partial W \times (1,\rho]$ for some $\rho$ and $h'(\rho)<\min \cS$.
	\item There are non-empty disjoint intervals $(a_i,b_i)$ moving towards infinity with $(a_0,b_0)=(1,\rho)$, such that $\bH|_{\partial W \times (a_i,b_i)}$ is a function $f(r)$ with $f''(r)>0$ and $f'(r)\notin \cS$.
	\item There exists $0=D_0<D_1 \ldots$ converging to $\infty$, such that all periodic orbits of action $\ge-D_i$ are contained in $W^i:=\{r\le a_i\}$. 
\end{enumerate}
\begin{definition}\label{def:admissible}
A Morse function $f$ on $W$ is admissible if $\partial_r f > 0$ on $\partial W$ and $f$ has a unique local minimum. The class of admissible Morse functions is denoted by $\cM(W)$.
\end{definition}
\begin{remark}\label{rmk:MB}
 $\bH$ is not strictly Morse-Bott, since the critical points of $\cA_H$ are not Morse-Bott non-degenerate along $\partial W \subset W$. We will see in Proposition \ref{prop:compact} that such points are invisible to our moduli spaces, also see \cite{diogo2019symplectic}.
\end{remark}

On the symplectization $\partial W \times (0,\infty)$, a compatible almost complex structure is called cylindrical convex if $J$ preserves $\xi=\ker \lambda|_{\partial W}$ and $J(r\partial_r) = R_{\lambda}$, where $R_{\lambda}$ is the Reeb vector field on $(\partial W, \lambda|_{\partial W})$.
\begin{definition}\label{def:convex}
	  A time-dependent almost complex structure $J:S^1\to \End(T\widehat{W})$ is admissible iff the following holds.
	  \begin{enumerate} 
	  	\item $J$ is compatible with $\rd\widehat{\lambda}$ on $\widehat{W}$.
	  	\item $J$ is cylindrical convex on $\partial W \times (a_i,b_i)$.
	  	\item $J$ is $S^1$-independent on $W$.
	  \end{enumerate}
	  The class of admissible almost complex structure is denoted by $\cJ(W)$. 
\end{definition}
For the Hamiltonian $\bH$, we are almost in a Morse-Bott case. We adopt the cascades treatment \cite{diogo2019symplectic} of the Morse-Bott situation, hence we pick any admissible Morse function $f$ on $W$. Let $\cC(f)$ denote the set of critical points of $f$. Let $g$ be a metric on $W$. Then we define the following three moduli spaces. 
\begin{eqnarray}
M_{x,y} & := &\{u:\R_s\times S^1_t \to \widehat{W}\st\partial_s u + J(\partial_t u -X_\bH)=0, \lim_{s\to -\infty} u=x, \lim_{s\to \infty} u = y \}/\R; \label{eqn:1}\\
M_{p,q} & := & \{\gamma:\R_s \to W\st \frac{\rd}{\rd s} \gamma + \nabla_g f=0, \lim_{s\to -\infty}\gamma =p, \lim_{s\to \infty}\gamma = q \}/\R; \label{eqn:2}\\
M_{p,y} &:= & \left\{\begin{array}{l} u:\C \to \widehat{W},\\ \gamma:(-\infty,0] \to W\end{array} \left| \begin{array}{l}\partial_s u + J(\partial_t u -X_\bH)=0, \partial_s \gamma + \nabla_g f=0,\\
\displaystyle\lim_{s\to -\infty}\gamma =p, u(0)=\gamma(0),\lim_{s\to \infty} u = y,
\end{array}\right.\right\}/\R;  \label{eqn:mix}
\end{eqnarray}
where $p,q\in \cC(f),x,y \in \cP^*(\bH)$. In \eqref{eqn:mix}, the $\R$ action is the dilation preserving $0$, which under the polar coordinate $z=e^{2\pi(s+it)}$ is the translation in $s$. Since $\partial_r f > 0$ on $\partial W$, the Floer equation is well-defined on $\C$, because the condition $u(0) = \gamma(0)$ implies $u(0) \in W^0$ (the interior of $W$), where $X_\bH=0$ near it. In particular, the Floer equation is the Cauchy-Riemann equation near $0\in \C$. There will be no $M_{x,q}$, for otherwise the Floer part $u$ will have negative energy.
\begin{remark}
	The Morse-Bott situation here is much simpler than \cite{diogo2019symplectic}, since we only have one non-isolated critical manifold $W$ and its action is the maximum among all critical points. Therefore only three types of moduli spaces above need to be considered, i.e. there is no cascades with multiple levels. Moreover, when one end is asymptotic to a point in $W^0$, i.e. in \eqref{eqn:mix}, the equation degenerates to the Cauchy-Riemann equation, hence there is no new analysis.
\end{remark}

The Gromov-Floer compactification of them is standard and was considered in \cite{diogo2019symplectic}, the only place that needs attention is \eqref{eqn:mix}. In particular, the curve $u$ in \eqref{eqn:mix} could break at a point in $W$. Since we choose $J$ to be convex near $\partial W$ and $u(0)\in W^0$, then the integrated maximum principle below implies that the component breaking at a point in $W$ is contained in $W$. But since $W$ is exact, such configuration can not exist in the compactification. Moreover, since $J$ is cylindrical on $\partial W \times (a_i,b_i)$, the integrated maximum principle prevent curves from escaping to infinity. The following is a special form of the integrated maximum principle of Abouzaid-Seidel \cite{abouzaid2010open}, we recall it from \cite{cieliebak2018symplectic} and state it for strong fillings, since we will use it in \S \ref{s8}.
\begin{lemma}[{\cite[Lemma 2.2]{cieliebak2018symplectic}}]\label{lemma:max}
	Let $(W,\omega)$ be a strong filling of $(Y,\alpha)$ with completion $(\widehat{W},\widehat{\omega})$. Let $H:\widehat{W}\to \R$ be a Hamiltonian such that $H=h(r)$ near $r=r_0$. Let $J$ be a $\widehat{\omega}$-compatible almost complex structure that is cylindrical convex on $Y\times(r_0,r_0+\epsilon)$ for some $\epsilon>0$. If both ends of a Floer cylinder $u$ (i.e. the $u$ component in \eqref{eqn:1} and \eqref{eqn:mix}) are contained inside $Y\times \{r_0\}$, then $u$ is contained inside $Y\times \{r_0\}$. This also holds for $H_s$ depending on $s\in \R$ if $\partial_s H_s \le 0$ on $r>r_0$ and $H_s=h_s(r)$ on $(r_0,r_0+\epsilon)$ such that $\partial_s(r h'_s(r)-h_s(r))\le 0$ and $J_s$ is cylindrical convex on $Y\times (r_0,r_0+\epsilon)$.
\end{lemma}
The original statement of \cite[Lemma 2.2]{cieliebak2018symplectic} requires that $W$ is exact and the conditions on $H$ and $J$ hold on a neighborhood of the hypersurface $r=r_0$. However, their proof is based on integration of the curve outside $r_0$. Moreover precisely, we can choose a sequence of regular values $r_i>r_0$ of $\pi_{\R}\circ u$ such that $r_i\to r_0$. Then the integration argument in \cite[Lemma 2.2]{cieliebak2018symplectic} implies that the curve is inside every $r=r_i$ hypersurface as long the symplectic manifold is exact outside $r=r_0$ and the conditions on $H$ and $J$ hold on the $(r_0,r_0+\epsilon)$ for some $\epsilon>0$.

\begin{proposition}\label{prop:compact}
	The Gromov-Floer compactification of the moduli spaces \eqref{eqn:1}-\eqref{eqn:mix} above are the following.
	\begin{enumerate}
		\item\label{compact1} $\cM_{x,y}:= \cup_{z_1,\ldots,z_k \in \cP^*(\bH)} M_{x,z_1}\times \ldots \times M_{z_k,y}$, for $x,y \in \cP^*(\bH)$.
		\item\label{compact2} $\cM_{p,q}:= \cup_{r_1,\ldots, r_k \in \cC(f)} M_{p,r_1}\times \ldots \times M_{r_k,q}$, for $p,q \in \cC(f)$.
		\item\label{compact3} $\cM_{p,y}:= \cup_{\substack{r_1,\ldots, r_j \in \cC(f) \\ z_1,\ldots,z_k \in \cP^*(\bH)}} M_{p,r_1}\times \ldots \times M_{r_j,z_1}\times \ldots  \times M_{z_k,y}$, for $y\in \cP^*(\bH)$ and $p\in \cC(f)$.
	\end{enumerate}
\end{proposition}
\begin{proof}
	The claim for \eqref{compact2}  is the standard result in Morse theory. In \eqref{compact1}, there are no other breakings due to action reasons. Curves in $\cM_{x,y}$ will not escape to infinity, since Lemma \ref{lemma:max} can be applied to $r\in (a_i,b_i)$ for some $i$ big enough. For \eqref{compact3}, there exists $\delta>0$, such that on $\partial W \times [1-\delta,1]$ we have $\partial_r f>0$. Then the curve $u$ in \eqref{eqn:mix} has the property that $u(0)\in W_{\delta}:=W\backslash(\partial W\times (1-\delta,1])$. Then for $x\in \cP^*(\bH)$, we can consider the following moduli space which will also be used later,
	\begin{equation}\label{eqn:B}
	B_{x}:=\left\{u:\C \to \widehat{W} \st \partial_s u + J(\partial_t u -X_\bH)=0, 
	\displaystyle\lim_{s\to \infty} u = x, u(0)\in W_{\delta}\right\}/\R.
	\end{equation} 
	Then we claim the Gromov-Floer compactification of $B_{x}$ is given by
	$$
	\cB_x:=\bigcup_{x_1,\ldots,x_k\in \cP^*(\bH)} B_{x_1}\times M_{x_1,x_2}\times \ldots \times M_{x_k,x}
	$$
	This is because if there is a limit curve with a nontrivial component $u$ breaks in $W$. Since $\cA_H$ is not Morse-Bott along $\partial W$,  hence it may not be true that $\lim_{s\to \infty} u$ exists. However we still have the limit set of $u$ for $s\to \infty$ is contained in $W$. Since we can apply Lemma \ref{lemma:max} to $r=1$. Hence $u$ is contained in $W$, then the equation degenerates to the Cauchy-Riemann equation and $u$ extends to $\CP^1$ by removal of singularity for $J$-curves. Due to the exactness assumption, there is no such nontrivial curve. It is clear that the compactness of $\cB_x$ implies the claim on \eqref{compact3}.
\end{proof}	

If $c_1(W) = 0$, then there is a $\Z$-grading assigned to elements of $\cC(f)$, $\cP^*(\bH)$ by 
$$|p|=\ind p, \quad p \in \cC(f); \qquad |x|=n-\mu_{CZ}(x), \quad x \in \cP^*(\bH).$$
The convention here is consistent with the grading rule in \cite{ritter2013topological}, i.e. the unit will have grading $0$. For each component of $\cM_{a,b}$, we can assign a well-defined virtual dimension $\virdim \cM_{a,b}:=|a|-|b|-1$. In general, the above assignment defines a $\Z/2$-grading.
\begin{remark}
	Most of the results in this paper except those in \S \ref{s5} require a $\Z$-grading, hence we will impose the $c_1(W)=0$ condition in most situations. To get a $\Z$-grading on symplectic cohomology generated by contractible orbits, one only needs $c_1(W)|_{\pi_2(W)}=0$. One special case of such condition is $c_1(W)$ is torsion. All results in paper holds for this generalized condition. However, it seems that such generalization does not add many new examples. Hence we state all results in the simplified condition $c_1(W)=0$.
\end{remark}

The following transversality is standard in symplectic cohomology. If we allow the Morse function to be generic, then transversality is easy to see. Using generic Morse function is enough for this paper, but for simplicity, we show transversality for any admissible Morse function $f$, this is made possible by that the evaluation map on the universal moduli spaces is submersive \cite[\S 3.4]{mcduff2012j}. 
\begin{proposition}\label{prop:transversality}
For every admissible Morse function $f$, let $g$ be a metric on $W$ such that $(f,g)$ is a Morse-Smale pair. Then there exists a subset $\cJ^{\le1}_{reg}(\bH,f,g)\subset \cJ(W)$ of second category (in particular it is dense), such that for every $J\in \cJ^{\le1}_{reg}(\bH,f,g)$, $\cM_{x,y}$ is a compact manifold of dimension $|x|-|y|-1$ with boundary when $|x|-|y|\le 2$ for $x,y\in \cC(f)\cup \cP^*(\bH)$ and $\partial \cM_{x,y}=\cup_{z\in \cC(f)\cup \cP^*(\bH)}\cM_{x,z}\times \cM_{z,y}$. Here $\le 1$ indicates that transversality holds for moduli spaces up to dimension $1$.
\end{proposition}
\begin{proof}
	The situation for $\cM_{p,q}, p,q \in \cC(f)$ follows from the Morse-Smale assumption, and the situation for $\cM_{x,y}, x,y \in \cP^*(\bH)$ is standard on symplectic cohomology. Hence we need to prove transversality for $\cM_{p,y}, y\in \cP^*(\bH), p\in \cC(f)$.  We consider the uncompactified moduli space $B_y$ in \eqref{eqn:B}.  Using the standard argument by universal moduli spaces and $u$ is somewhere injective \cite{floer1995transversality}, the universal moduli space
	\begin{equation}
	B_{y,\cJ^l}:=\{u:\C \to \widehat{W},J\in \cJ^l(W)| \partial_su+J(\partial_tu-X_{\bH})=0,\lim_{s\to \infty}u=y,u(0)\in W^0\}/\R.
	\end{equation}
	 is cut out transversely, where $\cJ^l(W)$ is the Banach manifold of $C^l$ admissible almost complex structures. Moreover, by \cite[Proposition 3.4.2, Lemma 3.4.3]{mcduff2012j}, $ev_0:B_{x,\cJ^l}\to W^0, u\mapsto u(0)$ is transverse to all stable manifolds of $\nabla_gf$. Then by the Sard-Smale theorem, there is a second category subset of $\cJ^l(W)$, such that the uncompactified $M_{p,y}$ is cut out transversely. Then using the argument of Taubes \cite[Theorem 5.1]{floer1995transversality}, we get a second category subset $\cJ^{\le 1}_{reg}(f,g)$ of smooth admissible almost complex structures where transversality holds\footnote{An alternative approach to get smooth almost complex structures is by considering smooth complex structure with derivative controlled by a sequence $\epsilon_0,\epsilon_1,\ldots$, see \cite[\S 5]{floer1988unregularized}. For results in this paper, using $C^l$ almost complex structures is also sufficient.}. Since the breaking in $\cM_{p,y}$ is either a Morse breaking at a critical point of $f$ or a Floer breaking at a non-constant periodic orbit of $H$, in particular no new gluing analysis is required. This finishes the proof. 
\end{proof}

Moreover, $\cM_{*,*}$ can be equipped with orientations following \cite{floer1993coherent}. We use $\pm$ instead of $\cup$ to indicate the relation on orientations in a union. For example, $\cM_{*,*}$ is oriented such that $\partial \cM_{x,y}=\sum \cM_{x,z}\times \cM_{z,y}$, i.e. the boundary orientation is the product orientation. We explain our orientation convention in Appendix \ref{app}. In the following, we will use the coherent orientations without proof. 

Using the almost complex structure $J$ in Proposition \ref{prop:transversality}, we can define the following cochain complexes by counting $\cM_{*,*}$.
\begin{enumerate}
	\item $C(\bH,J,f)$ is a free $\Z$-module generated by $\cP^*(\bH)\cup \cC (f)$.
	\item $C_0(\bH,J,f)$ is a free $\Z$-module generated by $\cC(f)$.
	\item $C_+(\bH,J,f)$ is a free $\Z$-module generated by $\cP^*(\bH)$.
\end{enumerate}
The differentials are defined by 
$$d y = \sum_x \# \cM_{x,y} x.$$
Since each Floer cylinder has non-negative energy, we have that $\cM_{x,y}\ne \emptyset$ implies that $\cA_{\bH}(x)\ge \cA_{\bH}(y)$. Hence there are finitely many $x$ such that $\cM_{x,y}\ne \emptyset$ for a fixed $y$. In particular, the differential is well-defined. $C_0(\bH,J,f)$ is the Morse complex of $f$ on $W$ and a subcomplex of $C(\bH,J,f)$, hence we abbreviate it to $C_0(f)$. We abbreviate the quotient complex $C_+(\bH,J,f)$ to $C_+(\bH,J)$ since it does not depend on $f$ (even the regularity requirement for $J$ to define $C_+(\bH,J)$ does not depend on $f$). 
\begin{remark}
   In this paper, the default coefficient is $\Z$, since usually the $\Z$-coefficient theory carries more information. However, all results in this paper hold for any coefficient except for the results in \S \ref{s8}, where we need to use the Novikov coefficient over $\Q$. 
\end{remark}
By construction there is a tautological short exact sequence,
\begin{equation}\label{eqn:short}
 0 \to C_0(f) \to C(\bH,J,f) \to C_+(\bH,J) \to 0,
\end{equation}
We use $d_0,d_+$ to denote the differential on $C_0(f)$ and $C_+(\bH,J)$ respectively. Let $d_{+,0}$ denote the map from $C_+(\bH,J)$ to $C_0(f)$ in the definition of $d$ for $C(\bH,J,f)$. Then $d_{+,0}$ is a cochain map $C_+(\bH,J) \to C_0(f)[1]$, and it induces the connecting map $\delta:H^*(C_+(\bH,J))\to H^{*+1}(C_0(f))$ in the induced long exact sequence. $d_{+,0}$ is defined by counting $\cM_{p,y}$ for $y\in \cP^*(\bH)$ and $p\in \cC(f)$. Moreover, since the differential always increases action, for every $i\in \N$, we can define $C^{D_i}(\bH,J,f)$ and $C^{D_i}_+(\bH,J)$ to be the subcomplexes of $C(\bH,J,f),C_+(\bH,J)$ generated by orbits with action $\ge -D_i$, or equivalently those contained in $W^i$. Moreover, by Lemma \ref{lemma:max}, the curve appears in the differential of $C^{D_i}_{(+)}$ is contained in $W^i$. We will call it the length filtration, in the exact case, it coincides the action filtration. Then we have
$$ \varinjlim H^*(C^{D_i}(\bH,J,f))=  H^*(C(\bH,J,f)), \qquad \varinjlim H^*(C_+^{D_i}(\bH,J))=  H^*(C_+(\bH)),$$
similarly for the tautological long exact sequence.

The next proposition shows that what we defined is a model for the symplectic cohomology.
\begin{proposition}\label{prop:cascades}
	There are isomorphisms $SH^*(W) \to H^*(C(\bH,J,f))$ and $SH_+^*(W) \to H^*(C_+(\bH,J))$, such that the following long exact sequences commutes,
	$$
	\xymatrix{
		\ldots \ar[r] & H^*(W) \ar[r] \ar[d] & SH^*(W) \ar[r] \ar[d] & SH^*_+(W) \ar[r] \ar[d] & H^{*+1}(W) \ar[r] \ar[d] & \ldots \\
		\ldots \ar[r] & H^*(C_0(f)) \ar[r]  & H^*(C(\bH,J,f)) \ar[r] & H^*(C_+(\bH,J)) \ar[r]  & H^{*+1}(C_0(f)) \ar[r] & \ldots \\
	}
	$$
\end{proposition}
\begin{proof}
	We prove the isomorphism using a continuation argument from a non-degenerate Hamiltonian. We use another Hamiltonian $H$, such that $H=\bH$ on $\partial W \times [\rho,\infty)$, $H=h(r)$ with $h''(r)>0$ on $\partial W \times[1,\rho]$ and $h'(\rho)<\min \cS$. Moreover, $H\le \bH$ everywhere and $H$ is a $C^2$ small admissible Morse function on $W$. Then for a generic choice of $J_1$, the Floer cohomology of $\cA_{H}$ defines $SH^*(W)$ and $SH^*_+(W)$. Then we choose a homotopy $H_s$ such that $H_s=\bH$ when $s<0$ and $H_s=H$ when $s>1$, $H_s=\bH=H$ when $r\ge \rho$ and $\partial_sH_s\le 0$. Let $J_s$ be a homotopy of admissible almost complex structures such that $J_s=J,s<0$ and $J_s=J_1,s>1$. Let $\cP(H)$ denote the set of \textit{contractible} periodic orbits of $H$, then we have $\cP^*(H)=\cP^*(\bH)$ and $\cP(H)=\cP^*(H)\cup \cC(H|_W)$. Then we define the following moduli spaces, 
	\begin{eqnarray}
	N_{x,y} &:=& \{u\st\partial_s u + J_s(\partial_t u -X_{H_s})=0, \lim_{s\to -\infty} u=x, \lim_{s\to \infty} u = y \}, x \in \cP^*(\bH),y\in \cP^*(H).\\
	N_{p,y} &:=&\left\{\begin{array}{l} u:\C \to \widehat{W},\\ \gamma:(-\infty,0] \to W\end{array} \left| \begin{array}{l}\partial_s u + J_s(\partial_t u -X_{H_s})=0, \partial_s \gamma + \nabla_g f=0,\\
	\displaystyle \lim_{s\to -\infty}\gamma =p, u(0)=\gamma(0),\lim_{s\to \infty} u = y
	\end{array}\right.\right\}, p\in \cC(f), y\in \cP(H). 
	\end{eqnarray}
	Since we assume $\partial_sH\le 0$, $N_{x,y}\ne \emptyset$ implies that $\cA_{\bH}(x)\ge \cA_{H}(y)$ by energy reasons. Therefore if $x\in \cP^*(\bH), y\in \cC(H|_W)$, then $N_{x,y}=\emptyset$. By the same reason in Proposition \ref{prop:compact}, we have the following compactification,
	$$\cN_{a,b}=\bigcup_{\substack{a_1\in \cC(f)\cup \cP^*(\bH),\\  b_1\in \cP(H)}} \cM_{a,a_1}\times N_{a_1,b_1}\times \cM_{b_1,b}.$$
	Similarly, we can find a generic $J_s$ such that transversality holds for all moduli spaces of dimension $\le 1$. This would yield a cochain map $\Phi$ from $C(H,J)$ to $C(\bH,J,f)$. Note that $\Phi$ preserves the length filtration and induces isomorphism on the first page of the spectral sequence. This is because on the action between $-D_{i+1}$ and $-D_i$ part, $\Phi$ induces identity on the first page, since $H_s=\bH=H$ there. This shows $\Phi$ is a quasi-isomorphism on the positive symplectic cohomology. On the part with action close to $0$, the contribution to $\Phi$ is from $\cN_{p,y}$ for $y\in \cC(H|_W)$, then we can apply Lemma \ref{lemma:max} to $r=\rho$. Hence everything is contained in $W\cup \partial W \times [1,\rho]$. Then following \cite{floer1995transversality}, for $H_s|_{W\cup \partial W \times [1,\rho]}$ sufficiently $C^2$-small, there is exists regular $J_s$ independent of $t\in S^1$. Then the moduli space can be identified with the usual continuation map in Morse theory. This shows $\Phi$ is a quasi-isomorphism and induces the commutative diagram of exact sequences.
\end{proof}
\begin{remark}\label{rmk:reverse}
	If we consider the continuation map from $C(\bH,J,f)$ to $C(H,J_1)$, then we need to flow in the direction of $\nabla_g f$ from $p\in \cC(f)$. Thus the degenerate $\partial W$ may not be invisible anymore. This could cause problem in compactness and Fredholm setups. As pointed out in \cite[\S 5]{diogo2019symplectic}, the difficulty can be overcome by choosing a particular homotopy. 
\end{remark}

\subsection{Naturality of the construction}\label{ss:natural}
Since we will need to change almost complex structure in the neck-stretching process and compare things in two different domains, to keep track of the naturality of maps we write down, it is important to specify the choice of almost complex structure and the regularity requirement. To emphasize this aspect, we will always spell out the almost complex structure when needed. 

In order to prove $d^2=0$, we need transversality for moduli spaces up to dimension $1$. However, to define $d_+$ or $d$, we only need transversality for moduli spaces up to dimension $0$. Therefore we have another two larger second category subsets of $\cJ(W)$: $\cJ_{reg,+}(\bH)\supset \cJ_{reg}(\bH,f,g)\supset \cJ^{\le 1}(\bH,f,g)$ so that $d_+$ and $d$ are defined respectively. Similarly we define $\cJ^{D_i}_{reg,+}(\bH)$, $\cJ^{D_i}_{reg}(\bH,f,g)$, i.e. sets of admissible almost complex structures such that $d_+,d$ are defined on $C_+^{D_i}(\bH,J)$ and $C^{D_i}(\bH,J,f)$ respectively. Due to the a priori energy control, $\cJ^{D_i}_{reg}$ are open.

\begin{proposition}\label{prop:lowreg}
	Let $J \in  \cJ_{reg,+}(\bH)$ or $\cJ_{reg}(\bH,f,g)$, then $d_+,d$ are differentials respectively.
\end{proposition}
\begin{proof}
	We will show that $d_+^2=0$ on $C^{D_i}_+(\bH,J)$, the other cases are similar. Since we have an a priori energy bound, by compactness, there exists an open neighborhood $\cU\subset \cJ(W)$ of $J$, such that $\cU\subset \cJ^{D_i}_{reg,+}(\bH)$. Since $\cJ^{\le 1}_{reg}(\bH,f,g)$ is dense,  we choose $J'\in \cU \cap \cJ^{\le 1}_{reg}(\bH,f,g)$. Let $d_{+,J'}$ be the differential defined using $J'$, then $d_{+,J'}^2=0$. We claim $d_+=d_{+,J'}$. Let $J_s, s\in [0,1]$ be a smooth path connecting $J$ and $J'$ in $\cU$, we can consider the moduli space $\cup_{s}\cM_{x,y,J_s}$ for $|x|-|y|=1$. The regularity of each $J_s$ implies it is a compact manifold with boundary $\cM_{x,y,J}$ and $\cM_{x,y,J'}$, since other boundary will involve $\cM_{x',y',J_s}$ with $|x'|-|y'|\le 0$, which is empty by regularity of $J_s$. This proves the claim. 
\end{proof} 
\begin{remark}
	Using $\cJ_{reg}$ instead of $\cJ^{\le 1}_{reg}$ is important in the neck-stretching argument. The index room provided by the ADC condition only allows us to argue that moduli spaces $\cM_{x,y}$ up to dimension zero stays completely in the symplectization of the boundary and are regular after neck-stretching.  This forces us to use $J\in \cJ_{reg}$ instead of $\cJ^{\le 1}_{reg}$, see Remark \ref{rmk:regular} for details.
\end{remark}

Next we recall the continuation map when varying $J$. Given $J_1,J_2\in \cJ_{reg}(\bH,f,g)$ and homotopy $J_s\in \cJ(W)$ such that $J_s=J_2$ for $s<0$ and $J_s=J_1$ for $s>1$. Then we can define the following moduli spaces.
\begin{enumerate}
	\item For $x,y\in \cP^*(\bH)$, $\cN_{x,y}$ is defined to be the compactification of the following
	$$\{u\st\partial_s u + J_s(\partial_t u -X_{\bH})=0, \lim_{s\to -\infty} u=x, \lim_{s\to \infty} u = y \}.$$
	\item For $p\in \cC(f), y\in \cP^*(\bH)$, $\cN_{p,y}$ is defined to the compactification of the following
	$$\left\{\begin{array}{l} u:\C \to \widehat{W},\\ \gamma:(-\infty,0] \to W\end{array} \left| \begin{array}{l}\partial_s u + J_s(\partial_t u -X_{\bH})=0, \partial_s \gamma + \nabla_g f=0,\\
	\displaystyle \lim_{s\to -\infty}\gamma =p, u(0)=\gamma(0), \lim_{s\to \infty} u = y\end{array}\right.\right\}.$$
\end{enumerate}
Then for generic choice of $J_s$, the moduli space $\cN_{a,b}$ is a compact manifold with boundary of dimension $|a|-|b|$ whenever $|a|-|b|\le 1$. Then boundary configuration of $\cN_{x,y}$ implies the following cochain map
$$\Phi_{J_s}:C^*(\bH,J_1,f)\to C^*(\bH,J_2,f), \quad \left\{\begin{array}{lr} y \mapsto \displaystyle \sum_{\substack{|a|-|y|=0,\\ a \in \cP^*(\bH)\cup \cC(f)}}\#\cN_{a,y}a, & y\in \cP^*(\bH),\\
q \mapsto q, & q \in \cC(f). \end{array}\right.$$ 
Similarly for $J_1,J_2\in \cJ_{reg,+}(\bH)$, for generic choice of $J_s$, we defined the following cochain map.
$$\Phi_{J_s,+}: C^*_+(\bH,J_1,f) \to C^*_+(\bH,J_2,f), \quad y \mapsto \sum_{\substack{|x|-|y|=0,\\ y \in \cP^*(\bH)}}\#\cN_{x,y}x.$$ 
Then we have the following standard result on the naturality of the construction.
\begin{proposition}\label{prop:natural}
	Cochain morphisms $\Phi_{J_s}$ and $\Phi_{J_s,+}$ up to homotopy are independent of the choice of $J_s$. Moreover, it is functorial with respect to concatenation of homotopies up to homotopy.
\end{proposition}
This proposition follows from a standard homotopy argument, c.f. \cite{audin2014morse}. Since we are not varying $\bH$, the analytic setups are similar to Proposition \ref{prop:compact} and Proposition \ref{prop:transversality}. Hence we omit the proof. By Proposition \ref{prop:natural}, we may suppress $J_s$ in $\Phi,\Phi_+$. Another observation is that $\Phi(C^{D_i}(\bH,J,f))\subset C^{D_i}(\bH,J,f)$, similarly for $C^{D_i}_+(\bH,J)$, we use $\Phi^{D_i},\Phi^{D_i}_+$ to denote the restrictions, they also satisfy Proposition \ref{prop:natural}. Similar to Proposition \ref{prop:lowreg}, we only need regularity of $J_s$ for moduli spaces up to dimension $0$ to get well-defined $\Phi^{D_i},\Phi^{D_i}_+$. The key property we need in the neck-stretching is the following.
\begin{lemma}\label{lemma:local}
	Assume we have a smooth family $J_s:[0,1]\to \cJ^{D_i}_{reg}(\bH,f,g)$ or $\cJ^{D_i}_{reg,+}(\bH)$. Then $\Phi^{D_i}:C^{D_i}(\bH,J_a,f)\to C^{D_i}(\bH,J_b,f)$ or $\Phi^{D_i}_+:C^{D_i}_+(\bH,J_a,f)\to C^{D_i}_+(\bH,J_b,f)$ are identities (up to homotopy)  for any $0\le a, b \le 1$ respectively. 
\end{lemma}
\begin{proof}
	We prove the $C^{D_i}$ case, as the $C_+^{D_i}$ case is similar. It is sufficient to prove the following: for any $a\in [0,1]$,  there exists a $\delta>0$ such that for every $|b-a|<\delta$, we have $\Phi^{D_i}$ is homotopic to identity form the chain complex using $J_a$ to the one using $J_b$. Let $\rho(s)$ be a smooth function such that $\rho(s)=0$ for $s\le 0$ and $\rho(s)=1$ for $s\ge 1$. Then we have a family of homotopies of almost complex structures $J_{\epsilon,s}=J_{a+\epsilon\rho(s)}$ for $\epsilon$ in a neighborhood of $0$.  In the following, by a regular homotopy of almost complex structures, we mean a homotopy of almost complex structures such that moduli spaces up to dimension $0$ in the definition of continuation map are cut out transversely. Then $J_{0,s}\equiv J_a$ is a regular homotopy of almost complex structures and the continuous map is identity corresponding to constant cylinders over periodic orbits. Since there is a universal energy bound when restricted on $C^{D_i}$, by compactness, there exists $\delta>0$ such that if $|\epsilon|<\delta$, $J_{\epsilon,s}$ is a regular homotopy. It is clear the trivial cylinders on periodic orbits contributes to the continuation map. But there are no other contribution, due to compactness and that $J_{0,s}$ has no other contributions to the continuation map.
\end{proof}
\begin{remark}
	From Proposition \ref{prop:lowreg}, $d$ is constant on each $J_s$. Lemma \ref{lemma:local} says the obvious identification is natural, i.e. it is the continuation map, which has the functorial property in Proposition \ref{prop:natural}.
\end{remark}
In Proposition \ref{prop:lowreg}, we know that $d_{+}$ on $C^{D_i}_+$ is defined using $J\in \cJ^{D_i}_{reg,+}(\bH)$. We still need to verify that $d_{+,0}$ is well-defined on $C^{D_i}_+$ for later application. Following the proof of Proposition \ref{prop:lowreg}, let $\cU\subset \cJ(W)$ be an open neighborhood of $J$ that is contained in $\cJ^{D_i}_{reg,+}(\bH)$. Pick any $J'\in \cU\cap \cJ_{reg}(\bH,f,g)$, then $d_{+,0,J'}$ is defined, and by Proposition \ref{prop:lowreg} it is a cochain map using $d_+$ defined by $J$ since $d_+$ is locally constant.
\begin{proposition}\label{prop:lowreg2}
	$d_{+,0}$ above is well-defined on $C^{D_i}_+$ up to homotopy.
\end{proposition}
\begin{proof}
	Let $J''\ne J'\in \cU \cap \cJ_{reg}(\bH,f,g)$, then there exists a regular homotopy $J_s$ connecting $J''$ and $J'$, then the continuation map gives the following relation
	$$\Phi_{+,0}\circ d_{+,J''}+d_{+,0,J''}=d_{+,0,J'}\circ \Phi_{+}+d_0\circ \Phi_{+,0}.$$
	By Lemma \ref{lemma:local} and Proposition \ref{prop:natural}, $\Phi_+$ is homotopic to identity.
	By the argument in Proposition \ref{prop:lowreg}, $d_{+,J'}=d_{+,J''}=d_+$. This implies that $d_{+,0,J''}$ and $d_{+,0,J'}$ are homotopic as cochain maps from $C_+^D(\bH,J)\to C_0(f)[1]$. 
\end{proof}

The above discussion shows that, down to an action lower bound, we only need regularity for moduli spaces up to dimension $0$ to define differentials or cochain maps up to homotopy.
\section{Invariance of restriction map and persistence of vanishing}\label{s3}
In this section we prove Theorem \ref{thm:A}, the strategy of the proof is contained in following picture, which we will explain in detail.
\begin{figure}[H]
	\begin{tikzpicture}
	\draw (0,0) to [out=90,in=180] (0.25,0.5);
	\draw (0.25,0.5) to [out=0,in=90] (0.5,0);
	\draw (0.5,0) to [out=270,in=0] (0.25,-0.5);
	\draw (0.25,-0.5) to [out=180,in=270] (0,0);
	\draw (0.25,-0.5) -- (-1,-0.5);
	\draw (0.25,0.5) -- (-1,0.5);
	\draw (-1.5,0) to [out=90,in=180] (-1,0.5);
	\draw (-1,-0.5) to [out=180,in=270] (-1.5,0);
	\draw[->] (-1.5,0) to [out=180, in=250](-2,1);
	\draw (-2,1) to [out=70, in=220](-1.5,1.5);
	\fill (-1.5,1.5) circle[radius=1pt];
	\node at (-1.5, 1.7) {$q$};
	\node at (-2.5,1) {$\nabla_g f$};
	\draw[dotted] (-0.5,5) -- (-0.5,-1);
	\draw[->] (-1.5,1.5) -- (-1, 1.5);
	\draw (-1,1.5) -- (-0.5, 1.5); 
	\draw[->] (-0.5,1.5) -- (-0.5, 2); 
	\draw (-0.5,2) -- (-0.5, 2.5);
	\fill (-0.5,2.5) circle[radius=1pt];
	\node at (-0.3, 2.5) {$p$};
	\node at (0, 2) {$\nabla_{g_{\partial}} h$};
	\node at (-0.5, 4) {$\partial W$};
	\node at (-0.5, -1.5) {$l=\infty$};
	
	\draw (4,0) to [out=90,in=180] (4.25,0.5);
	\draw (4.25,0.5) to [out=0,in=90] (4.5,0);
	\draw (4.5,0) to [out=270,in=0] (4.25,-0.5);
	\draw (4.25,-0.5) to [out=180,in=270] (4,0);
	\draw (4.25,-0.5) -- (3,-0.5);
	\draw (4.25,0.5) -- (3,0.5);
	\draw (2.5,0) to [out=90,in=180] (3,0.5);
	\draw (3,-0.2) to [out=180,in=270] (2.5,0);
	\draw (3,-0.2) to [out=0, in=180]  (3,-0.5);
	\draw[->] (2.5,0) to [out=180, in=250](2,1);
	\draw (2,1) to [out=70, in=220](3.5,1.5);
	\draw[dotted] (3.5,5) -- (3.5,-1);
	\node at (3.5, 4) {$\partial W$};
	\draw[->] (3.5,1.5) -- (3.5, 2); 
	\draw (3.5,2) -- (3.5, 2.5);
	\fill (3.5,2.5) circle[radius=1pt];
	\node at (3.7, 2.5) {$p$};
	\node at (3.5,-1.5) {$0<l<\infty$};

	\draw (8,0) to [out=90,in=180] (8.25,0.5);
	\draw (8.25,0.5) to [out=0,in=90] (8.5,0);
	\draw (8.5,0) to [out=270,in=0] (8.25,-0.5);
	\draw (8.25,-0.5) to [out=180,in=270] (8,0);
	\draw (8.25,-0.5) -- (7,-0.5);
	\draw (8.25,0.5) -- (7,0.5);
	\draw (6.5,0.25) to [out=90,in=180] (7,0.5);
	\draw (7,0) to [out=180,in=270] (6.5,0.25);
	\draw (7,0) to [out = 0, in = 0] (6.2, -0.25);
	\draw (6.2, -0.25) to [out=180, in=90] (6, -0.35);
	\draw (6,-0.35) to [out=270, in=180](7,-0.5);
	\draw[dotted] (6.5,5) -- (6.5,-1);
	\node at (6.5, 4) {$\partial W$};
	\draw[->] (6.5,0.25) -- (6.5, 1.5); 
	\draw (6.5,1.5) -- (6.5, 2.5);
	\fill (6.5,2.5) circle[radius=1pt];
	\node at (6.7, 2.5) {$p$};
	\node at (6.5,-1.5) {$l=0$};
	
	\draw (14,0) to [out=90,in=180] (14.25,0.5);
	\draw (14.25,0.5) to [out=0,in=90] (14.5,0);
	\draw (14.5,0) to [out=270,in=0] (14.25,-0.5);
	\draw (14.25,-0.5) to [out=180,in=270] (14,0);
	\draw (14.25,-0.5) -- (13,-0.5);
	\draw (14.25,0.5) -- (13,0.5);
	\draw (12.5,0) to [out=90,in=180] (13,0.5);
	\draw (13,-0.5) to [out=180,in=270] (12.5,0);
	\draw[dotted] (12.5,5) -- (12.5,-1);
	\node at (12.5, 4) {$\partial W$};
	\draw[->] (12.5,0) -- (12.5, 1.5); 
	\draw (12.5,1.5) -- (12.5, 2.5);
	\fill (12.5,2.5) circle[radius=1pt];
	\node at (12.7, 2.5) {$p$};
	\node at (12.5,-1.5) {After neck-streching along $Y$};
	\draw[dotted] (10.5,5) -- (10.5,-1);
	\node at (10.5, 4) {$Y$};
	\end{tikzpicture}
	\caption{Pictorial proof of Theorem \ref{thm:A}.}\label{fig:pic}
\end{figure}
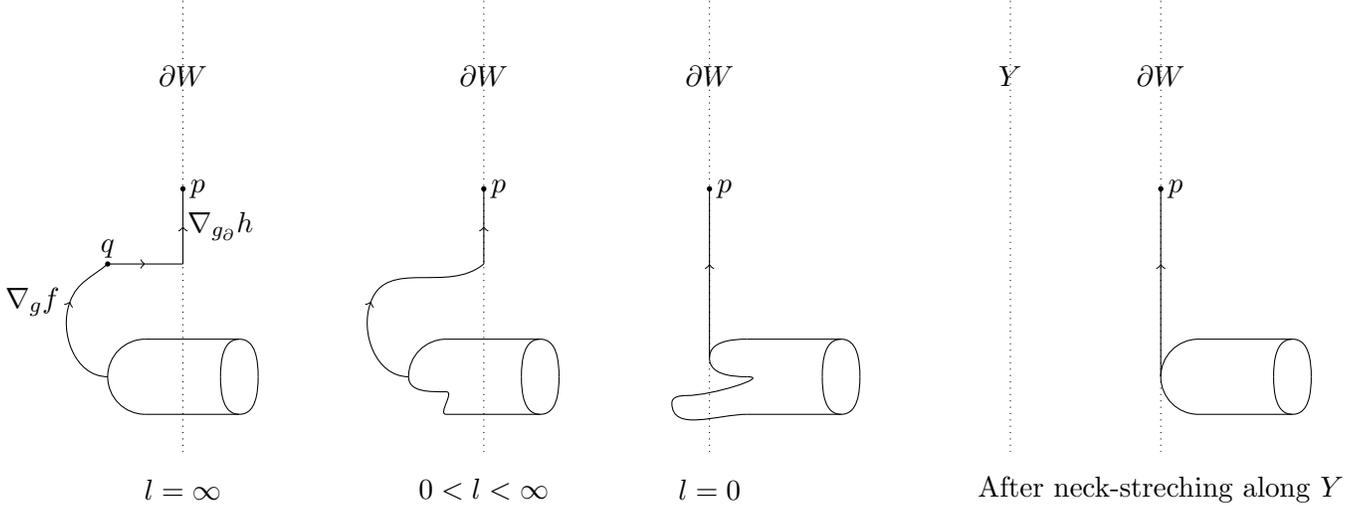

\subsection{Morse description of the restriction $H^*(W) \to H^*(\partial W)$}\label{sub:shrink}
In this section, we assume admissible Morse function $f$ satisfies $\partial_r f>0$ on $\partial W \times [1-\epsilon,1]$. Let $h$ be a Morse function on $\partial W\times \{1-\epsilon\}$. We pick Riemannian metrics $g, g_\partial$ on $W$ and $\partial W\times \{1-\epsilon\}$ respectively. Then we can define the following moduli spaces 
$$R_{p,q}:= \left\{\begin{array}{l} \gamma_1:(-\infty,0]\to \partial W\times \{1-\epsilon\}\\
\gamma_2:[0,\infty)\to W  \end{array}  \left| \begin{array}{l}
\frac{\rd}{\rd s}\gamma_1 + \nabla_{g_{\partial}} h = 0, \frac{\rd}{\rd s}\gamma_2 + \nabla_{g} f = 0, \\
\displaystyle\lim_{s\to -\infty} \gamma_1 = p,  \gamma_1(0)=\gamma_2(0), \lim_{s\to \infty} \gamma_2 = q \end{array}\right. \right\},$$
where  $p \in \cC(h),q\in \cC(f)$. By adding broken flow lines in $W$ and on $\partial W\times \{1-\epsilon\}$, we have that $R_{p,q}$ admits a compactification $\cR_{p,q}$. Then we have the following.
\begin{proposition}[{\cite[Proposition 2.2.8]{kronheimer2007monopoles}}]\label{prop:restriction}
	Given $f,h$ as above, there exist generic metrics $g,g_\partial$, such that the unstable manifolds of $\nabla_gf$ intersect stable manifolds of $\nabla_{g_{\partial}}h$ transversely in $W$. Then $\cR_{p,q}$ is a manifold with boundary of dimension $|p|-|q|$ if $|p|-|q|\le 1$. When $\dim \cR_{p,q} = 1$, $$\partial \cR_{p,q} = -\sum_{*\in \cC(h)} \cM_{p,*}\times \cR_{*,q} +\sum_{*\in \cC(f)} \cR_{p,*}\times \cM_{*,q}.\footnote{That is the boundary orientation on $\cM_{p,*}\times \cR_{*,q}$ is the opposite of the product orientation and the boundary orientation on $\cR_{p,*}\times \cM_{*,q}$ is the product orientation. }$$ Then counting $\cR_{p,q}$ defines a cochain map $R:C(f)\to C(h)$ between Morse cochain complexes. And on cohomology, it is the restriction map $H^*(W)\to H^*(\partial W\times \{1-\epsilon\})=H^*(\partial W)$. 
\end{proposition}
In the following we fix $(f,g,h,g_{\partial})$ such that Proposition \ref{prop:restriction} holds. Therefore on the complex level, the composition of $SH^*_+(W) \to H^{*+1}(W) \to H^{*+1}(\partial W)$ is given by the composition of $C_+(\bH,J) \to C_0(f)[1] \to C(h)[1]$, with the map  defined by counting the moduli space $\cR_{p,q}\times \cM_{q,y}$ for $y\in \cP^*(\bH),p\in \cC(h),q\in \cC(f)$, which is the $l=\infty$ part of Figure \ref{fig:pic}. Next we show that the middle $C_0(f)[1]$ can be bypassed. First, we define the following moduli space 
$$P_{p,y}:=\left\{\begin{array}{l} u:\C \to \widehat{W} \\ \gamma:(-\infty,0] \to \partial W\times \{1-\epsilon\}
\end{array}\left| \begin{array}{l}\partial_s u + J(\partial_t u -X_{\bH})=0, \frac{\rd}{\rd s} \gamma + \nabla_{g_\partial} h=0,\\
	\displaystyle\lim_{s\to -\infty}\gamma = p, u(0)=\gamma(0), \lim_{s\to \infty} u = y
\end{array}\right.\right\}/\R, \quad p \in \cC(h), y \in \cP^*(\bH),
$$
where the $\R$-dilation acts on the $u$ part. The equation makes sense, since on $\partial W\times \{1-\epsilon\}$ we have $\bH=0$. Then we have a compactification $\cP_{p,y}$ and when transversality holds, it defines a map from $C_+(\bH,J) \to C(h)[1]$, which is the $l=0$ part of Figure \ref{fig:pic}. In the following, we will show that it is homotopic to the composition $C_+(\bH,J) \to C_0(f)[1] \to C(h)[1]$. To such purpose, we define the the following moduli space for $p \in \cC(h), y \in \cP^*(\bH)$ involving finite time flow lines of $\nabla_gf$,
$$H_{p,y}:=\left\{\begin{array}{l}
u:\C \to \widehat{W},\\
l>0,\\
\gamma_1:(-\infty,0] \to \partial W\times \{1-\epsilon\},\\
\gamma_2:[0,l]\to W
\end{array} \left|\begin{array}{l}\partial_s u + J(\partial_t u -X_\bH)=0, \\
\frac{\rd}{\rd s} \gamma_1 + \nabla_{g_{\partial}} h=0,
\\ \frac{\rd}{\rd s} \gamma_2 + \nabla_{g} f=0,\\
\displaystyle \lim_{s\to -\infty}\gamma_1 =p, \gamma_1(0) = \gamma_2(0),\\ \displaystyle
 u(0)=\gamma_2(l), \lim_{s\to \infty} u = y
\end{array}\right.\right\}/\R,
$$
where the $\R$-dilation acts on the $u$ part. In addition to configurations of breaking at an orbit in $\cP^*(\bH)$ or a point in $\cC(f)$, the compactification of $H_{p,y}$ also contains $\cR_{p,q}\times \cM_{q,y}$ corresponding to $l=\infty$, i.e. a Morse breaking of the middle gradient flow line in $W$, and $\cP_{p,y}$ corresponding to $l=0$ on the nose. In particular, we have the following. 
\begin{proposition}\label{prop:boundary1}
	$H_{p,y}$ has a compactification $$\cH_{p,y}:= H_{p,y}\cup_{q\in \cC(f)} \cR_{p,q} \times \cM_{q,y}\cup \cP_{p,y} \cup_{q,\in \cC(f), x\in \cP^*(\bH)} \cM_{p,q}\times H_{q,x}\times \cM_{x,y}.$$
	For a generic choice of $J$, any $p,y$ such that $|p|-|y|\le 2$, we have $\cH_{p,y}$ is a manifold with boundary of dimension $|p|-|y|-1$. And $$\partial \cH_{p,y}= \sum_{q\in \cC(f)} \cR_{p,q} \times \cM_{q,y}- \cP_{p,y}+\sum_{x\in \cP^*(H)} \cH_{p,x}\times \cM_{x,y} +\sum_{q\in \cC(h)} \cM_{p,q}\times \cH_{q,y}.$$  
\end{proposition}

To keep track of the regularity of almost complex structures, we introduce the following notations.
\begin{enumerate}
	\item $\cJ^{\le 1}_{reg,P}(\bH, h,g_{\partial})$ is the set of regular admissible $J$ for moduli spaces $\cP_{p,y}$ of dimension up to $1$. And  $\cJ^{D_i}_{reg,P}(\bH, h,g_{\partial})$ is the set of regular admissible $J$ for moduli spaces $\cP_{p,y}$ of dimension up to $0$ and action down to $-D_i$.  
	\item $\cJ^{\le 1}_{reg,H}(\bH,f,g,h,g_{\partial})$ is the set of regular admissible $J$ for moduli spaces $\cH_{p,y}$ of dimension up to $1$.  
\end{enumerate}
Then all of them are of second category, and $\cJ^D_{reg}$ are open. By looking at the potential boundary configurations, we have various relations among $\cJ_{reg}$, e.g.  $\cJ^{\le 1}_{reg,P}(\bH,h,g_{\partial})\subset \cJ^{D_i}_{reg,+}(\bH)$, $\cJ^{\le 1}_{reg,H}(\bH,f,g,h,g_{\partial})\subset \cJ^{D_i}_{reg}(\bH,f,g)\cap \cJ^{D_i}_{reg,P}(\bH,h,g_{\partial})$. 

An instant corollary of Proposition \ref{prop:boundary1} is that if $J\in \cJ^{\le 1}_{reg}(\bH,f,g)\cap \cJ^{\le 1}_{reg,P}(\bH,f,g)\cap  \cJ^{\le 1}_{reg,H}(\bH,f,g,h,g_{\partial})$, then the composition of $C_+(\bH,J)\to C_0(f)[1]\to C(h)[1]$ is homotopic to $P:C_+(\bH,J)\to C(h)[1]$ defined by counting $\cP_{p,y}$. The following proposition is in the same spirit of Proposition \ref{prop:lowreg2}. Since $P$ is defined on $C^{D_i}_+$ if $J\in \cJ^{D_i}_{reg,P}(\bH, h,g_{\partial})$ and $d_{+,0}$ is defined if $J\in \cJ^{D_i}_{reg,+}(\bH)$, the following proposition shows that they are the same up to homotopy for such $J$ of low regularity.
\begin{proposition}\label{prop:lowreg3}
	Let $J\in \cJ^{D_i}_{reg,+}(\bH)\cap \cJ^{D_i}_{reg,P}(\bH,h,g_{\partial})$, then $P$ is defined. In this case, $R\circ d_{+,0}$ is homotopic to $P$ on $C^{D_i}_+(\bH,J,f)$. $P$ is compatible with continuation maps on $C_+^{D_i}$ up to homotopy, i.e. the following is commutative up to homotopy,
	$$
	\xymatrix{
		C_+(\bH,J_1) \ar[r]^P\ar[d]^{\Phi_+} & C(h)[1]\ar[d]^{=}\\
		C_+(\bH,J_2) \ar[r]^P & C(h)[1]}
	$$
\end{proposition}
\begin{proof}
	There exists an open neighborhood $\cU\subset \cJ(W)$ of $J$ contained in $\cJ^{D_i}_{reg,+}(\bH)\cap \cJ^{D_i}_{reg,P}(\bH,h,g_{\partial})$. Then using $J'\in \cU \cap \cJ^{\le 1}_{reg}(\bH,f,g)\cap \cJ^{\le 1}_{reg,P}(\bH,f,g)\cap  \cJ^{\le 1}_{reg,H}(\bH,f,g,h,g_{\partial})$, we have that $P$ defines a cochain map homotopic to $R\circ d_{+,0,J'}$. By Proposition \ref{prop:lowreg2}, $d_{+,0,J'}$ is well-defined up to homotopy for different $J'$. Finally, similar to Proposition \ref{prop:lowreg}, $d_+,P$ are locally constant. The compatibility with continuation map is standard, where the homotopy is defined by considering moduli spaces similar to $\cP_{*,*}$ but with an $s$-dependent almost complex structure and without quotienting the $\R$-action (which does not exist). 
\end{proof}

All above discussions about naturality lead to the following proposition that will be used in the proof of Theorem \ref{thm:A}.
\begin{proposition}\label{prop:compute}
	Let $J^i\in \cJ^{D_i}_{reg,+}(\bH)\cap \cJ^{D_i}_{reg,P}(\bH,h,g_{\partial})$, then the direct limit of the following computes $SH^*_+(W)\to H^{*+1}(\partial W)$:
	$$\left\{H^*(C^{D_1}_+(\bH,J_1)) \to H^*(C^{D_2}_+(\bH,J_2))\to \ldots \right\} \to H^{*+1}(\partial W),$$
	where horizontal arrows in the bracket are  continuation maps and  maps $H^*(C_+^{D_i}(\bH,J_i))\to H^{*+1}(\partial W)$ are defined by $P$.
\end{proposition}
\begin{proof}
	Pick a regular enough $J$, such that the direct limit of $\left\{H^*(C^{D_1}_+(\bH,J)) \to H^*(C^{D_2}_+(\bH,J))\to \ldots \right\} \to H^{*+1}(\partial W)$ computes $SH^*_+(W)\to H^{*+1}(\partial W)$. By Proposition \ref{prop:lowreg3} and functoriality of continuation maps, this diagram is isomorphic to the sequence in the claim by continuation maps.
\end{proof}

\subsection{Independence}\label{sub:ind}
Asymptotically dynamically convex (ADC) contact manifolds was introduced in \cite{lazarev2016contact}. Let $(Y,\xi)$ be a $2n-1$ dimensional contact manifold with a contact form $\alpha$, then we define
$$\cP^{<D}(\alpha):=\left\{\gamma\left| \gamma \text{ is a contractible Reeb orbit and } \int_{\gamma}\alpha<D \right.\right\}$$
If $c_1(\xi) = 0$, then for any contractible non-degenerate Reeb orbit $x$, there is an associated Conley-Zehnder index $\mu_{CZ}(x) \in \Z$. The degree of $x$ is defined to be $\deg(x):=\mu_{CZ}(x)+n-3$.
\begin{definition}\label{def:ADC}
	Let $(Y,\xi)$ be a contact manifold. $Y$ is called $k$-ADC iff there exists a sequence of contact form $\alpha_1>\ldots > \alpha_i >\ldots$ and real numbers $D_1 < \ldots < D_i < \ldots \to \infty$, such that all (contractible) Reeb orbits in $\cP^{<D_i}(\alpha_i)$ are non-degenerate and have degree greater than $k$. $Y$ is called strongly $k$-ADC if in addition, all Reeb orbits of $\alpha_i$ with period smaller than $D_i$ are contractible. We will abbreviate (strongly) $0$-ADC manifolds to (strongly) ADC manifolds.
\end{definition}
\begin{remark}
	In general, when $c_1(\xi)=0$, choosing a trivialization of $\det_{\C}\xi$ will assign a Conley-Zehnder index to every Reeb orbit. The Conley-Zehnder index of $\gamma$ is independent of trivializations if $\gamma$ is annihilated by $H^1(Y;\Z)=[Y,S^1]$ (which is the space of trivializations up to homotopy), i.e. $\beta(\gamma)=0$ for any $\beta \in H^1(Y;\Z)$ in the pairing $H^1(Y;\Z)\otimes H_1(Y;\Z) \to \Z$. The reason of considering topologically simple filling $W$ is to make sure a Reeb orbit $\gamma$ contractible in $W$ is assigned with a well-defined Conley-Zehnder index using only the boundary. From this point of view, one can consider a slight generalization of ADC manifolds and their topologically simple fillings, i.e. $\cP^{<D}$ now stands for orbits  annihilated by $H^1(Y;\Z)$ with period $<D$ and topologically simple filling now requires $H^1(W;\Z)\to H^1(Y;\Z)$ is injective and $c_1(W)=0$. Most results in the paper hold in such setting too.
\end{remark}

\begin{example}\label{ex:ADC}
	Contact manifolds admitting flexible Weinstein fillings with vanishing first Chern class are ADC by the work of Lazarev \cite{lazarev2016contact} and are strongly ADC for subcritically fillable contact manifolds. Moreover, Lazarev showed that ADC is preserved under flexible surgery\footnote{Some extra conditions need to be satisfied when attaching a $2$-handle, c.f. \cite[Theorem 3.17]{lazarev2016contact}.}. Cotangent bundle $T^*M$ is ADC whenever $\dim M \ge 4$. Quantization bundles over a positive monotone symplectic manifold are ADC, whenever the degree of the bundle is not too big. 
\end{example}

In the following, we introduce an analogous definition for filling, which will be used in the construction of ADC manifolds in \S \ref{s6}. Assume contact manifold $(Y,\xi)$ has a symplectic filling $W$ with $c_1(W)=0$. Let $x$ be a non-degenerate Reeb orbit, if $x$ is contractible in $W$, then a canonical Conley-Zehnder index can be assigned.  
\begin{definition}\label{def:ADCfilling}
	Let $(W,\lambda)$ be a Liouville domain with $c_1(W) = 0$. $W$ is called $k$-ADC if there exist positive functions on $\partial W$
	$f_1>\ldots > f_i >\ldots$ and real numbers $D_1 < \ldots < D_i < \ldots \to \infty$, such that all contractible (in $W$) Reeb orbits of $(\partial W, f_i\lambda)$ with period smaller than $D_i$ are non-degenerate and have degree greater than $k$. $W$ is called strongly $k$-ADC if in addition all Reeb orbits of $f_i\lambda_i$ with period smaller than $D_i$ are contractible in $W$. In particular if $W$ is ADC, then $\partial W$ is also ADC. 
\end{definition}
\begin{example}
	In \S \ref{s6}, we show that $V\times \C$ is always ADC for any  Liouville domain $V$ with $c_1(V)=0$ and $\dim V>0$. $V\times W$ is always ADC, given $V,W$ are both ADC Liouville domains of dimension $\ge 4$. They provide more examples of ADC contact manifolds.
\end{example}
\begin{example}
	It is possible that $\partial W$ is ADC but $W$ is not ADC, the source of this discrepancy is that we can have Reeb orbits that are non-contractible in the boundary but are contractible in the filling and have low SFT degree. For example, $S^1$ is ADC, but the $\D$ disk is not ADC.  
\end{example}

\begin{remark}\label{rmk:ADC}
	Dynamical convexity was introduced in \cite{hofer1998dynamics} on $(S^3,\xi_{std})$ as a substitute of the geometric convexity. For $(S^{2n-1},\xi_{std})$, a contact form is dynamical convex if the minimal Conley-Zehnder index is $n+1$, as it is the case for convex hypersurfaces in $\R^{2n}$. Note that this is the lowest degree that is nontrivial in the cylindrical contact homology of $(S^{2n-1},\xi_{std})$. Following this idea,  Abreu-Macarini \cite{abreu2017dynamical} defined dynamical convex for a larger class of contact manifolds, as a property of contact forms. Although with similar names, Lazarev's asymptotically dynamically convexity has a very different motivation, which is a generalization of index-positive in \cite{cieliebak2018symplectic} and is related to the existence of nice contact forms introduced in \cite{eliashberg2000introduction}. It is clear that ADC is equivalent to $\sup_{\substack{\alpha_1>\alpha_2>\ldots,\\D_1<D_2<\ldots \to \infty }}(\inf_{x\in \cP^{<D_i}(\alpha_i),i\in \N}\deg(x))>0$. Such number is an invariant of the contact topology. A similar number was defined by McLean \cite{mclean2016reeb} and shown to be equal to twice the minimal discrepancy for a large class of isolated singularities, when it is nonnegative.  In particular, the link is ADC if the singularity is terminal.
\end{remark}

From the first glance of $\cP_{p,y}$, $\cP_{p,y}$ has some chance to be independent of the filling since both $p,y$ only depend on the contact boundary. However, the curve $u$ in $\cP_{p,y}$ may rely on the filling. If we have the ADC property, a neck-stretching argument implies that $\cP_{p,y}$ actually does not see the interior of $W$. Neck-stretching argument was used to show independence of $SH_+^*(W)$ in \cite{cieliebak2018symplectic} for index-positive convex manifolds. It is easy to show $H^*(C^{D_i}_+(\bH,J))\to H^{*+1}(\partial W)$ is independent of filling for any $D_i$ by neck-stretching. But we also need some naturality of the independence. In the case of independence of $SH^*_+(W)$, naturality was discussed carefully in \cite[Proposition 3.8]{lazarev2016contact} for ADC manifolds. In the following, we give a simplified treatment. Since in our case $\bH$ will not change and is already constant on $W$, we can bypass the Viterbo transfer map in the proof of \cite[Proposition 3.8]{lazarev2016contact}.

Let $(Y,\alpha)$ be an ADC contact manifold with two topologically simple fillings $W_1,W_2$ with fixed Hamiltonians $\bH_1=\bH_2=\bH$ outside $W_1,W_2$ as in \S \ref{s2}. Note that $W_1, W_2$ both contain the negative end of symplectization $(Y\times (0,1), \rd(r\alpha))$. Since $Y$ is ADC, for every $i\in \N_+$, there exist contact type surfaces $Y_i\subset Y\times (0,1-\epsilon)\subset W_1,W_2$, such that $Y_i$ lies outside of $Y_{i+1}$ and contractible Reeb orbits of contact form $r\alpha|_{Y_i}$ has the property that the degree of an orbit is greater than $0$ if the period of the orbit is smaller than $D_i$. 
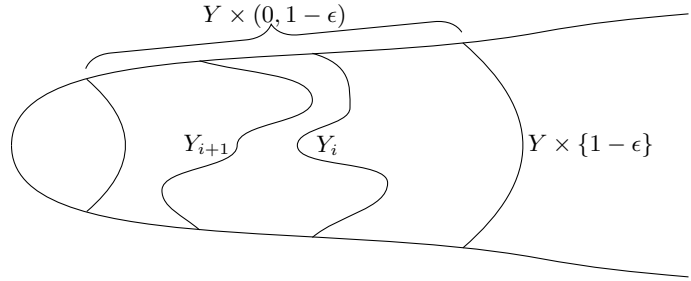
\begin{figure}[H]
	\begin{center}
	\begin{tikzpicture}[yscale=0.5]
	\draw (0,0) to [out=90, in =200] (7,3);
	\draw (7,3) to [out=20, in=190] (9,3.5);
	\draw (0,0) to [out=270, in =160] (7,-3);
	\draw (7,-3) to [out=340, in=170] (9,-3.5);
	\draw (6,2.73) to [out=300, in=60](6,-2.73);
	\draw (1,1.76) to [out=300, in=60](1,-1.76);
	\draw (4,2.45) to [out=330, in=90] (4.5,1) to [out=270, in=90] (3.8,0) to  [out=270, in=90](5,-1) to [out=270, in=60](4,-2.45);
	\draw (2.5,2.25) to [out=330, in=90](4,1.2) to [out=270, in=90](3,0) to [out=270, in=90](2,-1.2) to [out=270, in=120] (2.5,-2.25);
	\node at (4.2,0) {\footnotesize$Y_i$};
	\node at (2.6,0) {\footnotesize$Y_{i+1}$};
	\node at (7.7,0) {\footnotesize$Y\times \{1-\epsilon\}$};
	\draw [decorate,decoration={brace,amplitude=10pt,raise=4pt},yshift=0pt]
	(1,1.76) -- (6,2.73) node [black,midway,yshift=0.6cm] {\footnotesize$Y\times (0,1-\epsilon)$};
	\end{tikzpicture}
	\end{center}
    \caption{$Y_i\subset \widehat{W}_*$}
\end{figure}

Neck-stretching near $Y_i$ is given by the following. Assume $Y_i\times [1-\epsilon_i,1+\epsilon_i]_{r_i}$ does not intersect each other for some small $\epsilon_i$. Assume $J|_{Y_i\times [1-\epsilon_i,1+\epsilon_i]_{r_i}}=J_0$, where $J_0$ is independent of $S^1$ and $r_i$ and $J_0(r_i\partial_{r_i})=R_i,J_0\xi_i=\xi_i$, where $\xi_i=\ker r\alpha|_{Y_i}$. Then we pick a family of diffeomorphism $\phi_R:[(1-\epsilon_i)e^{1-\frac{1}{R}}, (1+\epsilon_i)e^{\frac{1}{R}-1}]\to [1-\epsilon_i,1+\epsilon_i]$ for $R\in (0,1]$ such that $\phi_1=\Id$ and $\phi_R$ near the boundary is linear with slope $1$. Then the stretched almost complex structure $NS_{i,R}(J)$ is defined to be $J$ outside $Y_i\times [1-\epsilon_i,1+\epsilon_i]$ and is $(\phi_R\times \Id)_*J_0$ on $Y_i\times [1-\epsilon_i,1+\epsilon_i]$. Then $NS_{i,1}(J)=J$ and $NS_{i,0}(J)$ gives almost complex structures on the completions of the cobordism between $Y$ and $Y_i$, the filling of $Y_i$, and  the symplectization $Y_i\times \R_+$. Since we need to stretch along different contact surfaces, we assume the $NS_{i,R}(J)$ have the property that $NS_{i,R}(J)$ will modify the almost complex structure near $Y_{i+1}$ to a cylindrical almost complex structure for $R$ from $1$ to $\frac{1}{2}$ and for $R\le \frac{1}{2}$, we only keep stretching along $Y_i$. We use $\cJ^{D_i}_{reg,SFT}(\bH,h,g_{\partial})$ to denote the set of admissible regular $J$, i.e. almost complex structures satisfying Definition \ref{def:convex} on the completion of the cobordism between $Y$ and $Y_i$ and asymptotic (in a prescribed way as in the stretching process) to $J_0$ on the negative cylindrical end, such that the following two moduli spaces in the picture up to dimension $0$ with action of the positive end $\ge -D_i$ are cut out transversely. It is an open dense set.
\begin{figure}[H]
\begin{center}
\begin{tikzpicture}
\draw[rounded corners](0,0) -- (0,4) -- (2,4) -- (2,6);
\draw[rounded corners](3,6)--(3,5)--(5,5)--(5,3);
\draw[rounded corners](4,3)--(4,4)--(3,4)--(3,0);
\draw[rounded corners](2,0)--(2,3)--(1,3)--(1,0);

\draw (0,0) to [out=270,in=180] (0.5,-0.25);
\draw (0.5,-0.25) to [out=0,in=270] (1,0);
\draw[dotted] (0,0) to [out=90,in=180] (0.5,0.25);
\draw[dotted] (0.5,0.25) to [out=0,in=90] (1,0);

\draw (2,0) to [out=270,in=180] (2.5,-0.25);
\draw (2.5,-0.25) to [out=0,in=270] (3,0);
\draw[dotted] (2,0) to [out=90,in=180] (2.5,0.25);
\draw[dotted] (2.5,0.25) to [out=0,in=90] (3,0);

\draw (4,3) to [out=270,in=180] (4.5,2.75);
\draw (4.5,2.75) to [out=0,in=270] (5,3);
\draw[dotted] (4,3) to [out=90,in=180] (4.5,3.25);
\draw[dotted] (4.5,3.25) to [out=0,in=90] (5,3);

\draw (2,6) to [out=270,in=180] (2.5,5.75);
\draw (2.5,5.75) to [out=0,in=270] (3,6);
\draw (2,6) to [out=90,in=180] (2.5,6.25);
\draw (2.5,6.25) to [out=0,in=90] (3,6);

\node at (2.5,6) {$y$};
\node at (4.5,3) {$x$};
\node at (0.5,0) {$\gamma_1$};
\node at (2.5,0) {$\gamma_2$};

\draw[rounded corners](8,0) -- (8,4) -- (10,4) -- (10,6);
\draw[rounded corners](11,6)--(11,5)--(13,5)--(13,3);
\draw[rounded corners](12,3)--(12,4)--(11,4)--(11,0);
\draw[rounded corners](10,0)--(10,3)--(9,3)--(9,0);

\draw (8,0) to [out=270,in=180] (8.5,-0.25);
\draw (8.5,-0.25) to [out=0,in=270] (9,0);
\draw[dotted] (8,0) to [out=90,in=180] (8.5,0.25);
\draw[dotted] (8.5,0.25) to [out=0,in=90] (9,0);

\draw (10,0) to [out=270,in=180] (10.5,-0.25);
\draw (10.5,-0.25) to [out=0,in=270] (11,0);
\draw[dotted] (10,0) to [out=90,in=180] (10.5,0.25);
\draw[dotted] (10.5,0.25) to [out=0,in=90] (11,0);

\draw (10,6) to [out=270,in=180] (10.5,5.75);
\draw (10.5,5.75) to [out=0,in=270] (11,6);
\draw (10,6) to [out=90,in=180] (10.5,6.25);
\draw (10.5,6.25) to [out=0,in=90] (11,6);

\draw[dotted] (12,3) to [out=270,in=180] (12.5,2.75);
\draw[dotted] (12.5,2.75) to [out=0,in=270] (13,3);
\draw[dotted] (12,3) to [out=90,in=180] (12.5,3.25);
\draw[dotted] (12.5,3.25) to [out=0,in=90] (13,3);

\draw (12,3) to [out=270,in=180] (12.5,2);
\draw (12.5,2) to [out=0,in=270] (13,3);

\draw[->] (12.5,2) -- (13.25,2);
\node at (13.4,2.25) {$\nabla_{g_{\partial}}h$};
\draw (13.25,2) -- (14,2);
\draw (14,2) circle[radius=1pt];
\fill (14,2) circle[radius=1pt];

\node at (10.5,6) {$y$};
\node at (14.2,2.2) {$p$};
\node at (8.5,0) {$\gamma_1$};
\node at (10.5,0) {$\gamma_2$};

\draw[dashed] (-0.5,0)--(0,0);
\draw[dashed] (1,0)--(2,0);
\draw[dashed] (3,0)--(8,0);
\draw[dashed] (9,0)--(10,0);
\draw[dashed] (11,0)--(14,0);
\node at (6,0.2) {$Y_i\times \{0\}$};

\draw[dashed] (11.5,2) -- (15,2);
\node at (13,1.6) {$\partial W\times \{1-\epsilon\}$};
\end{tikzpicture}
\end{center}
\caption{Moduli spaces for the definition of  $\cJ^{D_i}_{reg,SFT}(\bH,h,g_{\partial})$.}\label{fig:SFT}
\end{figure}
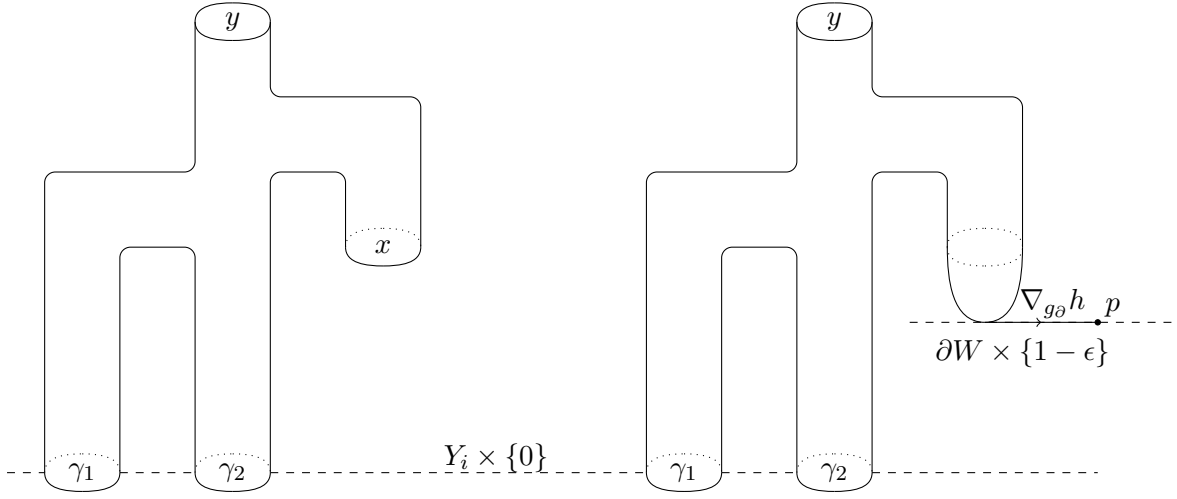

\begin{proposition}\label{prop:ind}
With setups above, there exist admissible $J^1_1,J^2_1,\ldots,J^1_2,J^2_2,\ldots$ on $\widehat{W_1}$ and $\widehat{W_2}$ respectively and positive real numbers $\epsilon_1, \epsilon_2,\ldots <\frac{1}{2}$ such that the following holds.
\begin{enumerate}
	\item\label{ns1} For $R<\epsilon_i$ and any $R'$, $NS_{i,R}(J^i_*), NS_{i+1,R'}(NS_{i,R}(J^i_*))\in \cJ^{D_i}_{reg,+}(\bH)\cap \cJ^{D_i}_{reg,P}(\bH,h,g_{\partial})$. Moreover, for such almost complex structures,  all zero dimensional moduli spaces $\cM_{x,y}$ and $\cP_{p,y}$ are the same for both $W_1,W_2$ and contained outside $Y_i$ for $x,y\in \cP^*(\bH)$ with action $\ge - D_i$, $q\in C(h)$.   
	\item\label{ns2} $J^{i+1}_*=NS_{i,\frac{\epsilon_i}{2}}(J^i_*)$ on $W^i_*$. 
\end{enumerate}
\end{proposition}
\begin{proof}
	We prove the proposition by induction. We first choose a $J^1$ such that $NS_{1,0}(J^1)\in \cJ^{D_1}_{reg,SFT}(\bH,h,g_{\partial})$. Assume $\cM_{x,y}$ is not contained outside $Y_1$ in the stretching process. Then a limit curve $u$ outside $Y_1$ has one component by \cite[Lemma 2.4]{cieliebak2018symplectic}\footnote{Note that our symplectic action has the opposite sign compared to \cite[Proposition 9.17]{cieliebak2018symplectic}.}. Moreover, by the argument in \cite[Lemma 2.4]{cieliebak2018symplectic}, $u$ can only asymptotic to Reeb orbits $\gamma_j$ that are contractible in $W_*$ on $Y_1$ with period smaller than $D_1$. Since $W_*$ is topological simple, $\gamma_j$ is contractible in $Y_1$.
	Note that $\ind(u)=|y|-|x|-\sum(\mu_{CZ}(\gamma_j)+n-3) < 1$ as $\mu_{CZ}(\gamma_j)+n-3>0$. Since we need to quotient the $\R$-action, the expected dimension of the moduli space of such $u$ is negative. By our regularity assumption on $NS_{1,0}(J^1)$, $u$ is cut out transversely, which contradicts its existence. Hence for $R$ close to $0$, using $NS_{1,R}(J^1)$, we have $\cM_{x,y}$ lives outside $Y_1$. Since $NS_{1,0}(J^1)\in \cJ^{D_1}_{reg,SFT}(\bH,h,g_{\partial})$ and every curve in $\cM_{x,y}$ lives outside $Y_1$, we have $NS_{1,R}(J^1)\in \cJ^{D_1}_{reg,+}(\bH)$ for $R$ small. Similarly, we have the same property for $\cP_{x,q}$ and $NS_{1,R}(J^1)\in \cJ^{D_1}_{reg,H}(\bH,h,g_{\partial})$ for $R$ small. The argument can be applied to stretching on both $Y_1,Y_2$ and a compactness argument shows that for every $R'\in [0,1]$, there exists $\epsilon_{R'}>0$ and $\delta_{R'}>0$ such that the same regular and outside property hold for $NS_{2,\delta}(NS_{1,\epsilon})(J^1)$ for $\epsilon < \epsilon_{R'}$ and $|\delta-R'|<\delta_{R'}$. Therefore Claim \eqref{ns1} holds for some $\epsilon_1$. Since $NS_{2,0}(NS_{1,0}(J^1))$ has the property that all curves in Figure \ref{fig:SFT} with $x,y\in C_+^{D_1}$ must be contained outside $Y_1$, i.e. curves in $\cM_{x,y}, \cP_{p,y}$ (which are viewed as limit curves after stretching). Then we may assume $\epsilon_1$ is small enough such that $NS_{2,0}(NS_{1,\frac{\epsilon_1}{2}}(J^1))\in \cJ^{D_1}_{reg,SFT}(\bH,h,g_{\partial})$. Therefore we can perturb $NS_{1,\frac{\epsilon_1}{2}}(J^1)$ outside $W^1_*$ near orbits in $W^2_*\backslash W^1_*$ to $J^{2}$ such that $NS_{2,0}(J^2)\in \cJ^{D_2}_{reg,SFT}(\bH,h,g_{\partial})$, this will not influence the previous regular property for periodic orbits with action down to $-D_1$ by the integrated maximum principle (Lemma \ref{lemma:max}). Then we can keep the induction going. It is clear the construction can be made on both $W_1,W_2$ yielding the same $\cM_{x,y}$ and $\cP_{p,y}$.  
\end{proof}
\begin{remark}\label{rmk:regular}
	From the proof above, it is clear that we can not guarantee $NS_{i,R}(J^i)$ in $\cJ^{\le1}_{reg,+}(\bH)$ for all $R$ small unless we assume $1$-ADC. Moreover, there is no guarantee for $NS_{i,R}(J^i)$ in $\cJ_{reg}(\bH,f,g)$, that is why we need Proposition \ref{prop:lowreg2} and Proposition \ref{prop:lowreg3}.
\end{remark}

\begin{proof}[Proof of Theorem \ref{thm:A}]
Using the almost complex structures from Proposition \ref{prop:ind}, we have the following sequence for both fillings
\begin{equation}\label{eqn:seq}
\left\{ C^{D_1}_+(NS_{1,\frac{\epsilon_1}{2}}J^1)\to C^{D_2}_+(NS_{2,\frac{\epsilon_2}{2}}J^2) \to  C^{D_3}_+(NS_{3,\frac{\epsilon_3}{2}}J^3) \ldots    \right\} \to C(h)[1],
\end{equation}
where each complex  and the map $P$ to $C(h)[1]$ are identified with each other for both fillings. Therefore it suffices to show the continuation map $C^{D_i}_+(NS_{{i},\frac{\epsilon_{i}}{2}}J^i) \to  C^{D_{i+1}}_+(NS_{{i+1},\frac{\epsilon_{i+1}}{2}}J^{i+1})$
is naturally identified. The continuation map is decomposed into continuation maps $\Phi:C^{D_i}_+(NS_{{i},\frac{\epsilon_{i}}{2}}J^i)\to C^{D_i}_+(NS_{i+1,\frac{\epsilon_{i+1}}{2}}(NS_{{i},\frac{\epsilon_{i}}{2}}J^i))$ and $\Psi:C^{D_i}_+(NS_{i+1,\frac{\epsilon_{i+1}}{2}}(NS_{{i},\frac{\epsilon_{i}}{2}}J^i)) \to C_+^{D_{i+1}}(NS_{i+1,\frac{\epsilon_{i+1}}{2}}J^{i+1})$. Then $\Phi$ is identity by Lemma \ref{lemma:local} using homotopy $NS_{i+1,s}(NS_{i,\frac{\epsilon_i}{2}})(J^i)$  for $s\in [\frac{\epsilon_{i+1}}{2}, 1]$. Since $J^{i+1}$ is the same as $NS_{i,\frac{\epsilon_i}{2}}(J^i)$ inside $W^i$, then the integrated maximum principle implies that $\Psi$ is composition $C^{D_i}_+(NS_{{i+1},\frac{\epsilon_{i+1}}{2}}(NS_{i,\frac{\epsilon_i}{2}}J^i)) \stackrel{\Id}{\to} C^{D_i}_+(NS_{i+1,\frac{\epsilon_{i+1}}{2}}J^{i+1})\hookrightarrow C^{D_{i+1}}_+(NS_{{i+1},\frac{\epsilon_{i+1}}{2}}J^{i+1})$, which is the inclusion, hence it is the same for both $W_1,W_2$. Therefore continuation maps in \eqref{eqn:seq} are inclusions, hence the whole diagram can be identified, and Proposition \ref{prop:compute} implies the theorem. 
\end{proof}
\begin{proof}[Proof of Corollary \ref{cor:B}]
	If $SH^*(W)=0$, then $1\in \Ima \delta_{\partial}$. Then by Theorem \ref{thm:A}, for any other topologically simple exact filling $W'$, we have $1\in \Ima \delta_{\partial}$. Then $SH^*(W')=0$ and Theorem \ref{thm:A} implies the invariance of $H^*(W;\Z)\to H^*(Y;\Z)$. 
\end{proof}

Using Theorem \ref{thm:A}, we derive the following obstruction to Weinstein fillability. 
\begin{corollary}\label{cor:ob}
	Let $Y$ be a $2n-1$ dimensional ADC contact manifold and $n\ge 3$. If $Y$ admits a topologically simple exact filling $W$ such that a non-trivial element of grading greater than $n$ is in the image of $SH_+^*(W) \to H^{*+1}(Y)$, then $Y$ does not admit Weinstein fillings.
\end{corollary}
\begin{proof}
	Assume otherwise that $Y$ admits a Weinstein filling $W'$. Since $n\ge 3$ and $W'$ is built from $Y$ by attaching index $k\ge n\ge 3$ handles, we have $c_1(W')=0$ and $\pi_1(Y)\to \pi_1(W')$ is isomorphism, i.e. $W'$ is topologically simple. Then $SH_+^*(W')\to H^{*+1}(W')\to H^{*+1}(Y)$ is isomorphic to $SH_+^*(W)\to H^{*+1}(Y)$. Since $H^*(W')$ is supported in degree $\le n$, we arrive at a contradiction. 
\end{proof}

When the obstruction in Corollary \ref{cor:ob} does not vanish for one contact manifold, it is easy to construct infinitely many obstructed examples by the following.
\begin{proposition}\label{prop:many}
	Let $Y$ be a contact manifold, such that conditions in Corollary \ref{cor:ob} hold. For any ADC contact manifold $Y'$ with a topologically simple exact filling $W'$, the contact connected sum $Y\# Y'$ is not Weinstein fillable.
\end{proposition}
\begin{proof}
	By assumption, the image of $SH^{*}_+(W) \to H^{*+1}(Y)$ contains an element $\alpha'$ of grading $k \ge n+1$. That is the image of $SH^*_+(W)\to H^{*+1}(W)$ contains an element $\alpha$, which restricts to $\alpha'$.  That is $\alpha$ is mapped to $0$ in $SH^k(W)$. Note that $k$ can not be $2n-1$, for otherwise it will imply that $H^{2n-1}(W;\R)\to H^{2n-1}(Y;\R)$ is surjective, which contradicts the Stokes' theorem. By \cite{cieliebak2002handle}, if we view $\alpha \in H^k(W\natural W')$, then $\alpha$ is mapped to $0$ in $SH^{k}(W\natural W')$. This implies that $\alpha$ is in the image of $SH^{k-1}_+(W\natural W') \to H^{k}(W\natural W')$, since $n+1\le k < 2n-1$, we have $\alpha$ restricted to the boundary $Y\#Y'$ is represented by $\alpha'$ and non-zero. This implies that $\delta_{\partial}$ contains a nontrivial element of degree $k$ for $W\natural W'$. By \cite{lazarev2016contact}, $Y\natural Y'$ is ADC, and it is direct to check that $W\natural W'$ is topologically simple. Then by Corollary \ref{cor:ob}, $Y\#Y'$ is not Weinstein fillable.
\end{proof}

\subsection{Symplectic cohomology for covering spaces}
Before proving Theorem \ref{thm:C}, we need to introduce symplectic cohomology for covering spaces.  Let $W$ be an exact domain and $\pi:\widetilde{W} \to W$ a covering (not necessarily connected). The idea is lifting all geometric data to the covering space to define a Floer theory. We define following two sets
\begin{eqnarray*}
\widetilde{\cP}^*(\bH) & := & \{(x,a)|x\in \cP^*(\bH),a \in \widehat{\widetilde{W}}, \pi(a)=x(0) \},\\
\widetilde{\cC}(f) & := & \{(p,a)|p\in \cC(f), a \in \widehat{\widetilde{W}}, \pi(a)=p \}.
\end{eqnarray*}
Then we can define $\cM_{(x,a),(y,b)}$ for $(x,a),(y,b) \in \widetilde{\cP}^*(\bH), \widetilde{\cC}(f)$ as follows, if $(x,a),(y,b)\in \widetilde{\cP}^*(\bH)$, 
$$\cM_{(x,a),(y,b)}:=\{ u\in\cM_{x,y}|u(s,0)\text{ lifts to a path from $a,b$ in $\widehat{\widetilde{W}}$}\},$$
other cases are similar. Hence for regular $J$, $\cM_{(x,a),(y,b)}$ is always diffeomorphic to some connected components of $\cM_{x,y}$ with the same orientation. In particular, by Proposition \ref{prop:transversality}, $\cM_{(x,a),(y,b)}$ is always a compact manifold with boundary of dimension $|x|-|y|-1$ when $|x|-|y|\le 2$ and $\partial \cM_{(x,a),(y,b)}=\sum_{(z,c)}\cM_{(x,a),(z,c)}\times \cM_{(z,c),(y,b)}$. Then we can use them to define a cochain complex $\widetilde{C}^*(\bH,J,f)$. However, the generator set is infinite even with an action bound, and since we are trying define the cohomology of the covering space, the cochain complex is the \textit{direct product} in the fiber direction, i.e. $\oplus_x \prod_a\Z\la(x,a) \ra$. The differential on a generator is again defined by  
$$\rd(y,b)=\sum_{(x,a)} \#\cM_{(x,a),(y,b)}(x,a).$$
The compactness of $\cM_{x,y}$ implies that for any $b$, there are at most finitely many $a$ such that $\cM_{(x,a),(y,b)}\ne \emptyset$. Therefore the differential is well-defined on the complex. We can similarly define $\widetilde{C}^*_+(\bH,J), \widetilde{C}^*_0(f)$. By definition, the cohomology $H^*(\widetilde{C}_0(f))$ is the cohomology of the universal cover  $H^*(\widetilde{W})$, since the cochain complex is dual to the Morse homology complex with $\Z[\pi_1(W)]$ local system, which computes the homology of the universal cover \cite[\S 3.H]{hatcher2002algebraic}. Moreover, there is always a cochain morphism $C^*_{(+/0)}(\bH,J,f)\to \widetilde{C}^*_{(+/0)}(\bH,J,f)$ defined by sending $x\to \prod_{a} (x,a)$ corresponding to the pull-back  to the covering.  We define $SH^*(\widetilde{W}), SH^*_+(\widetilde{W})$ to be the cohomology of $\widetilde{C}^*(\bH,J,f),\widetilde{C}^*_+(\bH,J)$. Then we have a long exact sequence,
$$\ldots \to H^{*}(\widetilde{W})\to SH^*(\widetilde{W})\to SH^*_+(\widetilde{W}) \to H^{*+1}(\widetilde{W})\to \ldots.$$
Moreover, the natural maps $SH_{(+/0)}^*(W) \to SH_{(+/0)}^*(\widetilde{W})$ are compatible with the long exact sequence. We can lift everything discussed before to $\widetilde{W}$, hence there is the following analogue of Theorem \ref{thm:A}. 
\begin{proposition}\label{prop:indlocal}
	Under the same assumption as in Theorem \ref{thm:A}, let $\pi:\widetilde{W}\to W$ be a covering, then $\widetilde{Y}:=\partial \widetilde{W}$ is a covering of $Y$. Then $\widetilde{\delta_\partial}:SH^*_+(\widetilde{W}) \to H^{*+1}(\widetilde{Y})$ is independent of the topologically simple exact filling $W'$ and covering $\widetilde{W}'$, as long as $\partial \widetilde{W}'=\widetilde{Y}$. 
\end{proposition}

\begin{theorem}\label{thm:fund}
	Assume $Y$ is an ADC contact manifold. If $W$ is a topologically simple exact filling of $Y$ such that $\pi_1(Y) \to \pi_1(W)$ is isomorphism and $SH^*(W) = 0$. Then for any other topological simple exact filling $W'$, we have $\pi_1(Y) \to \pi_1(W')$ is an isomorphism.
\end{theorem}
\begin{proof}
	Assume $\pi(Y) \to \pi(W')$ is only injective. Then the universal cover $\widetilde{W'}$ restricted to boundary is $\pi_1(W')/\pi_1(Y)$ copies of the universal universal cover $\widetilde{Y}$. Let $\widetilde{W}$ be the universal cover of $W$, since $\pi_1(Y)\to \pi_1(W)$ is an isomorphism, we have $\partial \widetilde{W} = \widetilde{Y}$. Since $SH^{-1}_+(\widetilde{W}) \to H^0(\widetilde{Y}) = \Z$ is surjective due to that $SH^{-1}_+(W) \to H^0(Y)$ is an isomorphism and the following commutative diagram
	$$
	\xymatrix{
		SH^{-1}_+(W) \ar[r] \ar[d] & H^0(W) \ar[r]\ar[d] & H^0(Y)\ar[d] \\
		SH^{-1}_+(\widetilde{W}) \ar[r] & H^0(\widetilde{W}) \ar[r] & H^0(\widetilde{Y}).
		}
	$$
	Since $\pi_1(W')/\pi_1(Y)$ copies of the universal universal cover $\widetilde{W}$ restricted to boundary is also $\pi_1(W')/\pi_1(Y)$ copies of $\widetilde{Y}$. Then by Proposition \ref{prop:indlocal}, we have that  $SH^{-1}_+(\widetilde{W}') \to H^{0}(\partial \widetilde{W}')=\prod_{\pi_1(W')/\pi_1(Y)}\Z$ is surjective.  However $SH^{-1}_+(\widetilde{W}') \to H^0(\partial \widetilde{W}')$ factors through $H^0(\widetilde{W}') = \Z$, which contradicts that the cardinality of $\pi_1(W')/\pi_1(Y)$  is greater than $1$. 
\end{proof}

\begin{proof}[Proof of Theorem \ref{thm:C}]
	It is sufficient to prove the strongly ADC case, as the ADC case is proven in Theorem \ref{thm:fund}. Since $\pi_1(Y) \to \pi_1(W)$ is an isomorphism, we have $H_1(Y) \to H_1(W)$ is an isomorphism. By the universal coefficient theorem. We have $H^1(W) \to H^1(Y)$ is an isomorphism and $H^2(W) \to H^2(Y)$ is an isomorphism on the torsion part. By Corollary \ref{cor:B}, we have the same thing holds on $W'$. Then using the universal coefficient theorem again, we have $H_1(Y) \to H_1(W')$ is an isomorphism. Since $\pi_1(Y)$ is abelian, this implies that $\pi_1(Y) \to \pi_1(W')$ is injective, then by Theorem \ref{thm:fund},  we have $\pi_1(Y) \to \pi_1(W')$ is an isomorphism.
\end{proof}

\subsection{Local systems}
Albers-Frauenfelder-Oancea \cite{albers2017local} found examples of cotangent bundles with vanishing symplectic cohomology after appropriately twisting by local systems. Since cotangent bundles are ADC, when base manifolds have dimension at least $4$, we discuss the analogue of Theorem \ref{thm:A} for local systems and its applications to symplectic topology in this part. The discussion here works for general local systems, however, to get interesting applications, we will only consider the following special local systems. We use $L_0W$ to denote the connected component of contractible loops in the free loop space $LW$ and let $R$ be a commutative ring.

\begin{definition}\label{def:localsystem}
	An admissible local system on a Liouville domain $W$ is a flat $R$ bundle over $L_0W$ and is trivial on the constant loops.
\end{definition}
Let $R^{\times}$ denote the group of units in $R$, then using parallel transportation, a flat $R$ bundle over $L_0W$ is represented by a class in $\Hom(\pi_1(LW),R^\times)/R^\times$, where the $R^{\times}$ action is given by conjugation. We use $\Omega_0W$ to denote the space of contractible based loops in $W$. Note that we have a short exact sequence of groups, 
$$0 \to \pi_1(\Omega_0W) \to \pi_1(L_0W)\to \pi_1(W) \to 1$$
from the fibration $\pi:L_0W \to W$, where $\pi$ is the evaluation map at the starting point. Moreover, we have
$$\pi_1(L_0W)\simeq \pi_1(\Omega_0 W)\rtimes \pi_1(W),$$
where $\pi_1(W)$ acts on $\pi_1(\Omega_0 W)$ by conjugation. As a consequence, we have
$$\Hom(\pi_1(L_0W),R^\times)\simeq \Hom_{inv}(\pi_2(W),R^\times)\times \Hom(\pi_1(W),R^\times),$$
where $\Hom_{inv}(\pi_2(W),R^\times)\subset \Hom(\pi_2(W),R^\times)$ is the subgroup of $\pi_1(W)$-invariant elements. The projection to the $\Hom(\pi_1(W),R^\times)$ corresponds to the restriction on constant loops. Therefore an admissible local system is represented by an element of $\Hom_{inv}(\pi_2(W),R^\times)/R^\times$, in other words, it can be represented by a trivial extension of a $\pi_1(W)$-invariant local system on $\Omega_0W$.  

With an admissible local system $\rho$, we can define symplectic cohomology with local system $\rho$, where the cochain complex is generated by fibers $\rho_x\simeq R $ of the local system over $x\in \cP^*(\bH)$ or constant loop $x\in \cC(f)$. The differential now needs to take the parallel transportation into account, i.e. 
\begin{equation}\label{eqn:local}
da = \sum_{x}\sum_{u\in \cM_{x,y}} \sign(u) \rho_u(a), \quad a\in \rho_y 
\end{equation}
where $\cM_{x,y}$ is a zero dimensional moduli space, and $\rho_u$ a $R$-module map $\rho_y\to\rho_x$, i.e. the parallel transportation determined by $u$ when we view $u$ as a path in $L_0W$ connecting $y$ and $x$. We use $SH^*(W;\rho)$ to denote the cohomology, such theory shares all properties of $SH^*(W)$. In particular, we have the following proposition from the same argument of Theorem \ref{thm:A}.
\begin{proposition}\label{prop:indloc}
Let $Y$ be an ADC manifold, then $\delta_{\partial,\rho}:SH_+^*(W;\rho)\to H^{*+1}(W;R)\to H^{*+1}(Y;R)$ is independent of topologically simple fillings $W$ and $\rho$, as long as $\rho|_{L_0Y}$ is fixed.
\end{proposition}
To illustrate the necessity of the ADC condition, we consider the cotangent circle bundle $ST^*S^2\simeq \RP^3$, then the minimal Conley-Zehnder index is $1$, which makes it not ADC. By \cite{albers2017local}, there is a local system on $L_0T^*S^2$, such that $SH^*(T^*S^2,\rho)=0$. However, the local system on $L_0\RP^3$ is trivial, hence it has a trivial extension to $L_0T^*S^2$, whose symplectic cohomology is not zero. As a consequence, $\delta_{\partial,\rho}$ depends on $\rho$ in this non-ADC case.

To make Proposition \ref{prop:indloc} useful as Theorem \ref{thm:A}, we will explain that $SH^*(W;\rho)$ is a ring and $H^*(W;R)\to SH^*(W;\rho)$ is a ring map. The product structure on symplectic cohomology is defined by counting holomorphic curves on $3$-punctured spheres, with two positive puncture and one negative puncture, c.f. \cite{ritter2013topological}. We fix a trivialization of $\rho$ on constant loops. Given any such curve $u$ with positive ends asymptotic to $x,y$ and negative end asymptotic to $z$, it determines a module morphism $\rho_u:\rho_x\otimes_{R}\rho_y\to \rho_z$ as follows,
\begin{equation}\label{eqn:product}
\rho_x\otimes_R \rho_y \stackrel{\rho_{u_1}\otimes_R \rho_{u_2}}{\longrightarrow} \rho_{x'}\otimes_R \rho_{y'}\stackrel{m}{\to} \rho_{x'\star y'} \stackrel{\rho_{u_3}}{\longrightarrow} \rho_z.
\end{equation}
where $x',y'$ are the based loops with the same base point as the following picture shows, $x'\star y'$ is the concatenation of $x'$ and $y'$, $u_1,u_2,u_3$ are the cylinders represented by the colored region and $m:\rho_{x'}\otimes_R \rho_{y'}\to \rho_{x'\star y'}$ is the conical isomorphism in \cite[Lemma 1]{albers2017local} if we fix a trivialization of $\rho$ on constant loops.

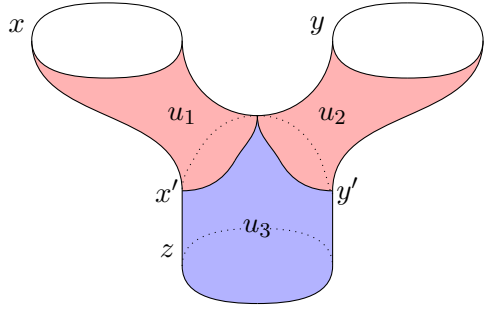
\begin{figure}[H]
	\begin{tikzpicture}
	\path [fill=red!30] (0,0) to [out=270,in=180] (1,-0.5)
	to [out=0,in=270] (2,0) to [out=270,in=180] (3,-1) 
    to [out=270, in=60] (2.7,-1.6) to [out=240, in=0] (2,-2)
    to [out=90,in =270] (0,0); 
	\node at (2,-1) {$u_1$};
	
	\path [fill=red!30] (6,0) to [out=270,in=0] (5,-0.5)
	to [out=180,in=270] (4,0) to [out=270,in=0] (3,-1) 
	to [out=270, in=120] (3.3,-1.6) to [out=300, in=180] (4,-2)
	to [out=90,in =270] (6,0); 
	\node at (4,-1) {$u_2$};
		
	\path [fill=blue!30] (2,-2) to (2,-3)
	to [out=270,in=180] (3,-3.5) to [out=0,in=270] (4,-3) 
	to (4,-2) to [out=180, in=300] (3.3,-1.6)
	to [out=120,in =270] (3,-1) to [out=270, in=60] (2.7,-1.6)
	to [out=240, in=0] (2,-2);
	\node at (3,-2.5) {$u_3$};
	
	\draw (0,0) to [out=90, in=180] (1,0.5);
	\draw (1,0.5) to [out=0, in=90] (2,0);
	\draw (2,0) to [out=270, in=0] (1,-0.5);
	\draw (1,-0.5) to [out=180, in=270] (0,0);
    \node at (-0.2,0.2) {$x$};

	\draw (4,0) to [out=90, in=180] (5,0.5);
	\draw (5,0.5) to [out=0, in=90] (6,0);
	\draw (6,0) to [out=270, in=0] (5,-0.5);
	\draw (5,-0.5) to [out=180, in=270] (4,0);
	\node at (3.8,0.2) {$y$};
	
	\draw (0,0) to [out=270, in=90] (2,-2);
	\draw (2,-2) to (2,-3);
	
	\draw[dotted] (2,-3) to [out=90, in=180] (3,-2.5);
	\draw[dotted] (3,-2.5) to [out=0, in=90] (4,-3);
	\draw (4,-3) to [out=270, in=0] (3,-3.5);
	\draw (3,-3.5) to [out=180, in=270] (2,-3);
	\node at (1.8,-2.8) {$z$};
	
	\draw (6,0) to [out=270, in=90] (4,-2);
	\draw (4,-3) to (4,-2);	
	
	\draw (2,0) to [out=270, in=180] (3,-1);
	\draw (4,0) to [out=270, in=0] (3,-1);
	
	\draw (3,-1) to [out=270, in=60] (2.7,-1.6);
	\draw (2.7,-1.6) to [out=240, in=0] (2,-2); 
	\node at (1.8,-2) {$x'$};
	\draw[dotted] (3,-1) to [out=180, in=60] (2.2,-1.5);
	\draw[dotted] (2.2,-1.5) to [out=240, in=90] (2,-2); 
	
	\draw (3,-1) to [out=270, in=120] (3.3,-1.6);
	\draw (3.3,-1.6) to [out=300, in=180] (4,-2); 
	\node at (4.2,-2) {$y'$};
	\draw[dotted] (3,-1) to [out=0, in=120] (3.8,-1.5);
	\draw[dotted] (3.8,-1.5) to [out=300, in=270] (4,-2); 
	\end{tikzpicture}
	\caption{Parallel transportation on pair of pants}\label{fig:pants}
\end{figure}

\begin{proposition}\label{prop:pants}
	Given a pair-of-pants curve $u$, $\rho_u$ in \eqref{eqn:product} is well-defined and only depends on the homotopy class of $u$.
\end{proposition}
\begin{proof}
	Since the merging of $x,y$ into $x'\star y'$ is not unique, the potential ambiguity follows from different ways of merging. More explicitly, if we have two choices of merged based loops $x',y'$ and $x'',y''$, then there are two cylinders $u_x,u_y$ from $x',y'$ to $x'',y''$ respectively, and $u_x(s,0)=u_y(s,0), \forall s$. Then we can form a concatenation of cylinders $u_x\star u_y$ from $x'\star y'$ to $x''\star y''$.  Then it suffices to prove the following commutative diagram,
	\begin{equation}\label{eqn:com}
	\xymatrix{
	\rho_{x'}\otimes_R \rho_{y'} \ar[r]^{m} \ar[d]^{\rho_{u_x}\otimes_R \rho_{u_y}} & \rho_{x'\star y'}\ar[d]^{\rho_{u_x\star u_y}} \\
	\rho_{x''}\otimes_R \rho_{y''} \ar[r]^m & \rho_{x''\star y''}
     }	
	\end{equation}
	Let $p$ be the common base point of $x',y'$, i.e. $p=x'(0)=y'(0)$, we recall $m$ from \cite[Lemma 1]{albers2017local}, 
	$$m: \rho_{x'}\otimes \rho_{y'}\to \rho_p\otimes \rho_p \to \rho_p \to \rho_{x'\star y'},$$
	where the first map is determined by two paths $\gamma_{x'},\gamma_{y'}$ of based loops from $x',y'$ to the constant loop $p$, the middle map is determined by the conical isomorphism $R\otimes_{R} R \to R$ since we fix a trivialization over constant loops, the last map is determined by the inverse of the concatenation of $\gamma_{x'}$ and $\gamma_{y'}$, which is a path from $p$ to $x'\star y'$ in the based loop space. By the homotopy lifting property of $L_0W \to W$, there exists a lift $\gamma_{x',x''}$ of $u_x(s,0)$ extending $\gamma_{x'}$, similarly for $y',y''$, then we have the following commutative diagram
	$$
	\xymatrix{
	\rho_{x'}\otimes_R \rho_{y'}  \ar@/^1pc/[rrr]^m \ar[r] \ar[d]^{\rho_{\gamma_{x',x''}}\otimes_R\rho_{\gamma_{y',y''}}} & \rho_{x'(0)}\otimes_R \rho_{y'(0)} \ar[r]\ar[d] & \rho_{x'(0)} \ar[r]\ar[d] & \rho_{x'\star y'}\ar[d]^{\rho_{\gamma_{x',x''}\star \gamma_{y',y''}}} \\
    \rho_{x''}\otimes_R \rho_{y''} \ar@/_1pc/[rrr]_m  \ar[r] & \rho_{x''(0)}\otimes_R \rho_{y''(0)} \ar[r] & \rho_{x''(0)} \ar[r] & \rho_{x''\star y''}
    }
	$$
	$\gamma_{x',x''}$ may not homotopic to $u_x$. But to prove \eqref{eqn:com}, it suffices to prove the following commutative diagram
	\begin{equation}\label{eqn:com2}
	\xymatrix{
		\rho_{x''}\otimes_R \rho_{y''} \ar[r]\ar[d]^{\rho_v\otimes_R \rho_w } & \rho_{x''\star y''}\ar[d]^{\rho_{v\star w}}\\
		\rho_{x''}\otimes_R \rho_{y''}\ar[r]  & \rho_{x''\star y''}
    }
	\end{equation}
	where $v$ and $w$ are two loops in $L_0W$ such that $v(s,0)=w(s,0)$ for $s\in S^1$. Note that the local system is trivial on $W$, we can modify $v$ by elements in $\Ima(\pi_1(W)\to \pi_1(L_0W))$ from the inclusion of constant loops without changing $\rho_v$. Therefore we can assume $v(s,0)=x''(0)$ for all $s\in S^1$. We may assume the same thing for $w$. Let $\gamma_{x''},\gamma_{y''}$ denote the path in $\Omega_0W$ connecting $x'',y''$ to $x''(0)=y''(0)$. Then it suffices to prove the following commutative diagram
		\begin{equation}
		\xymatrix{
		\rho_{x''(0)}\otimes_R \rho_{y''(0)} \ar[rrr]\ar[d]^{\rho_{\gamma_{x''}v \gamma^{-1}_{x''}}\otimes_R \rho_{\gamma_{y''}w\gamma_{y''}^{-1}} } & & & \rho_{x''(0) }\ar[d]^{\rho_{(\gamma_{x''}\star\gamma_{y''})(v\star w) (\gamma_{x''}\star \gamma_{y''})^{-1} }}\\
		\rho_{x''(0)}\otimes_R \rho_{y''(0)}\ar[rrr] & & & \rho_{x''(0)}}
        \end{equation}
	Therefore it suffices to prove that $\rho_{\gamma_{x''} v \gamma_{x''}^{-1}\gamma_{y''}w\gamma_{y''}^{-1}}=\rho_{(\gamma_{x''}\star\gamma_{y''})(v\star w) (\gamma_{x''}\star \gamma_{y''})^{-1} }$. Therefore it is enough to verify that $\gamma_{x''} v \gamma_{x''}^{-1}\gamma_{y''}w\gamma_{y''}^{-1}$ is homotopic to $(\gamma_{x''}\star\gamma_{y''})(v\star w) (\gamma_{x''}\star \gamma_{y''})^{-1} $ as a loop in $L_0W$. Note that  $(\gamma_{x''}\star\gamma_{y''})(v\star w) (\gamma_{x''}\star \gamma_{y''})^{-1} = (\gamma_{x''} v \gamma_{x''}^{-1})\star (\gamma_{y''} w \gamma_{y''}^{-1})$ and both $(\gamma_{x''} v \gamma_{x''}^{-1}),(\gamma_{y''} w \gamma_{y''}^{-1})$ represent a based loop of based loops in $\Omega\Omega_0W$, where the base is the constant loop $x''(0)$. Then the claim follows from the same proof for the fact that second homotopy groups are abelian. Note that \eqref{eqn:com} also implies that $\rho_u$ only depends on the homotopy class of $u$.
\end{proof}

As a consequence of Proposition \ref{prop:pants}, we can count pair-of-pants moduli spaces twisted by $\rho_u$ to define a ring structure on $SH^*(W;\rho)$. That $H^*(W;R)\to SH^*(W;\rho)$ is a ring map follows from the same argument for untwisted symplectic cohomology, see \cite{ritter2013topological}.
\begin{proof}[Proof of Theorem \ref{thm:local}]
	We use local systems with $R=\C$. The assumption implies that $\dim Q\ge 3$, in particular, $\pi_1(ST^*Q)\to \pi_1(T^*Q)$ is an isomorphism. We first assume that $Q$ is oriented and spin. Following \cite{albers2017local}, we consider local systems on $L_0W$ trivial on $W$ induced from $\Z/p\subset \C^*$, which are represented by $\Hom_{inv}(\pi_2(W),\Z/p)\simeq H^2_{inv}(\widetilde{W},\Z/p)$, where $\widetilde{W}$ is the universal cover of $W$ and $\pi_1(W)$ acts on the cohomology by deck transformation. By \cite[Proposition 6, 9]{albers2017local}, $H^2(W)\stackrel{\pi^*}{\to} H^2_{inv}(\widetilde{W})$ contains nontrivial images iff the Hurewicz map $\pi_2(W)\to H_2(W)$ is nonzero.  In the case of \eqref{l1}, by \cite[Theorem 1]{albers2017local}, there exists an admissible local system $\rho$ such that $SH^*(T^*Q;\rho)=0$. Moreover, such local system is represented by an element in the image of $H^2(T^*Q;\Z/p)\to H^2_{inv}(\widetilde{T^*Q};\Z/p)$ for some prime $p$. Then the local system on $L_0ST^*Q$ is represented by the restriction to $H^2_{inv}(\widetilde{ST^*Q};\Z/p)$	and it is from an element of $H^2(ST^*Q;\Z/p)$. By assumption, $H^2(W;\Z/p)\to H^2(ST^*Q;\Z/p)$ is surjective, which implies that $\rho|_{L_0ST^*Q}$ extends to $W$. Hence the claim follows from Proposition \ref{prop:indloc}. If $Q$ is non-orientable or not spin, given a local system $\rho$, we have $SH^*(T^*Q,\rho\otimes \sigma^{TQ})=H_{n-*}(LQ,\rho\otimes o_Q)$  \cite[Chapter 3]{abouzaid2013symplectic}, where $o_Q$ is the orientation local system on $Q$ and $\sigma^{TQ}$ is an admissible local system defined in \cite[Chapter 2]{abouzaid2013symplectic}, which is transgressed from $H^2(Q,\Z/2)$ (in the oriented not spin case, it is transgressed from the second Stiefel-Whitney class). Hence by assumption, this local system  $\sigma^{TQ}$ can be extended to $W$. Hence the invariance follows similarly as we can find an admissible local system $\rho$ such that $H_{n-*}(LQ,\rho\otimes o_Q)=0$ by \cite[Theorem 1]{albers2017local}. 
	
	For case \eqref{l2}, if $Q$ is spin, there exists some prime $p$ such that $H^2(W;\Z/p)\to H^2(ST^*Q;\Z/p)$ is nonzero. Since when $\pi_1(Q)=0$, the any nontrivial element of $H^2(T^*Q;\Z/p)=H^2(ST^*Q;\Z/p)$ determines a local system with vanishing symplectic cohomology, then the claim follows. For not spin manifold, $\sigma^{TQ}$ is the transgression of the second Stiefel-Whitney class, hence it can also be extended to $W$ by our condition.

	If $\chi(Q)=0$, then we have a section $s:Q\to ST^*Q$ and $Q \stackrel{s}{\to} ST^*Q \to T^*Q$ is a homotopy equivalence. Therefore the invariance of $H^*(W;\C)\to H^*(ST^*Q;\C)$ implies that $H^*(W;\Q)\to H^*(ST^*Q;\Q)\stackrel{s^*}{\to} H^*(Q;\Q)$ is an isomorphism. In particular, the composition $Q\stackrel{s}{\to} ST^*Q \hookrightarrow W$ induces a rational homotopy equivalence.       
\end{proof}
\begin{remark}
	It is possible to get invariance of $\Z/p$ cohomology for $p>2$ under the similar assumptions as in Theorem \ref{thm:local}. We can consider $\Z/p$-local systems that are trivial on constant loops, such local system is classified by $\Hom(\pi_1(L_0W), (\Z/p)^{\times})/(\Z/p)^{\times}$. Since for any $g\ne 1\in (\Z/p)^{\times}$, we have $1-g\in (\Z/p)^\times$, then \cite[Proposition 3]{albers2017local} applies to get vanishing of symplectic cohomology.
\end{remark}

\subsection{Diffeomorphism type}
In some situations, the invariance of cohomological information $H^*(W;\Z)\to H^*(Y;\Z)$ would yield information about the diffeomorphism type via h-cobordism. Such argument was used extensively in \cite{barth2016diffeomorphism}. In the following, we extract some of the topological conditions from \cite{barth2016diffeomorphism} such that the h-cobordism argument can be used. In particular, combining with the results from previous sections, we get the uniqueness of diffeomorphism type for many cases.

\begin{definition}
	An oriented manifold $W$ with boundary is good if the following conditions hold.
	\begin{enumerate}
		\item There exists $W_0$ diffeomorphic to $W$ and is contained in the collar neighborhood of $\partial W$.
		\item  There exists a locally closed manifold $V\subset \partial W$, such that $V\to \partial W \to W$ induces isomorphism on cohomology and the corresponding copy of $V_0$ on $W_0$ is homotopic to $V$ in the collar neighborhood of $\partial W$.
	\end{enumerate}
\end{definition}

\begin{example}
	The major source of good domains is the following.
	\begin{enumerate}
		\item $W=DT^*Q$ for manifold $Q$ with $\chi(Q)=0$. Then there exists a section $s:Q\to ST^*Q$. Then we may take $V=s(Q)$.
		\item $W=M\times \D$. Then $V$ may be taken as $M\times \{1\}$.
	\end{enumerate}
\end{example}

\begin{proposition}\label{prop:homology}
Let $W$ be a good domain with boundary  $Y$. If $Y$ has a another filling $W'$ such that $H^*(W')\to H^*(Y)$ is isomorphic to $H^*(W)\to H^*(Y)$, then $W'=W\cup H$, where $H$ is a homology cobordism from $Y$ to $Y$.
\end{proposition}
\begin{proof}
Since in the collar neighborhood of $Y$ in $W'$, we have a $W_0$ diffeomorphic to $W$. By deleting this $W_0$, we have cobordism $X$ from $Y$ to $Y$. We claim $H^*(W')\to H^*(W_0)$ is an isomorphism. Let $F:V\times [0,1]$ be the homotopy from $V$ to $V_0$ in the collar neighborhood. This follows from the following commutative diagram
$$
\xymatrix{H^*(W') \ar[r]^{F^*} \ar[d] & H^*(V\times [0,1]) \ar[r]^{\quad \simeq}\ar[rd]^{\simeq} & H^*(V)\\
H^*(W_0) \ar[r] & H^*(\partial W_0) \ar[r] & H^*(V_0)
}
$$
By assumption, we have $H^*(W_0)\to H^*(\partial W_0)\to H^*(V_0)$ is an isomorphism and the first arrow $H^*(W')\stackrel{F^*}{\to} H^*(V\times [0,1])\to H^*(V)$ is identified with $H^*(W')\to H^*(Y)\to H^*(V)$, hence an isomorphism. Therefore $H^*(W')\to H^*(W_0)$ is an isomorphism. Then by excision we have $H^*(W',W_0)=H^*(X,\partial W_0)=0$. Then the universal coefficient and Poincare duality implies that $H_*(X,\partial W_0)=H_*(X,\partial W')=0$. Therefore $X$ is a homology cobordism.
\end{proof}

Once we know that $Y$ is simply connected and the cobordism $X$ in Proposition \ref{prop:homology} is simply connected, the homology cobordism becomes $h$-cobordism by the relative Hurewicz theorem. Therefore, when $\dim X \ge 6$, we have $X$ is diffeomorphic to $Y\times [0,1]$. In particular, we can prove Theorem \ref{thm:diff}.

\begin{proof}[Proof of Theorem \ref{thm:diff}]
Contact manifolds in the statement of Theorem \ref{thm:diff} are ADC, where	case \eqref{f1} follows from  Theorem \ref{thm:product}. 	It follows from Corollary \ref{cor:B} and Proposition \ref{prop:homology}  that different fillings are differed by homology cobordisms. Note that by Theorem \ref{thm:fund}, we have $W$ is simply connected. Then the homology-cobordism $X$ from Proposition \ref{prop:homology} is simply connected by the Van Kampens Theorem. In particular, $X$ is an h-cobordism. Therefore the diffeomorphism type is unique.
\end{proof}

\section{Persistence of symplectic dilation}\label{s4}
Symplectic cohomology is naturally equipped rich algebraic structures, in particular, $SH^*(W)$ is a BV algebra with a degree $-1$ BV operator $\Delta$, c.f. \cite{seidel2012symplectic}. With such structure, symplectic dilation was introduced in \cite{seidel2012symplectic} as an element $x$ of $SH^1(W)$, such that $\Delta(x)=1$.
\begin{example}
	If $SH^*(W) = 0$, then $0$ is a symplectic dilation. If $W$ admits a symplectic dilation and $M$ be another Liouville domain, then $W\times M$ admits a symplectic dilation by \cite[Example 2.6]{seidel2014disjoinable}.
\end{example}

\begin{example}[{\cite[Example 6.4]{seidel2012symplectic}}]\label{ex:coeff}
	$T^*S^2$ admits a dilation with field coefficient $\bm{k}$ if $char(\bm{k}) \ne 2$, and $T^*S^n$ admits dilation for any coefficient when $n\ge 3$. $T^*\CP^n$ also admits dilation if $char(\bm{k})\ne 2$. $T^*K(\pi,1)$ does not admit dilation, in particular there is no dilation on $T^*{T^n}$. 
\end{example}

More examples with symplectic dilations are constructed by Lefschetz fibrations, since by \cite[Proposition 7.3]{seidel2012symplectic}, if the fiber $F$ of a $2n$-dimensional Lefschetz fibration $M\to \C$ contains a dilation for $n>2$, then so does $M$. By repeatedly applying this result, Seidel \cite[Example 2.13]{seidel2014disjoinable} showed that the Milnor fiber of a singularity of form $p=z_0^2+z_1^2+z_2^2 + q(z_3,\ldots,z_n)$ admits symplectic dilations.

\begin{example}\label{ex:ADCdilation}
	The link of $p=z_0^2+z_1^2+z_2^2 + z_3^{a_3}+\ldots + z^{a_n}_n=0$ with $2\le a_i \in \N$ is ADC. The associated Milnor fiber provides examples with ADC boundaries admitting symplectic dilations. The contact boundary admits a Morse-Bott contact form whose generalized Conley-Zehnder indices were computed in \cite[\S 5.3.1]{kwon2016brieskorn}. The Reeb orbits are completely classified by two natural numbers $T,N$. Let $I_T$ be the maximal subset of $ I:=\{0,1,\dots,n\}$ such that $lcm_{i\in I_T} a_i=T$ and $I_T$ has at least two elements. There are also some restrictions on $N$, see \cite{kwon2016brieskorn} for detail. Then there is a $2(\#I_{T}-2)$-dimensional family of Reeb-orbits with period $NT$ and generalized Conley-Zehnder index
	\begin{equation}\label{eqn:index}
	2\sum_{i\in I_T}\frac{NT}{a_i}+2\sum_{i\in I-I_T}\lfloor\frac{NT}{a_i}\rfloor+\#(I-I_T)-2NT.
	\end{equation}
	Then the minimal Conley-Zehnder index after a small perturbation is \eqref{eqn:index} minus $\#I_T-2$. When $T$ is even, the Conley-Zehnder index of a small perturbation is at least $NT+2N(\#I_T-3)+\#I -2 \#I_T+2$, hence the degree is at least $NT+2N(\#I_T-3)+2\#I-2\#I_T-2$ which is positive if $N>1$. When $N=1$, it is least $2\#I-6\ge 2$.  When $T\ge 3$ is odd, the degree is at least $2N\#I_T+3NT-3+2\#I-2\# I_T-2-2NT\ge NT+2\# I -5>0$. Therefore such manifold is ADC.
\end{example}

Like the vanishing of symplectic cohomology, the existence of symplectic dilation is also preserved under the Viterbo transfer map. That is, let $V\subset W$ be a Liouville subdomain, then the Viterbo transfer map $SH^*(W) \to SH^*(V)$ preserves the BV structure. In particular, if $W$ admits a symplectic dilation, so does $V$. Therefore, the existence of symplectic dilation may be viewed as an indication of the complexity of the Liouville domain, which is next to the vanishing of symplectic cohomology. The goal of this section is to prove for ADC contact manifolds, the existence of symplectic dilation is a property independent of the filling in many cases, hence measures the contact complexity. This is done by showing independence of a structure map as before, which also bears interests.
\subsection{BV operator $\Delta$}
Similar to the discussion in \S \ref{s2}, we will define $\Delta$ using $\bH$. However, to make sure Lemma \ref{lemma:max} can be applied. We need to consider two such functions $\bH_+,\bH_-$, such that the following holds.
\begin{enumerate}
	\item $\bH_+$ and $\bH_-$ satisfy same conditions of $\bH$ in \S \ref{s2}, and share the same $(a_i,b_i)$ and $D^+_i\le D^{-}_i$.
	\item $\min_{t\in S^1}\bH_-(t,x)\ge \max_{t\in S^1}\bH_+(t,x)$ for all $x\in \widehat{W}$.
	\item On each $(a_i,b_i)$, $\bH_+=f_+(r)$ and $\bH_-=f_-(r)$ and $rf'_+(r)-f_+(r) \le r f'_-(r)-f_-(r)$. 
\end{enumerate}
Roughly speaking the requirements above ask that $\bH_-$ grows faster than $\bH_+$, for example $\bH_-$ is a perturbation of $2r^2$ and $\bH_+$ is a perturbation of $r^2$.  We fix such two functions, we also fix a smooth decreasing function $\rho(s):\R \to \R$, such that $\rho(s)=1$ for $s<0$ and $\rho(s)=0$ for $s>1$. Then we define $$\bH^{\theta}_{s,t}:=\rho(s)\bH_-(t+\theta)+(1-\rho(s))\bH_+(t).$$
Then for $s<0$, we have $\bH^{\theta}_{s,t}=\bH_-(t+\theta)$ and for $s>1$ we have $\bH^{\theta}_{s,t}=\bH_+(t)$. Moreover by construction $\partial_s \bH^{\theta}_{s,t}=\rho'(s)(\bH_-(t+\theta)-\bH_+(t))\le 0$. Moreover, on region $(a_i,b_i)$, $\bH^{\theta}_{s,t}=\rho(s)f_-(r)+(1-\rho(s))f_+(r)$. Then 
\begin{eqnarray*}
	\partial_s(r\partial_r\bH^{\theta}_{s,t}-\bH^{\theta}_{s,t}) & = & \partial_s(r(\rho(s)f_-(r)+(1-\rho(s))f_+(r))'-(\rho(s)f_-(r)+(1-\rho(s))f_+(r)))\\
	& = & \rho'(s)(r f'_-(r)-f_-(r)-rf'_+(r)+f_+(r))\le 0.
\end{eqnarray*}
In particular, the conditions in Lemma \ref{lemma:max} are satisfied for $\bH^{\theta}_{s,t},\forall\theta \in S^1$ on $(a_i,b_i)$.
\begin{remark}
	The extra complexity is due to that $\bH$ depends on $t$. If we twist $\bH$ to get $\bH^{\theta}_{s,t}$ it is never true that $\partial_s\bH^{\theta}_{s,t}\le 0$. Hence Lemma \ref{lemma:max} can not be applied and the compactness proof fails. An alternative fix is using an autonomous Hamiltonian and the cascades moduli spaces to define the BV operator as in \cite{bourgeois2009symplectic}.
\end{remark}

Let $\cJ_{s,\theta}(W)$ be the set of smooth families of admissible almost complex structures $J^{\theta}_{s,t}:\R_s\times S^1_\theta \to \cJ(W)$, such that there exist $J_-,J_+\in \cJ(W)$ with $J^{\theta}_{s,t}=J_{-,t+\theta}$ when $s<0$ and $J^{\theta}_{s,t}=J_{+,t}$ when $s>1$. Let $\cJ_{s,\theta,J_-,J_+}(W)\subset \cJ_{s,\theta}(W)$ be the set of families with positive ends is given by $J_+$ and negative end is given by $J_-$. Then for $J_+\in \cJ_{reg}(\bH_+,f,g), J_-\in \cJ_{reg}(\bH_-,f,g)$, for a generic choice of $J^{\theta}_{s,t}\in \cJ_{s,\theta,J_-,J_+}(W)$, we have the following moduli spaces.
\begin{enumerate}
	\item For $x \in \cP^*(\bH_-), y\in \cP^*(\bH_+)$, $\cM^\Delta_{x,y}$ is defined to be the compactification of the  moduli space of solutions $(u,\theta)$ to the following
	\begin{equation}\label{eqn:BV1}
	\partial_s u + J^{\theta}_{s,t}(\partial_t u - X_{\bH^\theta_{s,t}}) = 0, \quad \lim_{s\to -\infty} u = x(\cdot + \theta),  \lim_{s\to \infty} u = y.
	\end{equation}
	\item  For $x \in \cP^*(\bH_+), p\in \cC(f)$, $\cM^{\Delta}_{p,x}$ is defined to be the moduli space of solutions $(u,\theta,\gamma)$ to the following
	\begin{equation}\label{eqn:BV2}
	\partial_s u + J^{\theta}_{s,t}(\partial_t u - X_{\bH^\theta_{s,t}}) = 0, \frac{\rd}{\rd s} \gamma +\nabla_g f=0, \quad \gamma(-\infty)=p,  u(0) = \gamma(0), \lim_{s\to \infty} u = x.
	\end{equation}
\end{enumerate}
Since $\partial_s \bH^{\theta}_{s,t}\le 0$, any solution $u$ to $\partial_s u + J^{\theta}_{s,t}(\partial_t u - X_{\bH^\theta_{s,t}}) = 0$ will have the property that 
\begin{equation}\label{eqn:actioncontrol}
\cA_{\bH_-}(u(-\infty)(\cdot - \theta))-\cA_{\bH_+}(u(\infty))\ge 0.
\end{equation}
As a consequence,  there is no $\cM^{\Delta}_{x,p}$ for $p\in \cC(f)$ and $x\in\cP^*(\bH_-)$. By the construction of $\bH^{\theta}_{s,t}$, Lemma \ref{lemma:max} can be applied to get compactness of $\cM^{\Delta}_{*,*}$. Therefore we have the following with a similar proof to Proposition \ref{prop:compact} and Proposition \ref{prop:transversality}.
\begin{proposition}\label{prop:delta}
	For a generic choice of $J^{\theta}_{s,t}$, we have $\cM^{\Delta}_{a,b}$ is a compact manifold with boundary of dimension $|a|-|b|+1$ when $|a|-|b|\le 0$. And $\partial \cM^{\Delta}_{a,b}:= -\sum \cM^\Delta_{a,*}\times \cM^+_{*,b}-\sum \cM^-_{a,*}\times \cM^{\Delta}_{*,b}$, where $\cM^-_{*,*},\cM^+_{*,*}$ are moduli spaces associated to $\bH_-$ and $\bH_+$. 
\end{proposition}
By this boundary configuration, $\Delta$ defines a cochain map $C(\bH_+,J_+,f)\to C(\bH_-,J_-,f)[-1]$. Moreover $\Delta$ decomposes into $\Delta_+$ and $\Delta_{+,0}$, which count $\cM^{\Delta}_{x,y}$ and $\cM^{\Delta}_{p,x}$ respectively. In particular, $\Delta_{+}$ is a cochain map  $C_+(\bH_+,J_+)\to C_+(\bH_-,J_-)[-1]$.  By \eqref{eqn:actioncontrol}, $\Delta$ maps $C^{D^+_i}(\bH_+,J_+,f)$ to $C^{D^-_i}(\bH_-,J_-,f)$. Therefore we use $\Delta^{D^+_i},\Delta^{D^+_i}_+$ and $\Delta^{D^+_i}_{+,0}$ to denote the restrictions respectively.  Lemma \ref{lemma:max} implies that curves appearing in $\Delta^{D^+_i}$ are contained in $W^i$.

Let $\delta^-$ denote the connecting map $H^*(C_+(\bH_-,J_-)) \to H^{*+1}(C_0(f))$. Then we can define a degree $1$ map 
\begin{equation}\label{eqn:phi}
\phi: \ker \Delta_+ \subset H^*(C_+(\bH_+,J_+)) \to \coker \delta^{-}, \quad x \mapsto \Delta_{+,0}(x)-d_{+,0}(b),
\end{equation}
where $x\in C_+(\bH_+,J_+),b\in C_+(\bH_-,J_-)$ such that $\Delta_+(x) = d_+(b)$.
\begin{proposition}\label{prop:welldefined}
	$\phi$ is well-defined. If $\psi$ is a cochain map on $C=C_+\oplus C_0$
	 that can be decomposed into $\psi_++\psi_0+\psi_{+,0}$, such that $\Delta \circ \psi-\psi\circ \Delta=\eta \circ d-d\circ \eta$ for $\eta=\eta_++\eta_{+,0}$. Then on cohomology we have $\phi\circ \psi_+=\psi_0\circ \phi$. 
\end{proposition}
\begin{proof}
	Note that we have $d^2 = 0$ and $\Delta\circ d + d \circ \Delta = 0$, when we write them using the zero and positive decomposition, we have the following formula.
	\begin{eqnarray}
	d_0 \circ \Delta_{+,0}+ d_{+,0} \circ \Delta_+(x) + \Delta_{+,0}\circ d_+(x) & = & 0, \\
	d_0\circ d_{+,0}+ d_{+,0}\circ d_+ &= & 0.
	\end{eqnarray}
	We first show that $d_0\circ \phi(x) = 0$, that is 
	\begin{eqnarray*}
		d_0\circ \phi (x) & = & d_0 \circ \Delta_{+,0}(x) - d_0\circ d_{+,0}(b) \\
		& = & -d_{+,0} \circ \Delta_+(x) - \Delta_{+,0}\circ d_+(x)  - d_0\circ d_{+,0}(b) \\
		&= & -d_{+,0} \circ d_+(b) - \Delta_{+,0}\circ d_+(x)  - d_0\circ d_{+,0}(b) \\
		& = & -\Delta_{+,0}\circ d_+(x). 
	\end{eqnarray*}
	Since $d_+(x)=0$, we have $\phi(x)$ is closed. 
	
	Now we consider $x':= x + d_+(y)$, then we can choose $b':=b+\Delta_+(y)$, then we have
	\begin{eqnarray*}
		\phi(x')-\phi(x) & = & \Delta_{+,0}\circ d_+(y) - d_{+,0}\circ \Delta_+(y) \\
		& = & d_0\circ \Delta_{+,0}(y).
	\end{eqnarray*}
	Hence the difference is exact. Finally, we consider $b'=b+c$ where $d_+(c) = 0$. Then we have the difference is $d_{+,0}(c)$ which is 
	in the image of $\delta$.
	
	If we have such $\psi$ and $\eta$, then we have the following 
	\begin{eqnarray}
	\psi_+\circ d_+& = & d_+\circ \psi_+,\label{eqn:r1}\\
	\psi_{0}\circ d_{+,0}+\psi_{+,0}\circ d_+ &= & d_0\circ \psi_{+,0}+d_{+,0}\circ \psi_{+},\label{eqn:r2} \\
	\psi_+\circ \Delta_+-\Delta_+\circ\psi_+ & = &  d_+\circ\eta_+ -\eta_+\circ d_+,\label{eqn:r3}\\
	\psi_0\circ \Delta_{+,0} + \psi_{+,0}\circ \Delta_{+}-\Delta_{+,0}\circ\psi_+ & = &  d_0\circ \eta_{+,0} + d_{+,0}\circ \eta_{+} -\eta_{+,0}\circ d_+. \label{eqn:r4}
	\end{eqnarray}
	Then \eqref{eqn:r1} implies that $\psi_+$ induces a map on $H^*(C_+)$. \eqref{eqn:r2} implies that $\delta\circ \psi_+=\psi_0\circ \delta$ on  cohomology. \eqref{eqn:r3} implies that $\psi_+\circ \Delta_+=\Delta_+\circ \psi_+$ on cohomology. Lastly, for $x,b\in C_+$ with $\Delta_+(x)=d_+(b), d_+(x)=0$, we have
	\begin{eqnarray*}
		\Delta_{+}\circ \psi_+(x) & \stackrel{\eqref{eqn:r3}}{=} & \psi_+ \circ \Delta_{+}(x)-d_+\circ \eta_+(x)+\eta_{+}\circ d_+(x)\\
		& = & \psi_+\circ d_+(b)-d_+\circ \eta_+(x) \\
		& \stackrel{\eqref{eqn:r1}}{=} & d_+\circ \psi_+(b)-d_+\circ \eta_+(x).
	\end{eqnarray*}
	Therefore we have
	\begin{eqnarray*}
		\phi \circ \psi_+(x) & = & \Delta_{+,0}\circ \psi_+(x)-d_{+,0}\circ \psi_{+}(b)+d_{+,0}\circ \eta_+(x)\\
		& \stackrel{\eqref{eqn:r4}}{=} &  \psi_0\circ\Delta_{+,0}(x) + \psi_{+,0}\circ \Delta_{+}(x)- d_0\circ \eta_{+,0}(x) -d_{+,0}\circ \psi_{+}(b)\\
		&= &  \psi_0\circ\Delta_{+,0}(x) + \psi_{+,0}\circ d_{+}(b)-d_{+,0}\circ \psi_{+}(b)-d_0\circ \eta_{+,0}(x)\\
		& \stackrel{\eqref{eqn:r2}}{=} & \psi_0\circ \Delta_{+,0}(x) -\psi_0\circ d_{+,0}(b)+d_0\circ \psi_{+,0}(b)-d_0\circ \eta_{+,0}(x)\\
		& = & \psi_0\circ\phi(x) +d_0\circ \psi_{+,0}(b)-d_0\circ \eta_{+,0}(x).
	\end{eqnarray*}
	That is $\phi\circ \psi_+=\phi_0\circ \phi$ on cohomology.
\end{proof}

\subsection{Continuation maps}
There are two continuation maps we need to consider, one is the continuation from $\bH_+$ to $\bH_-$. The other is the continuation from homotopies of almost complex structures and its compatibility with $\Delta_+$ and $\phi$.

\subsubsection{Continuation map from $\bH_+$ to $\bH_-$}\label{ssb:cont} Let $\bH_s:=\rho(s)\bH_-+(1-\rho(s))\bH_+$, that is $\bH^0_{s,t}$. In particular, Lemma \ref{lemma:max} can be applied to $\bH_s$ on $(a_i,b_i)$.  Then we can consider the compactified moduli spaces of the following
\begin{eqnarray}
\{u:\R\times S^1 \to \widehat{W}\st\partial_s u + J_s(\partial_t u -X_{\bH_s})=0, \lim_{s\to -\infty} u=x, \lim_{s\to \infty} u = y \}&& x \in \cP(\bH_-),y\in \cP^*(\bH_+).\\
\left\{\begin{array}{l}u:\C \to \widehat{W},\\ \gamma:(-\infty,0]\to W \end{array}\left| \begin{array}{l}\partial_s u + J_s(\partial_t u -X_{\bH_s})=0, \partial_s \gamma + \nabla_g f=0,\\
\displaystyle\lim_{s\to -\infty}\gamma =p, u(0)=\gamma(0), \lim_{s\to \infty} u = x
\end{array}\right.\right\} &&   x \in \cP(\bH_+), p\in \cC(f). 
\end{eqnarray}
We denote them by $\cN_{x,y}$ and $\cN_{p,x}$. Then for $J_+\in \cJ_{reg}(\bH_+,f,g), J_-\in \cJ_{reg}(\bH_-,f,g)$, there exists generic homotopy $J_{s}$ from $J_-$ to $J_+$, so that the moduli spaces above are compact smooth manifolds with boundaries, when the expected dimension is $\le 1$. In particular, it defines a cochain map $\Theta:C(\bH_+,J_+,f)\to C(\bH_-,J_-,f)$ by counting $\cN_{x,*}$ when $x\in \cP^*(\bH_+)$ and is identity on $C_0(f)$. $\Theta$ respects the splitting into $C_0$ and $C_+$. Moreover $\Theta$ maps $C^{D^+_{i}}(\bH_+,J_+.f)\to C^{D^{-}_i}(\bH_-,J_-,f)$. We will show in \S \ref{sub:nondeg}, $\Theta$ is an isomorphism on cohomology. The following functoriality of the continuation map is also verified in \S \ref{sub:nondeg}.
\begin{proposition}\label{prop:natrualBV}
	Let $J^1_+,J^2_+\in \cJ^{D^+_i}_{reg}(\bH_+,f,g)$ and $J^1_-,J^2_- \in \cJ^{D^-_i}_{reg}(\bH_-,f,g)$, then the following diagram is commutative up to homotopy.
	$$
	\xymatrix{
	C_+^{D^+_i}(\bH_+,J^1_+) \ar[r] \ar[d] & C_+^{D^-_i}(\bH_-,J^1_-) \ar[d] \\
	C_+^{D^+_i}(\bH_+,J^2_+) \ar[r] & C_+^{D^{-}_i}(\bH_-,J^2_-)
    }	
	$$
	where the horizontal maps are continuation maps constructed above and vertical maps are continuation maps in Proposition \ref{prop:natural}.
\end{proposition}

\subsubsection{Compatibility with homotopies of $J$}\label{ssb:comp}
Assume we have $J_{+,1},J_{+,2}\in \cJ_{reg}(\bH_+,f,g)$ and $J_{-,1},J_{-,2}\in \cJ_{reg}(\bH_-,f,g)$ and regular homotopies $J_{s,+}, J_{s,-}$ from $J_{+,2}, J_{-,2}$ to $J_{+,1},J_{-,1}$ respectively, so that the continuation map in Proposition \ref{prop:natural} is defined, and are denoted by $\Phi^+,\Phi^-$. Assume $J^{\theta}_{s,t,1}$ and $J^{\theta}_{s,t,2}$ are two homotopies with ends $J_{-,1},J_{+,1}$ and $J_{-,2}, J_{+,2}$ respectively,  such that Proposition \ref{prop:delta} holds. 

Following \cite[\S2.2.3]{abouzaid2013symplectic}, we consider a family of almost complex structures $J^{r,\theta}_{s,t}$ such that 
\begin{enumerate}
	\item When $r\ll 0$, $J^{r,\theta}_{s,t}=J_{s-r,t+\theta,-}$ if $s<0$ and 
	is $J^{\theta}_{s+r,t,1}$ if $s\ge 0$. They patch up smoothly by our definition of $J_s$ and $J^{\theta}_{s,t}$
	\item  When $r\gg 0$, $J^{r,\theta}_{s,t}=J^{\theta}_{s-r,t,2}$ if $s<0$ and 
	is $J_{s+r,t,+}$ if $s\ge 0$.
	\item For every $r$, when $s\gg 0$, $J^{r,\theta}_{s,t}=J_{+,1}$, when $s \ll 0$, $J^{r,\theta}_{s,t}=J_{t+\theta, -, 2}$.
\end{enumerate}
Then for $x\in \cP^*(\bH_-), y\in \cP^*(\bH_+)$, we can consider the moduli space of solutions $(u,\theta, r)$ to 

\begin{equation}\label{eqn:BVhomotopy}
\partial_s u + J^{\theta,r}_{s,t}(\partial_t u-X_{\bH^{\theta}_{s,t}})=0, \quad \lim_{s\to -\infty}u=x(\cdot + \theta), \lim_{s\to \infty} u = y.
\end{equation}
Let $\cT_{x,y}$ denote the compactification, we can similarly consider $\cT_{p,x}$ for $x\in \cP^*(\bH_+),p\in \cC(f)$. Then for generic a choice of $J^{r,\theta}_{s,t}$, $\cT_{*,*}$ is a compact manifold with boundary if the expected dimension is $\le 1$. Let $\Delta^1,\Delta^2$ be the BV operators defined using $J^{\theta}_{s,t,1}$ and $J^{\theta}_{s,t,2}$. Let $\eta:C(\bH_+,J_{+,1},f) \to C(\bH_-,J_{-,2},f)$ be the operator counting rigid points in $\cT_{*,*}$. Then the boundary configuration leads to the following relation
$$ \Phi^-\circ \Delta^1-\Delta^2\circ \Phi^+  = d_{J_{-,2}}\circ \eta - \eta \circ d_{J_{+,1}}.$$
Moreover, both $\Phi^+, \Phi^-$ have splittings into $\Phi^{\pm}_++\Id_{C_0}+\Phi^{\pm}_{+,0}$ and $\eta$ has splitting into $\eta_{+}+\eta_{+,0}$. That is they are in the form that Proposition \ref{prop:welldefined} can be applied.

The following propositions show that the structure we defined is the same BV operator defined in \cite{seidel2012symplectic}. We will prove them in \S \ref{sub:nondeg}.
\begin{proposition}\label{prop:BV}
	Using the identification in Proposition \ref{prop:cascades} to identify $SH^*(W)$ with $H^*(C(\bH_+,J_+,f))$, then $\Theta^{-1}\circ \Delta:H^*(C(\bH_+,J_+,f)) \to H^{*-1}(C(\bH_+,J_+,f))$ is the BV operator on $SH^*(W)$.
\end{proposition}
\begin{proposition}\label{prop:inv}
	$\Theta^{-1}\circ \Delta_+:H^*(C_+(\bH_+,J_+,f)) \to H^{*-1}(C_+(\bH_+,J_+,f))$ and $\phi:\ker \Delta_+\to \coker \delta $ are invariants of exact domains up to exact symplectomorphisms using the identification in Proposition \ref{prop:cascades}.
\end{proposition}

Hence, from now on, we will proceed the construction using $\bH_-,\bH_+$. 
\begin{proposition}\label{prop:dilation}
	$W$ admits a symplectic dilation iff $1\in \Ima \phi(\Ima(SH(W)\to SH_+(W))\cap \ker \Delta_+)= \Ima \phi(\ker \delta\cap \ker \Delta_+)$.  
\end{proposition}
\begin{proof}
	If $SH^*(W) = 0$, then $\coker \delta = 0$, hence $1$ is always in the image and there exists a dilation. In the following, we will consider the case when $SH^*(W) \ne 0$.  Using Proposition \ref{prop:BV}, assume $x'\in C(\bH_+,J_+,f)$ closed satisfies $\Delta([x']) = 1$ on cohomology, where $1$ is represented by the unique local minimum of $f$. Since $x'$ can be written as $x+y$, where $x\in C_+(\bH_+,J_+),y\in C_0(f)$. Therefore we have $\Delta_+(x)+\Delta_{+,0}(x) = 1 + d(b)$ for $b \in C_+(\bH_-,J_-)$. Then we have $\Delta_+(x) = d_+(b)$, hence $\phi(x) = \Delta_{+,0}(x) -d_{+,0}(b) = 1$. Since $\Theta$ preserves $1$, we have $1$ is in the image of $\phi:\Ima(SH(W)\to SH_+(W))\cap \ker \Delta_+ \to \coker \delta$.
	
	On the other hand, if $\phi(x) = 1$ for $x\in C_+(\bH_+,J_+)$ and $[x]\in \Ima(H^*(C(\bH_+,J_+,f))\to H^*(C_+(\bH_+,J_+)))\cap \ker \Delta_+$. That is there exists $y\in C_0(f)$, such that $x+y$ is closed. We have $\Delta_{+,0}(x) - d_{+,0}(b) = 1 + c$, where $c\in C_0(f)$ is closed and $[c] \in \Ima \delta^{-}$ and $d_+(b)=\Delta_+(x)$.  Since $[c] \in \Ima \delta^{-1}$, we have $[c]$ in $H^*(C(\bH_-,J_-,f))$ is zero. Therefore we have $\Delta(x+y)=\Delta(x) = \Delta_{+,0}(x) + \Delta_+(x) = d(b)+1+c$, which is $1$ in cohomology. Therefore $x+y$ is a dilation.
\end{proof}

By composing with the restriction map $H^*(W) \to H^*(\partial W)$, we define a map 
\begin{equation}\label{eqn:delta}
\Delta_{\partial}:\ker \Delta_+\to \coker \delta_{\partial}.
\end{equation}
And this is the second structure map we are interested in. 

\subsection{Shrinking the gradient flow}
In this subsection, we will apply the same idea in \S \ref{s3} to $\Delta_{\partial}$ and rewrite it without using the Morse function $f$. As in \S \ref{s3}, we choose a Morse function $h$ on $\partial W\times \{1-\epsilon\}$ along with a metric $g_{\partial}$, so that Proposition \ref{prop:restriction} holds. On the complex level, the composition of $\ker \Delta_+ \to \coker \delta \to \coker \delta_{\partial}$ is represented  by counting the moduli space $\cR_{p,q}\times \cM^{\Delta}_{q,x}$ and $\cR_{p,q}\times\cM^-_{q,b}$ for $x\in \cP^*(\bH_+), b \in \cP^*(\bH_-), q\in \cC(f),p\in \cC(h)$. In particular, if we define $\Delta_{+,0,\partial}(x)$ by counting $\cR_{p,q}\times \cM^{\Delta}_{q,x}$ and $d_{+,0,\partial}$ by counting $\cR_{p,q}\times \cM^-_{q,b}$. Then we define a map $\Delta_+^{-1}(\Ima d_+)\cap \ker d_+ \to \coker \delta_{\partial}$ representing $\Delta_{\partial}$ by 
\begin{equation}\label{eqn:phib}
x  \mapsto \Delta_{+,0,\partial}(x) - d_{+,0,\partial}(b),
\end{equation}
for $d_+(b)=\Delta_+(x)$. To bypass the middle $C^{*-1}(f)$, we define $\cP^\Delta_{p,x}$ to be the compactification of the following moduli space
$$\left\{\begin{array}{l}u:\C \to \widehat{W},\\
\theta\in S^1,\\
\gamma:(-\infty,0]\to \partial W\times \{1-\epsilon\}. \end{array}\left| \begin{array}{l}\partial_s u + J^\theta_{s,t}(\partial_t u -X_{\bH^\theta_{s,t}})=0, \frac{\rd}{\rd s} \gamma + \nabla_{g_\partial} h=0,\\
\displaystyle\lim_{s\to -\infty}\gamma =p, u(0)=\gamma(0), \lim_{s\to \infty} u = x
\end{array}\right.\right\}, \quad p \in \cC(h), x \in \cP^*(\bH_+).
$$
The space makes sense, since we push $\partial W$ into interior, where $\bH^{\theta}_{s,t}=0$. Then $\cP^\Delta_{p,x}$ defines a map from $P^{\Delta}:C^*_+(\bH_+,J_+)) \to C^{*-1}(h)$. To show it defines the same thing as \eqref{eqn:phib}, we define $\cH^{\Delta}_{p,x}$ to be the compactification of the following moduli space for $p \in \cC(h), x \in \cP^*(\bH_+)$,
$$\left\{\left.\begin{array}{l}u:\C \to \widehat{W},\\ \theta \in S^1, \\ l\in \R_+, \\\gamma_1:(-\infty,0]\to \partial W \times \{1-\epsilon\},\\ \gamma_2:[0,l] \to W. \end{array}\right| \begin{array}{l}\partial_s u + J^{\theta}_{s,t}(\partial_t u -X_{\bH^{\theta}_{s,t}})=0, \\
\frac{\rd}{\rd s} \gamma_1 + \nabla_{g_{\partial}} h=0, \frac{\rd}{\rd s} \gamma_2 + \nabla_{g} f=0,\\
\displaystyle \lim_{s\to -\infty}\gamma_1 =p,\gamma_1(0) = \gamma_2(0), \\
\gamma_2(l)=u(0), \lim_{s\to \infty} u = x.
\end{array}\right\}.
$$
Then we have following regularity property.
\begin{proposition}\label{prop:boundary2}
	For generic choices of $J_-,J_+$ and $J^{\theta}_{s,t}$, for any $x,p$ with $|p|-|x|\le 0$, $\cH^{\Delta}_{p,x}$ is a manifold with boundary of dimension $|p|-|x|+1$ and	
	$$\partial \cH^{\Delta}_{p,x} = -\cP^{\Delta}_{p,x}+\sum_{q\in \cC(f)}\cR_{p,q}\times \cM^{\Delta}_{q,x} + \sum_{y\in \cP^*(\bH_-)} \cH_{p,y}\times \cM^{\Delta}_{y,x} + \sum_{q\in \cC(h)} \cM_{p,q}\times \cH^{\Delta}_{q,x} -\sum_{y\in \cP^*(\bH_+)} \cH^{\Delta}_{p,y}\times \cM_{y,x}.$$
\end{proposition}
\begin{proposition}\label{prop:equal}
	For generic choices of $J_-,J_+$ and $J^{\theta}_{s,t}$ as in Proposition \ref{prop:boundary2}. Let $x\in \ker d_+ \subset C_+^*(\bH_+,J_+)$ such that $\Delta_+(x)=d_+(b)$, then $P^{\Delta}(x)-P(b)$ represents the same class in $\coker (H^{*-2}(C_+(\bH_-,J_-))\to H^{*-1}(\partial W))$ as \eqref{eqn:phib}. Hence on cohomology, $\Delta_{\partial}$  in \eqref{eqn:phib} equals to the difference of counting $\cP^{\Delta}_{x,q}$, $\cP_{b,q}$.
\end{proposition}
\begin{proof}
	By Proposition \ref{prop:boundary1} and \ref{prop:boundary2},  we have $	P^{\Delta}(x)-P(b)-\Delta_{+,0,\partial}(x) + d_{+,0,\partial}(b)$ is the following,
	$$\sum(\sum \# \cH_{p,y} \# \cM^{\Delta}_{y,x}+\sum \#\cM_{p,q}\#\cH^{\Delta}_{q,x}-\sum \#\cH^{\Delta}_{p,y}\# \cM_{y,x} - \sum \# \cH_{p,y} \# \cM_{y,b}-\sum \#\cM_{p,q}\# \cH_{q,b})p.$$
	It is $H\circ \Delta_+(x)+d_{\partial W}\circ H^{\Delta}(x)-H^{\Delta}\circ d_+(x)- H\circ d_+(b)-d_{\partial W} \circ H(b)$, where $H,H^{\Delta}$ are defined by counting $\cH$ and $\cH^{\Delta}$ and $d_{\partial W}$ is the Morse differential on $\partial W$. Since $d_+(x)=0$ and $\Delta_+(x)=d_+(b)$, we have the above term is $d_{\partial W}\circ H^{\Delta} (x)-d_{\partial W}\circ H(b)$, which is exact. Hence the claim is proven.
\end{proof}

\subsection{Naturality}
To keep track of the regularity of almost complex structures, we introduce the following.
\begin{enumerate}
	\item $\cJ^{\le 1}_{reg,\Delta}(f,g,J_-,J_+)\subset \cJ_{s,\theta,J_-,J_+}(W)$ with ends $J_-\in \cJ_{reg}(\bH_-,f,g), J_+ \in \cJ_{reg}(\bH_+,f,g)$ is set of regular almost complex structures for moduli spaces $\cM^{\Delta}$ of dimension up to $1$. $\cJ^{D^+_i}_{reg,\Delta,J_-,J_+}(f,g) \subset \cJ_{s,\theta,J_-,J_+}(W)$ with ends $J_-\in \cJ^{D^-_i}_{reg}(\bH_-f,g), J_+\in \cJ^{D^+_i}_{reg}(\bH_+,f,g)$ is the set of regular almost complex structures for $\cM^{\Delta}_{*,x}$ of dimension up to $0$ and action of $x$ down to $-{D^+_i}$. Similarly for $\Delta_+^{D_i}$. 
	\item $\cJ^{\le 1}_{reg,P^\Delta}(h,g_{\partial},J_-,J_+)$ with ends $J_-\in \cJ_{reg,+}(\bH_-)\cap \cJ_{reg,P}(\bH_-,h,g_{\partial})$ and $J_+\in \cJ_{reg,+}(\bH_+)$ for moduli space $\cP^{\Delta}_{p,x}$ of dimension up to $1$.  $\cJ^{D^{+}_i}_{reg,P^{\Delta}}(h,g_{\partial},J_-,J_+)$ is the set of regular admissible homotopy $J^{\theta}_{s,t}$ with ends $J_-\in \cJ^{D^{-}_i}_{reg,+}(\bH_-)\cap \cJ^{D^{-}_i}_{reg,P}(\bH_-,h,g_{\partial}), J_+\in \cJ^{D^+_i}_{reg,+}(\bH_+)$ for moduli spaces $\cP^{\Delta}_{p,x}$ of dimension up to $0$ and action of $x$ down to $-D^{+}_i$.  
	\item $\cJ^{\le 1}_{reg,H^{\Delta}}(f,g,h,g_{\partial},J_-,J_+)\subset \cJ_{s,\theta,J_-,J_+}(W)$ with ends $J_-\in \cJ_{reg,+}(\bH_-)\cap \cJ_{reg,H}(\bH_-,f,g,h,g_{\partial})\cap \cJ_{reg,P}(\bH_-,h,g_{\partial}), J_+\in \cJ_{reg,+}(\bH_+)$ is the set of regular $J^{\theta}_{s,t}$ for moduli spaces $\cH^{\Delta}_{p,x}$ of dimension up to $1$.
	\item $\cJ^{\le 1}_{reg,N,+}(J_-,J_+)\subset \cJ_{s,\theta,J_-,J_+}(W)$ with ends $J_-\in \cJ_{reg,+}(\bH_-), J_+\in \cJ_{reg,+}(\bH_+)$ is set of regular $J_s$ for moduli spaces $\cN_{x,y}$ in \S \ref{ssb:cont} of dimension up to $1$.  $\cJ^{D^+_i}_{reg,N,+}(J_-,J_+)\subset \cJ_{s,\theta,J_-,J_+}(W)$ with ends $J_-\in \cJ^{D^-_i}_{reg,+}(\bH_-), J_+\in \cJ^{D^+_i}_{reg,+}(\bH_+)$ is set of regular $J_s$ for moduli spaces $\cN_{x,y}$ up to dimension $0$ with action of $y$ down to $-D^+_i$.
\end{enumerate}
As usual, $\cJ_{reg}$ is of second category, and $\cJ^D_{reg}$ is open dense.
\begin{proposition}\label{prop:BVnatural}
	We summarize naturality properties in the following.
	\begin{enumerate}
		\item\label{bvns1} For $J^{\theta}_{s,t}\in \cJ^{D^+_i}_{reg,\Delta}(f,g,J_-,J_+)$ with $J_-\in \cJ^{D_i^-}_{reg}(\bH_-,f,g), J_+\in \cJ^{D_i^+}_{reg}(\bH_+,f,g)$, $\Delta^{D^+_i}$ is a well-defined cochain map.  Similarly for $\Delta^{D^+_i}_+$.
		\item\label{bvns2} For $J^{\theta}_{s,t} \in \cJ^{D^+_i}_{reg,\Delta,+}(J_-,J_+)$ with $J_-\in \cJ^{D^-_i}_{reg,+}(\bH_-), J_+\in \cJ^{D^+_i}_{reg,+}(\bH_+)$, $\phi$ is well-defined on $\ker \Delta^{D^+_i}_+$.
		\item\label{bvns3} For $J^{\theta}_{s,t}\in \cJ^{D^+_i}_{reg,\Delta,+}(J_-,J_+)\cap \cJ^{D^+_i}_{reg,P^{\Delta}}(h,g_{\partial},J_-,J_+)$ for $J_-\in \cJ^{D^-_i}_{reg,+}(\bH_-)\cap \cJ^{D^{-}_i}_{reg,P}(\bH_-,h,g_{\partial}), J_+\in \cJ^{D^+_i}_{reg,+}(\bH_+)$, $\Delta_{\partial}$ is same as counting the difference between  $\cP^{\Delta}$ and $\cP$, when restricted to $C^{D_i^+}_+$.
	\end{enumerate} 
\end{proposition}
\begin{proof}
	For \eqref{bvns1}, there exists a neighborhood $\cU$ of $J^{\theta}_{s,t}\in \cJ_{s,\theta,J_-,J_+}(W)$ contained in $\cJ^{D^{+}_i}_{reg,\Delta,J_-,J_+}(f,g)$. Then $\Delta^{D^+_i}_+$ is locally constant by the proof of Proposition \ref{prop:lowreg}. The claim follows from that $\Delta^{D^+_i}_+$ is a cochain map for more regular $J^{\theta}_{s,t}$ nearby.
	
	For \eqref{bvns2}, we have $\Delta_+$ and $d_{+}$ are locally constant with respect to the almost complex structures. Since $d_{+,0}$ is defined using two nearby $J'',J'\in \cJ^{D^-_i}_{reg}(\bH_-,f,g)$ of $J_-$. Then there are two homotopy $J^{\theta,''}_{s,t}\in \cJ^{\le 1}_{reg,\Delta}(f,g,J'',J_+)$, $J^{\theta,'}_{s,t}\in \cJ^{\le 1}_{reg,\Delta}(f,g,J',J_+)$ close to $J^{\theta}_{s,t}$ in $\cJ_{s,\theta}(W)$. It is sufficient to prove for $x\in C^{D^+_i}_+(\bH_+)$ that $d_+(x)=0,\Delta_+(x)=d_+(b)$ for $b\in C^{D^-_i}_+(\bH_-)$, we have $\Delta_{+,0,J',J_+}(x)-d_{+,0,J'}(b)$ and $\Delta_{+,0,J'',J_+}(x)-d_{+,0,J''}(b)$ are differed by exact form. Since we have a continuation map $\Phi=\Phi_++\Phi_{+,0}+\Id_{C_0(f)}:C^{D^-_i}(\bH_-,J',f)\to C^{D^{-}_i}(\bH_-,J'',f)$, such that $\Phi_+$ is identity by the proof of Lemma \ref{lemma:local}. Then we have
	$$d_{+,0,J''}(b)-d_{+,0,J'}(b)=d_0\circ \Phi_{+,0}(b)-\Phi_{+,0}\circ d_{+}(b)=d_0\circ \Phi_{+,0}(b)-\Phi_{+,0}\circ \Delta_{+}(x)$$
	On the other hand,  by \S \ref{ssb:comp},  we have
	$$\Delta_{+,0,J'',J_+}(x)+\Phi_{+,0}\circ \Delta_{+}(x)-\Delta_{+,0,J',J_+}(x)=d_0\circ \eta_{+,0}(x)+d_{+,0}\circ \eta_+(x).$$
	We claim $\eta_+$ could be chosen to be zero. By assumption, the almost complex structure $J^{\theta,r}_{s,t}$ in the construction of $\eta$ can be chosen such that for every $r$, $J^{\theta,r}_{s,t}\in \cJ^{D^+_i}_{reg,\Delta,+}(J_-,J_+)$.  As a consequence, $\eta_+ = 0$. Combining them together proves the claim.
	
	For \eqref{bvns3}, $P,P^{\Delta},\Delta_+,d_+$ are all locally constant, for a nearby $J^{\theta}_{s,t}$, we have (the perturbed) $\phi$ defines the same thing as counting $\cP,\cP^{\Delta}$ by Proposition \ref{prop:equal}. By claim \eqref{bvns2}, $\phi$ is independent of the perturbation. This finish the proof.
\end{proof}

Combining Proposition \ref{prop:natrualBV}, Proposition \ref{prop:inv} and Proposition \ref{prop:BVnatural} and the compatibility of continuation maps with $\Delta$ in \S \ref{ssb:comp}, we have the following.
\begin{proposition}\label{prop:BVnatural2}
	Let $J^i_-\in \cJ^{D^-_i}_{reg,+}(\bH_-)\cap \cJ^{D^-_i}_{reg,P}(\bH_-,h,g_{\partial}), J^i_+\in \cJ^{D^+_+}_{reg,+}(\bH_+)$, $J^{\theta,i}_{s,t}\in \cJ^{D_i^+}_{reg,\Delta,+}(J^i_-,J^i_+)\cap J^{D^+_i}_{reg,P^{\Delta}}(h,f_{\partial}, J^i_-.J^i_+)$ and $J^{i}_{s,t} \in \cJ^{D_i^+}_{reg,N,+}(J^i_-,J^i_+)$
	then we have the following 
	$$
	\xymatrix@C=0.3cm{
	H^*(C^{D^+_1}_+(\bH_+,J_+^1)) \ar[r] \ar[d]^{\Delta_+} &  H^*(C^{D^+_2}_+(\bH_+,J_+^2)) \ar[r]\ar[d]^{\Delta_+} & \ldots & 	H^*(C^{D^+_1}_+(\bH_+,J_+^1)) \ar[r] \ar[d]^{\Theta_+} &  H^*(C^{D^+_2}_+(\bH_+,J_+^2)) \ar[r]\ar[d]^{\Theta_+} & \ldots\\
	H^{*-1}(C^{D^-_1}_+(\bH_-,J_-^1)) \ar[r]  &  H^{*-1}(C^{D^-_2}_+(\bH_-,J_-^2)) \ar[r] &\ldots & H^{*}(C^{D^-_1}_+(\bH_-,J_-^1)) \ar[r]  &  H^{*}(C^{D^-_2}_+(\bH_-,J_-^2)) \ar[r] &\ldots }	
	$$
	and $ \varinjlim H^*(C^{D^+_i}_+(\bH_+,J_+^i))\xrightarrow{\varinjlim \Delta_+}\varinjlim H^*(C^{D^-_i}_+(\bH_-,J-^i)) \xrightarrow{(\varinjlim \Theta_+)^{-1}} \varinjlim H^*(C^{D^+_i}_+(\bH_+,J_+^i))$ computes $\Delta_+:SH^*_+(W)\to SH^{*-1}_+(W)$. Similarly for $\Delta_{\partial}$, which can either be computed using $\Delta_{+,0},d_{+,0}$ or $P^{\Delta},P$.
\end{proposition}


\subsection{Independence}
The following statement follows from the same proof of Proposition \ref{prop:ind}.
\begin{proposition}\label{prop:ind2}
	With the same setup in Proposition \ref{prop:ind}, there exist admissible $J^1_{\pm,1},J^2_{\pm,1},\ldots$, $J^1_{\pm,2},J^2_{\pm,2},\ldots$ on $\widehat{W_1}$ and $\widehat{W_2}$ respectively and admissible homotopies $J^1_{s,1},J^2_{s,1},\ldots,J^1_{s,2},J^2_{s,2},\ldots$ with ends $J^i_{\pm,*}$ and admissible $S^1$ families of homotopies $J^{\theta,1}_{s,t,1},J^{\theta,2}_{s,t,1},\ldots$,$J^{\theta,1}_{s,t,2},J^{\theta,2}_{s,t,2},\ldots$ with ends $J^i_{\pm,*}$ and positive real number $\epsilon_1, \epsilon_2,\ldots$ such that the following holds.
	\begin{enumerate}
		\item For $R<\epsilon_i$ and any $R'$, we have $$NS_{i,R}(J^i_{\pm,*}), NS_{i+1,R'}(NS_{i,R}(J^i_{\pm.*}))\in \cJ^{D^{\pm}_i}_{reg,+}(\bH_{\pm})\cap \cJ^{D^{\pm}_i}_{reg,P}(\bH_{\pm},h,g_{\partial}),$$
		$$NS_{i,R}(J^{\theta,i}_{s,t,*}), NS_{i+1,R'}(NS_{i,R}(J^{\theta,i}_{s,t,*}))\in  \cJ^{D^{+}_i}_{reg,P^{\Delta}}(h,g_{\partial})\cap J^{D_i^+}_{reg, \Delta,+},$$
		$$NS_{i,R}(J^{i}_{s,*}), NS_{i+1,R'}(NS_{i,R}(J^{i}_{s,*}))\in  \cJ^{D^{+}_i}_{reg,N,+},$$
		such that all zero dimensional $\cM_{x,y}$,$\cP_{p,x}, \cM^{\Delta}_{x,y}$ and $\cP^{\Delta}_{p,x}$, $\cN_{x,y}$ are the same for both $W_1,W_2$ and contained outside $Y_i$ for $x,y\in C^{D^{\pm}_i}_+,p\in C(h)$. 
		\item $J^{i+1}_{s,*}=NS_{i,\frac{\epsilon_i}{2}}(J^i_{s,*})$ on $W^i_*$. $J^{\theta,i+1}_{s,t,*}=NS_{i,\frac{\epsilon_i}{2}}(J^{\theta,i}_{s,t,*})$ on $W^i_*$. Note that they imply that $J^{i+1}_{\pm.*}=NS_{i,\frac{\epsilon_i}{2}}(J^i_{\pm,*})$ on $W^i_*$.
	\end{enumerate}
\end{proposition}

Now, we are ready to prove Theorem \ref{thm:E}.

\begin{proof}[Proof of Theorem \ref{thm:E}]
	By Proposition \ref{prop:ind2}, we have the following two diagrams for both $W_1,W_2$.
	$$\xymatrix{
		H^*(C^{D^+_1}_+(\bH_+,NS_{1,\frac{\epsilon_1}{2}}J^1_{+,*})) \ar[r] \ar[d]^{\Delta_+} &  H^*(C^{D^+_2}_+(\bH_+,NS_{2,\frac{\epsilon_2}{2}}J^2_{+,*})) \ar[r]\ar[d]^{\Delta_+} & \ldots \\
		H^{*-1}(C^{D^-_1}_+(\bH_-,NS_{1,\frac{\epsilon_1}{2}}J^1_{-,*})) \ar[r]  & H^{*-1}(C^{D^-_2}_+(\bH_-,NS_{2,\frac{\epsilon_2}{2}}J^2_{-,*})) \ar[r] &\ldots}	
	$$
	$$\xymatrix{
		H^*(C^{D^+_1}_+(\bH_+,NS_{1,\frac{\epsilon_1}{2}}J^1_{+,*})) \ar[r] \ar[d]^{\Theta_+} &  H^*(C^{D^+_2}_+(\bH_+,NS_{2,\frac{\epsilon_2}{2}}J^2_{+,*})) \ar[r]\ar[d]^{\Theta_+} & \ldots \\
		H^{*}(C^{D^-_1}_+(\bH_-,NS_{1,\frac{\epsilon_1}{2}}J^1_{-,*})) \ar[r]  & H^{*}(C^{D^-_2}_+(\bH_-,NS_{2,\frac{\epsilon_2}{2}}J^2_{-,*})) \ar[r] &\ldots}	
	$$
	where $\Theta_+$ is continuation defined using $NS_{i,\frac{\epsilon_i}{2}}(J^i_{s,*})$ and $\Delta_+$ is the positive BV operator defined by $NS_{i,\frac{\epsilon_i}{2}}(J^{\theta,i}_{s,*})$. In particular, they are identified on $W^1,W^2$. The horizontal arrows are continuation maps. As proved in \ref{thm:A}, they are represented by inclusions hence are also independent of fillings. Moreover, $P,P^{\Delta}$ on them are independent of the filling. That is the diagrams can be identified for $W_1,W_2$. Proposition \ref{prop:BVnatural2} implies the claim. Note that the diagram contain the part used in the proof of Theorem \ref{thm:A}. Hence the proof gives a (non-conical) isomorphism $\Gamma$, identifying $\delta_{\partial},\Delta_+$ and $\Delta_{\partial}$.
\end{proof}

\begin{remark}
	It is a more delicate question to give a ``natural" isomorphism in Theorem \ref{thm:A} and Theorem \ref{thm:E}. Nevertheless, the isomorphism can be chosen so that match up all the structures. 
\end{remark}

The following corollary will imply Corollary \ref{cor:F} in the introduction.
\begin{corollary}\label{cor:dilation}
	Let $Y$ be an ADC manifold with a topologically simple filling admitting a symplectic dilation, then for any other topologically simple filling $W'$ such that $H^2(W')\to H^2(Y)$ is injective, we have $W'$ also admits a dilation. In particular, it holds for $W'$ being Weinstein of dimension at least $6$.
\end{corollary}

\begin{proof}
	We assume $SH^*(W)\ne 0$, for otherwise, it is proven by Corollary \ref{cor:B}. By Proposition \ref{prop:dilation}, $W$ admits a dilation implies that there exists $x\in \ker \delta\cap \ker \Delta_+ \subset SH^*_+(W)$, such that $\phi(x) = 1$. In particular $\Delta_{\partial}(x) = 1$. By Theorem \ref{thm:E}, we have an identification and $\Delta_{\partial}(x) = 1$ for $W'$. In particular $\phi(x)=1$ on $W'$, because in degree $0$, we have both $\coker \delta_{\partial}$ and $\coker \delta$ spanned by $1$ for both $W$ and $W'$. We only need to prove $x$ is from $SH^*(W')$. Since $\delta_{\partial}(x)=0$ on $W$, we have $\delta_{\partial}(x)=0$ on $W'$. Since $H^2(W')\to H^2(Y)$ is injective, then $\delta(x) = 0$ on $W'$. Then by Proposition \ref{prop:dilation}, we have that $W'$ admits a dilation.
\end{proof}

The existence of symplectic dilations puts strong restrictions on the symplectic topology. Let $W^{2n}$ be a Liouville domain with $c_1(W)=0$ and $n>1$, assume $W$ admits a dilation, Seidel \cite{seidel2014disjoinable} showed that there are at most finitely many Lagrangian spheres in $W^{2n}$ that are pairwise disjoint. Given Corollary \ref{cor:dilation}, we can ask the following natural question.
\begin{question}
	Let $W$ be a ADC manifold with a symplectic dilation, is there a universal upper bound for the maximal collections of pairwise disjoint Lagrangian spheres for all Weinstein fillings of $\partial W$?
\end{question}

Similar to Corollary \ref{cor:ob}, we have following obstruction to Weinstein fillings.
\begin{corollary}\label{cor:ob2}
	Let $Y$ be a $2n-1$ dimensional ADC contact manifold and $n\ge 3$. If $Y$ admits a topologically simple exact filling $W$ and $\Ima \Delta_{\partial}$ contains an element of grading greater than $n$, then $Y$ does not admit Weinstein filling.
\end{corollary}
\begin{proof}
	Since $\Delta_{\partial}$ factors through $\coker \delta = \coker (SH^*_+(W)\to H^{*+1}(W))$, $\Ima \Delta_{\partial}$ can not contain an element of grading greater than $n$ if $W$ is Weinstein. By Theorem \ref{thm:E}, $\Delta_{\partial}$ is independent of topologically simple exact fillings, in particular Weinstein fillings, the claim follows. 
\end{proof}

We apply Corollary \ref{cor:ob} and Corollary \ref{cor:ob2} in \S \ref{s6} to construct many exactly fillable, but not Weinstein fillable manifolds. Although Corollary \ref{cor:ob}  and Corollary \ref{cor:ob2} use the topology of $W$, those obstructions as contact invariants of the ADC boundary is not topological.  In particular, we will show that they are very different from the topological criterion in \cite{bowden2014topology} by proving Theorem \ref{thm:I}. 

\subsection{$\Delta_+,\phi$, $\Delta_{\partial}$ as invariants}\label{sub:nondeg} In this subsection, we will explain that $\Delta_+$ and $\phi$ are invariants of the exact domains up exact symplectomorphisms. It is clear that $\Delta_+,\phi$ can be defined on Hamiltonians with finite slope and $C^2$-small on $W$. Using Proposition \ref{prop:welldefined}, they commute with continuation maps, hence the direct limit of them define $\Delta_+:SH^*_+(W)\to SH^{*-1}_+(W)$ and $\phi:\ker \Delta_+ \to \coker \delta$. The following proposition implies that $\Delta_+,\phi$ are invariants of the exact domains. 

\begin{proposition}\label{prop:viterbo}
	Let $V\subset W$ be a subdomain, then Viterbo transfer map preverse $\Delta_+$ and $\phi$. In particular, we have the following commutative diagram,
	$$\xymatrix{\ker \Delta_+\subset SH^*_+(W) \ar[r] \ar[d] & \coker \delta_W \ar[d]\\
		\ker \Delta_+\subset SH^*_+(V) \ar[r]& \coker \delta_V,}
	$$
	where $\coker \delta_W\to \coker \delta_V$ is induced by 	$$\xymatrix{SH^*_+(W) \ar[r] \ar[d] & H^{*+1}(W) \ar[d]\\
		SH^*_+(V) \ar[r]& H^{*+1}(V).}
	$$
\end{proposition}
Since the BV operator is the second term in the differential of the $S^1$-equivariant symplectic (co)chain complexes \cite{zhao2014periodic}, this proposition follows from the functoriality of positive $S^1$-equivariant symplectic cohomology, which was proven in \cite{gutt2017positive}. In fact, it is sufficient to consider the approximation $ES^1$ by $S^3$ in the $S^1$-equivariant symplectic (co)homology. Using the fact that Viterbo transfer $SH^*_+(W\cup \partial W \times (1,r))\to SH^*_+(W)$ is an isomorphism and exactly symplectomorphic exact domains have nested embeddings into each other \cite[Propsition 11.8]{cieliebak2012stein}, Proposition \ref{prop:viterbo} implies that $\Delta_+$ and $\phi$ are invariants of exact domains up to exact symplectomorphisms.

Next we verify that $\Delta_+,\phi$ defined on $H^*(C(\bH_+,J_+))$ match with $\Delta_+,\phi$ on $SH^*_+(W)$ naturally.
\begin{proof}[Proof of Proposition \ref{prop:natrualBV},\ref{prop:BV}, \ref{prop:inv}]
We first prove Proposition \ref{prop:BV} and \ref{prop:inv}. We need to relate the structure maps defined using $\bH_{\pm}$ to those defined as limits of non-degenerate linear Hamiltonians. Let $H_{\pm}$ be the non-degenerate perturbation to $\bH_{\pm}$ in the proof of Proposition \ref{prop:cascades}. Let $H_{\pm,s}$ be the decreasing homotopies between them. Then $\Delta,\Delta_+,\Phi$ are defined similarly using $H_{\pm}$ as the integrated maximum principle holds for the moduli spaces. We pick two sequences of Hamiltonians with finite slope $H^{D^{\pm}_i}_{\pm}$ such that the following holds.
\begin{enumerate}
	\item $H^{D^{\pm}_i}_{\pm}=H_{\pm}$ on $W^i$.
	\item $H^{D^{\pm}_i}_{\pm}=f'_{\pm}(c_i)r$ for $r\ge c_i$, where $c_i\in (a_i,b_i)$.
	\item $H^{D^{\pm}_i}_{\pm}\le H^{D^{\pm}_{i+1}}_{\pm}$.
	\item $H^{D^+_i}_+\le H^{D^-_i}_-$. 
\end{enumerate}
Then by the continuation maps on finite slope Hamiltonians and the compatibility with $\Delta$ following \cite[\S 2.2.3]{abouzaid2013symplectic}, we have the following commutative diagram.
$$
\xymatrix{
	H^*(C^{D^+_1}_{(+)}(H_+)) \ar[r]\ar@/_2.0pc/@[red][ddd]_{\Delta_{(+)}} & H^*(C^{D^+_2}_{(+)}(H_+)) \ar[r]\ar@/_2.0pc/@[red][ddd]_{\Delta_{(+)}} & \ldots \ar[r] & H^*(C_{(+)}(H_+))\ar@/^2.0pc/@[red][ddd]^{\Delta_{(+)}} \\
	H^{*}(C_{(+)}(H^{D^+_1}_+)) \ar[r]\ar[u]\ar[d]^{\Delta_{(+)}} & H^{*}(C_{(+)}(H^{D^+_2}_+)) \ar[r]\ar[u]\ar[d]^{\Delta_{(+)}} & \ldots \ar[r] & \varinjlim  H^{*}(C_{(+)}(H^{D^+_i}_+))\ar[u]\ar[d]_{\Delta_{(+)}} \\
	H^{*-1}(C_{(+)}(H^{D^-_1}_-)) \ar[r]\ar[d] & H^{*-1}(C_{(+)}(H^{D^-_2}_-)) \ar[r]\ar[d] & \ldots \ar[r] & \varinjlim  H^{*-1}(C_{(+)}(H^{D^-_i}_-))\ar[d] \\	
	H^{*-1}(C^{D^-_1}_{(+)}(H_-)) \ar[r] & H^{*-1}(C^{D^-_2}_{(+)}(H_-)) \ar[r] & \ldots \ar[r] & H^{*-1}(C_{(+)}(H_-))
	}
$$
the unmarked arrows are continuation maps, and the horizontal arrows on top and bottom row are inclusions, which are continuation maps induced by the trivial homotopy. By Proposition \ref{prop:welldefined}, the digram also induces an commutative diagram of $\phi$ for different pairs of Hamiltonians $H_{\pm}, H^{D^\pm_i}$.  This proves Proposition \ref{prop:BV} and \ref{prop:inv} for $H_{\pm}$.

Therefore to prove the claim, it is sufficient to prove commutativity of following two diagrams (and the $+$ version). 
$$
\xymatrix{
	H^*(C(H_+)) \ar[r]\ar[d] & H^*(C(\bH_+))\ar[d]^{\Theta} & 	H^*(C(H_+)) \ar[r]\ar[d]^{\Delta} & H^*(C(\bH_+))\ar[d]^{\Delta}\\
	H^*(C(H_-)) \ar[r] & H^*(C(\bH_-)) & H^{*-1}(C(H_-)) \ar[r] & H^{*-1}(C(\bH_-))
	}
$$
where the unmarked arrows are continuation maps, and the horizontal ones are those in Proposition \ref{prop:cascades}. They can be shown using a homotopy argument. The only new thing we need to verify is compactness and regularity for moduli spaces, since it involves degenerate Hamiltonians $\bH$, the moduli spaces have some cascades part. For the first diagram, let $\bH_s=\rho(s)\bH_-+(1-\rho(s))\bH_+, H_s=\rho(s)H_-+(1-\rho(s))H_+$, they are used to define continuation maps $\Theta$ and $H^*(C(H_+))\to H^*(C(H_-))$. Then we define a smooth homotopy of homotopy $H^a_{s,t}$ in the following. Let $\chi$ be increasing function such that $\chi(a)=a$ for $a\gg 0$ and $\chi(a)=2$ near $a=1$.
\begin{enumerate}
	\item For $a>1$, $H^a_{s,t}=H_{+,s-\chi(a)}$ for $s>0$ and $H^a_{s,t}=\bH_{s+\chi(a)}$ for $s\le 0$.
	\item For $a<0$, $H^a_{s,t}=H_{s-\chi(1-a)}$ for $s>0$ and $H^a_{s,t}=H_{-,s+\chi(1-a)}$ for $s\le 0$.
	\item For $0\le a \le 1$, $H^{a}_{s,t}=\rho(a)H^0_{s,t}+(1-\rho(a))H^1_{s,t}$.
\end{enumerate}
It is clear that we have $\partial_s H^a_{s,t}\le 0$. Moreover, on $(a_i,b_i)$, we have the following three cases.
\begin{enumerate}
	\item For $a>1$, $H^{a}_{s,t}=f_+(r)$ when $s>0$ and $H^a_{s,t}=\rho(s+\chi(a))f_-(r)+(1-\rho(s+\chi(a)))f_+(r)$ when $s\le 0$. Then $\partial_s(r\partial_rH^a_{s,t}-H^a_{s,t})\le 0$.
	\item For $a<0$, $H^{a}_{s,t}=f_-(r)$ when $s\le 0$ and $H^a_{s,t}=\rho(s-\chi(1-a))f_-(r)+(1-\rho(s-\chi(1-a)))f_+(r)$ when $s> 0$. Then $\partial_s(r\partial_rH^a_{s,t}-H^a_{s,t})\le 0$.
	\item For $0\le a \le 1$, $H^{a}_{s,t}=\rho(a)(\rho(s-2)f_-(r)+(1-\rho(s-2))f_
	+(r))+(1-\rho(a))(\rho(s+2)f_-(r)+(1-\rho(s+2))f_+(r))$, then 
	$\partial_s(r\partial_rH^a_{s,t}-H^a_{s,t})=(\rho(a)\rho'(s-2)+(1-\rho(a))\rho'(s+2))(rf'_-(r)-f_-(r)-rf'_+(r)-f_+(r))\le 0$.
\end{enumerate}
In particular, Lemma \ref{lemma:max} can be applied to get compactness. The regularity is standard, except when $a\gg 0$ and breaks at $W$. In principle, it is a Morse-Bott breaking at $W$. However, in our special case, for $a\gg 0$, solutions near such breaking is isomorphic to solutions to the equation using $H_{+,s}$ shifted by $\chi(a)$. Hence such type of breaking is again a boundary corresponding to composing with $\Theta$ on $C_0(\bH_+)$, which is identity. The proof for compatibility with $\Delta$ is the same. Since all the maps respect the $0,+$ splitting, the proof above also yields the identification on $\Delta_+$ and $\phi$ by Proposition \ref{prop:welldefined}. 

To prove Proposition \ref{prop:natrualBV}, we can view $H_{\pm,s}\equiv \bH_{\pm}$. Then the above construction yields a homotopy of homotopy of Hamiltonians $H^{a}_{s,t}$, such that Lemma \ref{lemma:max} can be applied. Then Proposition \ref{prop:natrualBV} follows from the standard homotopy argument.
\end{proof}

\section{Uniruledness}\label{s5}
A variety is uniruled iff it is covered by a family of rational curves. Similarly, an affine variety is uniruled if it is covered by a family of rational curves possibly with punctures. In the symplectic setup, one can replace rational curve by pseudo-holomorphic rational curves. The uniruledness for Liouville domains was introduced in \cite{mclean2014symplectic}.

\begin{definition}[{\cite[\S 2]{mclean2014symplectic}}]\label{def:convexgeneral}
	A $\rd\lambda$-compatible almost complex structure $J$ on $W$ is convex iff there is a function $\phi$ such that
	\begin{enumerate}
		\item $\phi$ attains its maximum on $\partial W$ and $\partial W$ is a regular level set,
		\item $\lambda \circ J=\rd \phi$ near $\partial W$.
	\end{enumerate}
\end{definition}
This is more general than the cylindrical convex almost complex structure used in Definition \ref{def:convex}, where $\phi=r$ near the boundary. A maximal principle still holds for holomorphic curves near $\partial W$ using the function $\phi$. Proposition \ref{prop:extension} relates a general convex almost complex structure with a cylindrical convex almost complex structure.

\begin{definition}[{\cite[Definition 2.2]{mclean2014symplectic}}]\label{def:unirule}
	Let $k>0$ be an integer and $\Lambda>0$ a real number. We say that a Liouville domain $(W,\lambda)$ is $(k,\Lambda)$-uniruled if, for every convex almost complex structure $J$ on $W$ and every point $p \in  W^0$ where $J$ is integrable on a neighborhood of $p$, there is a proper $J$-holomorphic map $u: S \to W^0$ to the interior $W^0$ of $W$ passing through
	this point. We require that $S$ is a genus $0$ Riemann surface, the rank of $H_1(S;\Q)$ is at most $k-1$, and the area of $u$ is at most $\Lambda$.
\end{definition}

\begin{proposition}\label{prop:extension}
	Let $Y$ be a contact manifold with a contact form $\alpha$, and there is a function $\phi$ on $Y\times [1,3]_r$, such that $\partial_r \phi >0$ and $Y\times \{3\}$ is a level set. Assume $J$ is $\rd(r\alpha)$ compatible almost complex structure, such that $r\alpha \circ J = \rd \phi$ on $Y\times [1,2]$. Then there exists an extension of $\rd(r\alpha)$ compatible almost complex structure $\widetilde{J}$, such that  $r\alpha \circ \widetilde{J} = \rd \phi$ on $Y\times [1,3]$.
\end{proposition}
\begin{proof}
	The Liouville vector field on $Y\times [1,3]$ is $r\partial r$, since $\partial_r \phi >0$, every level set of $\phi$ is of contact type. On each level surface $\phi^{-1}(a)$, let $\xi_{a}$ denote the contact structure $\ker \alpha\cap T\phi^{-1}(a)$ and $E_x$ be the $\rd(r\alpha)$ complement of $\xi_a$ for $x\in \phi^{-1}(a)$. Then $J$ being compatible with $\rd(r\alpha)$ and $r\alpha \circ J = \rd \phi$ determines a complex structure on $E_x$ and $J|_{\xi_a}$ is a $\rd(r\alpha)$ compatible almost complex structure, and those two descriptions are equivalent. It is clear we can extend $J$ to $\widetilde{J}$ maintaining such properties.
\end{proof}

McLean proved the symplectic uniruledness is equivalent to the algebraic uniruledness for affine varieties. In particular, the algebraic uniruledness is rather a symplectic property. In the following, we use the constructions in \S \ref{s3} and \S \ref{s4} to prove that uniruledness is implied by the existence of symplectic dilation, which is, of course, a symplectic property. Unlike the results in other sections, we do not assume $c_1=0$ in this section. Hence all gradings should be understood as $\Z/2$ gradings.

\begin{theorem}\label{thm:unirule}
	Let $R$ be a ring and $W$ a Liouville domain such that $SH^*(W;R) = 0$ or there exists a symplectic dilation in $SH^1(W;R)$. Then $W$ is $(1,\Lambda)$-uniruled for some $\Lambda$.
\end{theorem}
\begin{remark}
	For Liouville domain with vanishing symplectic cohomology, a version of uniruledness different from Definition \ref{def:unirule} was obtained, see \cite{albers2017local}. However, to obtain Theorem \ref{thm:nonvanishing}, we need to achieve the more refined version of uniruledness in Definition \ref{def:unirule}.
\end{remark}
\begin{proof}
	We write $W_{\delta}=W \backslash \partial W \times(1-\delta,1]$. Then by \cite[Theoerem 2.3]{mclean2014symplectic}, there exists a $\Lambda>0$ such that $W$ is $(1,\Lambda)$ uniruled iff $W_{\delta}$ is $(1,\Lambda')$ uniruled for some $\Lambda'>0$. Assume $SH^*(W;R)=0$, then $1\in \Ima(SH_+^{-1}(W;R) \to H^0(W;R))$. That is $1\in \Ima (H^{-1}(C_+(\bH,J))\to H^0(C_0(f)))$, in particular this means that for some $i>0$, $1\in \Ima(H^{-1}(C^{D_i}_+(\bH,J))\to H^0(C_0(f)))$. Since $1$ is generated by the unique local minimum point $m$ of $f$. Therefore for a generic almost complex structure $J$, we have a dimension zero compact manifold $\cM_{m,x} \ne \emptyset$ for some $x\in \cP^*(\bH)$ with grading $|x|=-1$ and action $\ge -D_i$. Since $m$ is a minimum, an element of $\cM_{m,x}$ is simply a solution to the following,
	\begin{equation}\label{eqn:jcurve}
	u:\C \to \widehat{W}, \quad \partial_su+J(\partial_tu-X_{\bH})=0, \quad u(0)=m, \lim_{s\to \infty}u=x.
	\end{equation}
	For any convex almost complex structure $J_{\delta}$ on $W_{\delta}$. Assume $\phi$ is the function in Definition \ref{def:convexgeneral}, we may assume $\phi|_{\partial W_{\delta}}=1-\delta$. Then we can extend $\phi$ to $\widehat{\phi}$ on $\widehat{W}$ such that $\partial_r\widehat{\phi}>0$ on $\partial W \times [1-\delta, \infty)$ and $\widehat{\phi}=r$ when $r\ge 1$. By Proposition \ref{prop:extension}, $J_{\delta}$ can be extended to a $J\in \cJ(W)$. We may assume $J\in \cJ^{\le 1}_{reg}(\bH,f,g)$, since we perturb $J$ near $\cP^*(\bH)$ to achieve transversality. Therefore we have a curve $u$ solving \eqref{eqn:jcurve} and when restricted to $W_{\delta}$ it is a $J_{\delta}$-curve. Such curve has an energy bound by $D_i$. Let $S$ denote the connected component of $0$ of $u^{-1}(W_{\delta}^0)$. We claim $H_1(S;\Q)$ must be rank $0$. We can find a small $0<\epsilon<\delta$, such that $S'\subset u^{-1}(W \backslash \partial W \times (1-\epsilon,1] )$ is the connected compact Riemann surface containing $0$ with boundary and $u(\partial S')\subset \partial W \times \{1-\epsilon\}$ and $J$ is cylindrical convex on $\partial W \times \{1-\epsilon\}$. Then $S'$ is a disk, for otherwise, there is a domain $D\subset \C$ diffeomorphic to a disk, such that $u|_D$ solves \eqref{eqn:jcurve} and $u(\partial D) \subset  \partial W \times \{1-\epsilon\}$, this contradicts Lemma \ref{lemma:max}. Note that $S\subset S'$. If $H_1(S;\Q)$ is not rank $0$, then there is a loop $\gamma \subset S$ bounding a disk $D \subset S'$ and $D$ is not contained in $S$. Since we have $\widehat{\phi}\circ u|_\gamma < 1-\delta$ and $\max \widehat{\phi}\circ u|_{D\backslash S}\ge 1-\delta$, which contradicts the maximal principle for holomorphic curves. Therefore $u|_S$ is the $J_{\delta}$-curve we are looking for in $W_{\delta}$ passing through $m$.
	
	Next, we need to show such construction can be applied to any $m$ with a universal energy bound $D_i$. What we need is a universal $D_i$, such that $1\in \Ima(H^0(C^{D_i}_+(\bH,J,f))\to H^{-1}(C_0(f)))$ for any admissible Morse function $f$. This can be seen from Proposition \ref{prop:cascades}, since there is a $D^i$, such that $1\in \Ima(H^{-1}(C^{D_i}_+(H))\to H^{0}(C_0(H)))$ where $H$ is the perturbation of $\bH$. Therefore $W_{\delta}$ is $(1,D_i)$-uniruled.	

	Next, we assume $SH^*(W;R)\ne  0$ and admits a dilation. Then by Proposition \ref{prop:dilation}, we have $1\in \Ima \phi$. By the same argument as above, there is a $D^+_i$, such that
	$$\phi^{D^+_i}: \ker \Delta^{D^+_i}_+ \subset H^1(C^{D_i^+}_+(\bH_+,J_+))\to H^{0}(C_0(f))/\Ima \delta^{D_i^{-}}$$ 
	contains $1$ in the image. Since $SH^*(W;R)\ne 0$, we have $\Ima \delta^{D_i^{-}}$  does not contain $1$. Therefore we have $x\in C^{D^+_i}_+(\bH_+,J_+), b \in C^{D_i^{-}}_+(\bH_-,J_-)$, such that $\Delta_{+,0}(x)+d_{+,0}(b) = m+c$ for $c\in C^*_0(f)$ is a cochain representing class in $\Ima \delta^{D_i^{-}}$, where $m$ is the minimum point of $f$.\footnote{When $c_1(W)=0$, we have $c=0$ by degree reasons.} Similar to the argument above, since $m$ is a minimum, for generic $J_+,J_-$ and $J^{\theta}_{s,t}$, either there is a $x\in \cP^*(\bH_+)$ and $\theta \in S^1$ and a solution to 
	$$\partial_su+J^\theta_{s,t}(\partial_tu-X_{\bH^{\theta}_{s,t}})=0, u(0)=m, \lim_{s\to \infty}u=x.$$
	Or there is a $b\in \cP^*(\bH_-)$ and a solution to 
	\begin{equation}\label{eqn:b}
		\partial_su+J_{-}(\partial_tu-X_{\bH_-})=0, u(0)=m, \lim_{s\to \infty}u=b.
	\end{equation}
	We may assume $J_{\pm}, J^{\theta}_{s,t}$ restrict to $W_{\delta}$ is $J_{\delta}$ as before. Then in either case, when $u$ restricted to the preimage of $W_{\delta}^0$ is a $J_{\delta}$-curve passing through $m$, with energy bound $D_i^{-}$. Note that integrated maximal principle can also be applied to $r=1-\epsilon$ for $\bH^{\theta}_{s,t}$, since $\bH^{\theta}_{s,t}\equiv 0$ there and $\partial_s \bH^{\theta}_{s,t}\le 0$. Hence the component of $u^{-1}(W_{\delta}^0)$ containing $0$ has rank $0$ in rational first homology in either case as before. Then by the functoriality in \S \ref{sub:nondeg}, we can change $f$ as before to find curves passing though any point with a universal energy bound. 	 
\end{proof}

\begin{remark}
	If one considers the holomorphic planes with one marked point in the completion. Then virtual dimension of such space is given by $\mu_{CZ}(\gamma)+n-1$, when the plane is asymptotic to a Reeb orbit $\gamma$. To hope the evaluation map at the marked point covers every point in $W$, we need $\mu_{CZ}(\gamma)\ge n+1$. If we expect the evaluation map behave like a covering map then we need $\mu_{CZ}(\gamma)= n+1$. In many cases, the vanishing of symplectic cohomology is due to a Reeb orbit of index $n+1$. Note that a non-degenerate Reeb orbit of index $n+1$ can be perturbed into two non-degenerate Hamiltonian orbits with indices $n+1$, $n+2$. In our grading convention, the index $n+1$ Hamiltonian orbits is of grading $-1$, which is often responsible for the vanishing. This is the case for the standard symplectic ball. On the other hand, when we consider dilation $x\in SH^1(W)$, the associated Reeb orbit is of index $n-1$ or $n-2$, hence the uniruledness should not be provided by such Reeb orbits. However, the $b\in C^{-1}_+(\bH_-,J_-)$ in the proof of Theorem \ref{thm:unirule} provides the Reeb orbit of the right index. This suggests that the uniruledness should be from a solution to \eqref{eqn:b}. In fact, this is the case for $T^*S^n$. We will investigate them in more detail in \cite{higher}.
\end{remark}
\begin{remark}\label{rmk:counter}
	On the other hand, $1$-uniruledness does not imply the existence of symplectic dilation. For example, we take a smooth cubic hypersurface $V$ in $\CP^6$ and $W$ be the complement of the intersection of $V$ with a generic hyperplane. Then $W$ is an affine variety and $W$ does not admit a symplectic dilation by degree reasons, see \cite[Example 2.7]{seidel2014disjoinable}. However, by \cite[Corollary 5.4]{keel1999rational}, $W$ is $1$-uniruled.
\end{remark}
\begin{remark}
	In fact, the vanishing of symplectic cohomology and existence of symplectic dilation are the first two simplest conditions implying uniruledness in a whole family. The next one is whether $1$ is in the image of the map $\Delta^2: \ker \Delta \to \coker \Delta$, which is defined using $\Delta$ and the homotopy operator in the $(\Delta)^2=0$ relation and is with a degree shift by $3$. One way of packaging all the homotopic relations is using the $S^1$-equivariant theory as in \cite{zhao2014periodic}, then the spectral sequence from the $u$-filtration (on any $S^1$-cochain complex) induces maps	 $\Delta^{i+1}:\ker \Delta^i\to \coker \Delta^{i}$ with $\Delta^1$ is the BV operator considered here.  We say $W$ admits a $k$-dilation iff $1\in \Ima \Delta^k$. In fact, the counterexample in Remark \ref{rmk:counter} satisfies that $1\in \Ima \Delta^2$, hence it has a $2$-dilation but has no $1$-dilation. Every such structure has a related map $\phi^i$ defined on a subspace of $SH^*_+(W)$ to a quotient space of $H^*(W)$ generalizing $\phi$ in \eqref{eqn:phi}. And there is a boundary version for all $i$ in a similar way to $\Delta_\delta,\delta_{\partial}$. These structures have similar property to $\delta_{\partial}, \Delta_{\partial}$ in \S \ref{s3}, \S \ref{s4}. That is they are independent of fillings for ADC manifolds and $1$ being in the image of any of them implies uniruledness. They can be used to develop more obstructions to Weinstein fillability. Moreover, they give an infinity hierarchy on the complexity of symplectic manifolds, and also an infinity hierarchy on the complexity of ADC contact manifolds. Details of such construction will appear in the sequel paper \cite{higher}. The existence of $k$-dilation for some $k$ is equivalent to the existence of cyclic dilation for $h=1$, which is defined independently by Li \cite{li2019exact}, where the open string implication of cyclic dilations is studied. 
\end{remark}	
\begin{remark}
	From the proof of Theorem \ref{thm:unirule}, it is clear that $1\in \Ima  \Delta_{\partial}$ would also imply uniruledness. By Proposition \ref{prop:dilation}, $1\in \Ima \Delta_{\partial}$ is potentially weaker than existence of dilation.
\end{remark}

Using Theorem \ref{thm:unirule}, non-uniruledness is an implication of non-existence of symplectic dilation in any coefficient. Since log Kodaira dimension provides an obstruction to uniruledness, we have the following. 
\begin{theorem}\label{thm:nonvanishing}
	Let $W$ be an affine variety. In either of the following cases, we have $SH^*(W;R) \ne 0$ and there is no symplectic dilation using any coefficient ring $R$.
	\begin{enumerate}
		\item The log Kodaira dimension is not $-\infty$. 
		\item $W$ admits a projective variety $V$ compactification, such that $V$ is not uniruled, for example the Kodaira dimension of $V$ is not $-\infty$.
	\end{enumerate}
\end{theorem}
\begin{proof}
	By \cite[Theorem 2.5, Lemma 7.1]{mclean2014symplectic}, if $W$ does not have log Kodaira dimension $-\infty$, $W$ can be not $(1,\Lambda)$-uniruled for any $\Lambda$. Then by Theorem \ref{thm:unirule}, we have $SH^*(W,R) \ne 0$ and there is no symplectic dilation for any coefficient ring $R$.
	
	Let $V$ be projective compactification of $W$. If we have $SH^*(W,R) = 0$ or there is a symplectic dilation. Then $W$ is $(1,\Lambda)$-uniruled. By \cite[Theorem 2.5]{mclean2014symplectic}, $W$ is $1$-algebraically uniruled. Hence $V$ is algebraically uniruled and the Kodaira dimension of $V$ is $-\infty$, hence it is a contradiction.
\end{proof}
Now, Corollary \ref{cor:H} follows from Theorem \ref{thm:nonvanishing} directly. Symplectic exotic $\C^{n}$ was constructed in \cite{abouzaid2010altering,mclean2009lefschetz,seidel2005symplectic} for all $n\ge 3$. The proofs of exoticity were based no non-vanishing of symplectic cohomology for some coefficient, but the mechanisms for non-vanishing were very different. In view of Corollary \ref{cor:H}, the exoticity can also be explained by the failure of uniruledness. 
\begin{example}
	Seidel-Smith \cite{seidel2005symplectic} proved that product $M^{m}, m \ge 2$ is an exotic $\C^{2m}$ for Ramanujam surface $M$ \cite{ramanujam1971topological} by showing that $M$ contains an essential Lagrangian, hence $SH^*(M) \ne 0$ and $M^{m}$ is exotic. Since $M$ has log Kodaira dimension $2$, we have $M^m$ has log Kodaira dimension $2m$. Therefore by Theorem \ref{thm:nonvanishing}, $SH^*(M^m)\ne 0$ and $SH^*(M^m)$ does not admit symplectic dilations. This reproves Seidel-Smith's result. 
\end{example}	

\begin{remark}
	There are many algebraic exotic $\C^n$ with log Kodaira dimension $-\infty$, e.g. $M\times \C$ for Ramanujam surface $M$. However, $M\times \C$ is symplectically standard by h-principle, since it is subcritical and diffeomorphic to $\C^3$.  The atomic exotic $\C^4$ considered by McLean \cite{mclean2009lefschetz} from Kaliman modifcation also has negative log Kodaira dimension. Therefore log-Kodaira dimension being non-negative is not the criterion of being symplectic exotic.
\end{remark}

By Corollary \ref{cor:H}, one can look for symplectic exotic $\C^n$ in complex exotic $\C^n$. They exist in abundance when $n\ge 3$, see \cite{zaidenberg1996exotic}. 

Non-vanishing of symplectic cohomology also implies non-displaceablity \cite{kang2014symplectic}. Hence we have the following.
\begin{example}
	Let $V$ be a projective variety with non negative Kodaira dimension. Let $W=V\backslash D$ be an affine variety for a divisor $D$. Then $SH^*(W) \ne 0$, and $W$ is not displaceable. That is when we view $W$ as a Liouville domain, there is no Hamiltonian $F \in C^\infty_c(S^1\times \widehat{W})$ such that the generated time one Hamiltonian diffeomorphism $\phi_F$ displaces $W$, i.e. $\phi_F(W)\cap W = \emptyset$. This holds in particular, if $V$ is Calabi-Yau, since the Kodaira dimension is $0$.
\end{example}

As showed in \S \ref{s3} and \S \ref{s4}, the vanishing of symplectic cohomology and the existence of symplectic dilations are properties independent of fillings for ADC contact manifolds. A natural question is whether uniruledness shares the similar property.
\begin{question}
	Is uniruledness independent of the (topologically simply) filling for ADC contact manifold?
\end{question}
\section{Constructions of ADC manifolds}\label{s6}
As explained in Example \ref{ex:ADC}, ADC contact manifolds exist in abundance. Moreover, a flexible handle attachment does not change the ADC property \cite{lazarev2016contact}. In this section, we provide two other simple constructions of ADC manifolds, which provides many examples, where Corollary \ref{cor:ob} and Corollary \ref{cor:ob2} can be applied. In particular we will prove Theorem \ref{thm:I}.

\subsection{Product with the complex plane}
In this subsection, we show that the boundary of $V\times \C$ is $0$-ADC for any Liouville domain $V$ such that $c_1(V) = 0$ and $\dim V>0$. Moreover by \cite{oancea2006kunneth}, $V\times \C$ has vanishing symplectic cohomology. Since there are examples of $V\times \C$ such that it can not be a Weinstein domain, such construction proves many non-flexible examples where results from \S \ref{s3} can be applied. In the following, we fix the symplectic form on $\C$ by $\rd(r^2\rd \theta)$.

\begin{definition}
	Let $V$ be a connected manifold with non-empty boundary, the Morse dimension $\dim_M V$ is defined to be the minimum of the max index of an admissible Morse function on $V$. Then $\dim_M V \le \dim V -1$.
\end{definition}

We also introduce the notion of tamed asymptotically dynamically convex contact manifold, which is ADC and $\alpha_i$ does not collapse to $0$. This will be important for our discussion on non-exact fillings.

\begin{definition}\label{def:TADC}
	$(Y,\xi)$ is $k$-TADC if there exist contact forms $\alpha_1>\alpha_2>\ldots $, positive numbers $D_1<D_2<\ldots \to \infty$ and a contact form $\alpha$, such that $\alpha_i>\alpha$ and all elements of $\cP^{<D_i}(\alpha_i)$ are non-degenerate and have degree greater than $k$. We have similar definition for exact domains.
\end{definition}

Basic examples of ($0$-)TADC manifolds are index-positive contact manifolds, e.g. cotangent bundles when dimension of the base is at least $4$. The other case is index-positive in the Morse-Bott sense, e.g. Example \ref{ex:ADCdilation}. That they are TADC follows from that up to action $D$, there exist small perturbations into non-degenerate contact forms with prescribed Conley-Zehnder indexes \cite[Lemma 2.3, 2.4]{bourgeois2002morse}, and the perturbation can be made arbitrarily small so that the perturbed contact form does no collapse to $0$. The following result will provide some more examples of TADC manifolds.
\begin{theorem}\label{thm:product}
	Let $V$ be a $2n$-dimensional Liouville domain such that $c_1(V) = 0$ and $n>0$. We define $Y$ to be the contact boundary of $V\times \C$,  then we have the following. 
	\begin{enumerate}
		\item $Y$ is $0$-ADC.
		\item If $\partial V$ supports a non-degenerate contact form such that all Reeb orbits are contractible in $V$, then $Y$ is strongly $0$-ADC.  
	\end{enumerate}
In fact, $Y$ is (strongly) $(2n-\dim_MV-1)$-ADC.  If $V$ is $k$-TADC for some $k$, then $Y$ is TADC. 
\end{theorem}

We will prove the theorem by setting up the following propositions. 
\begin{proposition}\label{prop:reeb}
	Assume $Y$ is a contact manifold with a contact form $\alpha$ and $V$ is a Liouville domain with a Liouville form $\lambda_V$. Let $f$ be a strictly positive function on $V$, then $\alpha_f:=f\alpha+\lambda_V$ is a contact form on $Y\times V$ iff $\rd(\frac{1}{f}\lambda_V)$ is a symplectic form on $V$. The Reeb vector field is given by $R_f:=\frac{1}{f}R_\alpha+\widetilde{X_{\frac{1}{f}}}$. Here $\widetilde{X_\frac{1}{f}}-X_{\frac{1}{f}}\in \la R_\alpha \ra$ and $\alpha_f(\widetilde{X_{\frac{1}{f}}})=0$ and $X_{\frac{1}{f}}$ is the Hamiltonian vector field for $\frac{1}{f}$ using the symplectic form $\rd(\frac{1}{f}\lambda_V)$. 
\end{proposition}
\begin{proof}
	Assume $\dim Y = 2m+1$ and $\dim V = 2n$, to show $\alpha_f$ is a contact form it is sufficient to show $\frac{1}{f}\alpha_f$ is a contact form. Note that we have
	\begin{eqnarray*}
		\frac{1}{f}\alpha_f \wedge (\rd (\frac{1}{f}\alpha_f))^{m+n} & = & (\alpha+\frac{1}{f}\lambda_V) \wedge (\rd \alpha + \rd(\frac{1}{f}\lambda_V))^{m+n}\\
		& = & \binom{m+n}{m} (\alpha+\frac{1}{f}\lambda_V) \wedge (\rd \alpha)^m \wedge (\rd(\frac{1}{f}\lambda_V)^{n} \\
		& = &  \binom{m+n}{m} \alpha \wedge (\rd \alpha)^m\wedge (\rd(\frac{1}{f}\lambda_V))^n
	\end{eqnarray*}
	Hence $\alpha_f$ is a contact form iff $\rd(\frac{1}{f}\lambda)$ is non-degenerate. Note that by construction, we have $\alpha_f(R_f)=1$. Moreover, we have
	\begin{eqnarray*}
		\iota_{R_f} \rd(\alpha_f) & =& \iota_{R_f} \rd (f(\alpha+\frac{1}{f}\lambda_V))\\
		& = & \iota_{R_f} \rd f \wedge (\alpha+\frac{1}{f}\lambda_V) + \iota_{R_f} f \rd(\alpha+\frac{1}{f}\lambda_V)
	\end{eqnarray*} 
	Since $\alpha_f(\widetilde{X_{\frac{1}{f}}}) = 0$, we have $\iota_{R_f}(\alpha+\frac{1}{f}\lambda_V) = \frac{1}{f}$. Since $\iota_{R_f} \rd f = 0$, $\iota_{R_f} \rd \alpha = 0$ and $\iota_{R_f}\rd(\frac{1}{f}\lambda_V) = -\rd \frac{1}{f}$ by definition, we have
	$\iota_{R_f} \rd(\alpha_f) = -\frac{\rd f}{f}-f \rd \frac{1}{f}=0$. Hence $R_f$ is the associated Reeb vector.
\end{proof}
\begin{remark}\label{rmk:convex}
	Note that $Y\times V$ can be viewed as a hypersurface $r = f$ in the Liouville cobordism $Y\times \R_+ \times V$ with Liouville form $r\alpha + \lambda_V$. Then this hypersurface is of contact type if the Liouville vector $r\partial_r + X_{\lambda_V}$ is transverse to the surface and pointing out, that is $(r\partial_r + X_{\lambda_V})(r-f)=f-X_{\lambda_V}f>0$.
\end{remark}

\begin{remark}\label{rmk:contact}
	Proposition \ref{prop:reeb} can be viewed as a generalization of the computation carried out for prequantization bundles. Let $\gamma$ be a Reeb trajectory on $Y$, then over $\gamma \times V$ we have a connection given by $\alpha+\frac{1}{f}\lambda_V$ such that the curvature is the symplectic form $\rd(\frac{1}{f}\lambda_V)$. The contact structure is given by the horizontal subspace direct sum with the contact structure of $Y$, the splitting holds in the symplectic sense.
\end{remark}
We will divide the boundary of $V\times \C$ into two parts, one of them is diffeomorphic to $\partial V \times \D$ and the other is diffeomorphic to $V\times S^1$, where $\D\subset \C$ is a disk. In the following, we discuss the Reeb dynamics on the $\partial V \times \D$  part.
\begin{proposition}\label{prop:cap}
	Let $(Y,\xi)$ be a contact manifold with a contact form $\alpha$. We fix a small $\epsilon>0$. Assume $g$ is a smooth function on $[0,1]$ such that the following holds.
	\begin{enumerate}
		\item $g(x) = x$ for $x$ near $0$, $g(1-\epsilon)=1$.
		\item $g(x)$ is increasing and $g''(x)>0$ unless $g(x)=x$.
	\end{enumerate}
	Then for $\rho\in \R_+$, $(2-g(\rho r^2))\alpha+r^2 \rd\theta$ is a contact form on $Y\times \D_{\sqrt{1-\epsilon/\rho}}$. And Reeb orbits of it are of the form $(\gamma(At),re^{Bi t+\theta_0})$, where $\gamma$ is a Reeb orbit of $R_\alpha$ on $Y$ and 
	$$A=\frac{1}{2-g(\rho r^2)+\rho r^2 g'(\rho r^2)}, \quad B = \frac{\rho g'(\rho r^2)}{2-g(\rho r^2)+\rho r^2 g'(\rho r^2)}.$$ 
	If $\gamma$ is non-degenerate of period $D$, then we have the following. 
	\begin{enumerate}
		\item\label{r1} When $r=0$, then the Reeb orbit $(\gamma(\frac{1}{2}t),0,0)$ is non-degenerate iff $\frac{D\rho}{2\pi} \notin \N$ and the Conley-Zehnder index is given by $\mu_{CZ}(\gamma)+2\lfloor \frac{D\rho}{2\pi} \rfloor+1$.
		\item\label{r2} When $\rho r^2$ is in the domain where $g(x)\ne x$, then $(\gamma(At),re^{Bi t+\theta_0})$ is an orbit iff $D \rho g'(\rho r^2) \in 2\pi \N$. It is a Morse-Bott  $S^1$-family of orbits  with period $\frac{D}{A}$, and the generalized Conley-Zehnder index is given by $\mu_{CZ}(\gamma)+\frac{D\rho g'(\rho r^2)}{\pi} +\frac{1}{2}$. 
	\end{enumerate}
\end{proposition}
\begin{proof}
	Note that $(2-g(\rho r^2)) - \frac{r}{2}\partial_r(2-g(\rho r^2)) = 2-g(\rho r^2) + \rho r^2 g'(\rho r^2) > 0 $ on $Y\times \D_{\sqrt{1-\epsilon/\rho}}$. Then by Remark \ref{rmk:convex},  $\alpha_g:=(2-g(\rho r^2))\alpha+r^2 \rd\theta$ is a contact form. By Proposition \ref{prop:reeb}, the Reeb vector field is given by
	$$R:= \frac{1}{2-g(\rho r^2)}R_\alpha + \widetilde{X}_{1/(2-g(\rho r^2))}.$$
	By a direct computation, we have 
	$$\widetilde{X}_{1/(2-g(\rho r^2))}= \frac{\rho g'(\rho r^2)}{2-g(\rho r^2)+\rho r^2 g'(\rho r^2)} \partial_\theta + \frac{-\rho r^2 g'(\rho r^2)}{(2-g(\rho r^2)+\rho r^2 g'(\rho r^2))(2-g(\rho r^2))}R_\alpha.$$
	Therefore the Reeb vector is 
	$$R= \frac{1}{2-g(\rho r^2)+\rho r^2 g'(\rho r^2)}R_\alpha +  \frac{\rho g'(\rho r^2)}{2-g(\rho r^2)+\rho r^2 g'(\rho r^2)} \partial_\theta = A R_\alpha + B \partial_\theta.$$
	Hence all Reeb orbits are in the form prescribed in the proposition.
	
	As explained in Remark \ref{rmk:contact}, the contact structure on $Y\times \D_{\sqrt{1-\epsilon/\rho}}$ is the direct sum of $\xi$ with kernel of $\alpha+\frac{r^2}{2-g(\rho r^2)}\rd \theta$ in the space $\la R_{\alpha}\ra\oplus T\D_{\sqrt{1-\epsilon/\rho}}$. Since the Reeb vector is tangent to $Y\times S_r^1$ and the $R_\alpha$ component is a constant, we know that Reeb flow of $R$ preserves $\xi$ and the splitting. Therefore when computing Conley-Zehnder index, it is sufficient to figure out the Conley-Zehnder index of the horizontal direction. In the $r=0$ case, the linearized return map is given by the linearized return map of $\gamma$ direct sum with the rotation by $e^{D\rho i}$ in the horizontal direction. Therefore it is non-degenerate iff $\frac{D\rho }{2\pi} \notin \N$ and the Conley-Zehnder index in the horizontal direction is $1+2\lfloor \frac{D \rho}{2\pi}\rfloor$. 
	
	When $\rho r^2$ is in the domain where $g(x)\ne x$, then periodic orbits are in the form of $(\gamma(At),re^{Bit+\theta_0})$, hence they always come in an $S^1$ family for $\theta_0\in S^1$. To verify that it is a Morse-Bott in the sense of \cite[Definition 1.7]{bourgeois2002morse}, note that the linearized return map $\phi$ is the direct sum of the linearized return map of $\gamma$ and the return map on the horizontal direction is given by 
	$$\phi_{T}=\left[\begin{array}{cc}
	1 &  0 \\
	T\frac{\rd}{\rd r} \frac{\rho g'(\rho r^2)}{2-g(\rho r^2)+\rho r^2 g'(\rho r^2)} & 1
	\end{array}\right],$$
	here we use $\partial_r,\partial_\theta$ as the basis and $T$ is the period $D \rho g'(\rho r^2)$. Note that we have
	$$\frac{\rd}{\rd r} \frac{\rho g'(\rho r^2)}{2-g(\rho r^2)+\rho r^2 g'(\rho r^2)} = \frac{2\rho r g''(\rho r^2)(2-g(\rho r^2))+2\rho^2 r (g'(\rho r^2))^2}{(2-g(\rho r^2)+\rho r^2 g'(\rho r^2))^2} > 0.$$ 
	Therefore $\ker (\phi_T - \Id)$ is spanned by $\partial_\theta$. Hence the $1$-eigenspace of the total linearized return map is spanned by $\partial_{\theta}$ and $R_\alpha$, which is the tangent space of $\Ima \gamma \times S_r^1$. Moreover $\rd \alpha_g=(2-g(\rho r^2))\rd\alpha-2\rho r g'(\rho r^2) \rd r \wedge \alpha + 2r\rd r \wedge \rd \theta$ is rank zero on $\Ima \gamma \times S^1_r$. Therefore such Reeb orbits are of Morse-Bott type. To compute the generalized Maslov index, it is enough to compute the generalized Maslov index in the horizontal direction. Note that the linearized map is given by
	$$\phi_t:=\left[\begin{array}{cc}
	\cos t & -\sin t\\
	\sin t & \cos t\\
	\end{array}
	\right]\cdot
	\left[\begin{array}{cc}
	1 &  0 \\
	Ct & 1
	\end{array}\right], \quad 0\le t \le T=\frac{D\rho g'(\rho r^2)}{2\pi},$$
	where $C=\frac{2\rho r g''(\rho r^2)(2-g(\rho r^2))+2\rho^2 r (g'(\rho r^2))^2}{(2-g(\rho r^2)+\rho r^2 g'(\rho r^2))^2}>0$. Therefore the generalized Conley-Zehnder index \cite{robbin1993maslov} is given by $\frac{D\rho g'(\rho r^2)}{\pi}$ from the rotation plus the generalized Conley-Zehnder index of 
	$$
	\left[\begin{array}{cc}
	1 &  0 \\
	Ct & 1
	\end{array}\right], \quad 0\le t \le T.$$
	The latter is  a symplectic sheer. Since $C>0$, the index is given by $\frac{1}{2}$. Then the total generalized Conley-Zehnder index is given by $\mu_{CZ}(\gamma)+\frac{D\rho g'(\rho r^2)}{\pi}+\frac{1}{2}$.
\end{proof}

On the complement $V\times S^1$, we will use the following lemma to guarantee all Reeb orbits are either over critical points or have very large period. 
\begin{lemma}\label{lemma:long}
	Let $(V,\omega)$ be a compact symplectic manifold possibly with boundary. For all $C>0$, there exists $\epsilon \in \R^+$, such that for all functions $H$ with $|\rd H|_{C^0}<\epsilon$ and all symplectic forms $\omega'$ such that $|\omega'-\omega|_{C^1}<\epsilon$, we have all non-constant periodic orbits of Hamiltonian vector field $X_H$ using symplectic form $\omega'$ have period at least $C$.
\end{lemma}
\begin{proof}
	Let $g$ be a Riemannian metric on $V$. We can find a cover of $V$ by finitely many Darboux charts $\{U_i\}$, such that there exists $\delta>0$ so that every $\delta$-ball is contained in one of the Darboux charts. Moreover, for smaller enough $\epsilon'$, we can assume that each $U_i$ is a Darboux chart of $\omega'$ for $|\omega'-\omega|_{C^1}<\epsilon'$. Using the standard metric on each $U_i\subset \R^{2n}$ by viewing $U_i$ as a $\omega$ Darboux chart, we have a $C^0$ norm on $\Omega^1(V)$. Since $\omega'$ is $C^1$ close to $\omega$, we know that the induced $C^0$ norm on $\Omega^1(V)$ are uniformly equivalent for any $\omega'$ nearby and is equivalent to the $C^0$ norm we use to state the lemma. If $|\rd H|_{C^0}<\frac{\delta}{C}$ then any trajectory $\phi(t)$ of $X_H$ has the property that $\rd(\phi(0),\phi(C))<\delta$. By \cite[Proposition 6.1.5]{audin2014morse}, on each Darboux chart, if $|\rd H|_{C^0}<\frac{2\pi}{C}$, then any $C$-periodic orbit in the chart is a constant. Therefore the claim holds for $\epsilon = \min(\frac{\delta}{C},\frac{2\pi}{C},\epsilon')$.  	
\end{proof}

\begin{proof}[Proof of Theorem \ref{thm:product}]
	Since we have $c_1(V) = 0$, then $c_1(V\times \C) = 0$. Hence $c_1(Y)=0$, that is the Conley-Zehnder index is well-defined in $\Z$ for contractible orbits on $Y$. Now we pick a Liouville form $\lambda$ on $V$ such that the associated Reeb dynamic on $\partial V$ is non-degenerate. Then there exist positive real numbers converging to infinity $D_1 < D_2 <\ldots $ and integers $N_1,N_2,\ldots$, such that all periodic orbits of $R_\lambda$ contractible in $V$ with period smaller than $D_i$ have Conley-Zehnder index greater than $N_i$. We consider orbits contractible in $V$, since $(\gamma,p)$ for $\gamma \in \cL\partial V, p \in \C$ is homotopically trivial on $Y:=\partial(V\times C)$ if $\gamma$ is contractible in $V$.
	
	We view $Y$ as the hypersurface $Y_{\rho,f,g}$ in $\widehat{V} \times \C$ such that $Y^1_{\rho,f,g}:=Y_{\rho,f,g}\cap (\partial V \times [1,\infty)_t \times \C)$ is given by $t = 2-g(\rho r^2)$ for big $\rho \in \R_+$ and function $g$ as in Proposition \ref{prop:cap},  $Y^2_{\rho,f,g}:=Y_{\rho,f,g}\cap (V \times \C)$ is given by $r^2 = \frac{f}{\rho}$ for a Morse function $f$ on $V$ such that near the boundary $f = g^{-1}(2-t)$, where $t$ is the collar coordinate smaller than $1$. To see $Y_{\rho,f,g}$ is of contact type, it is sufficient to prove that the Liouville vector field
	$X_\lambda+\frac{1}{2}r\partial_r$ is transverse to $Y_{\rho}$ and points out. On $Y^1_{\rho,f,g}$, this is verified in Proposition \ref{prop:cap}. On $Y^2_{\rho,f,g}$, we have $(X_\lambda+\frac{1}{2}r\partial_r)( r^2-\frac{f}{\rho})=r^2-\frac{1}{\rho}X_\lambda f = \frac{1}{\rho}(f-X_\lambda f)$. Note that if $g'(1-\epsilon)$ is very big, $f'(1) = -\frac{1}{g'(1-\epsilon)}$ is very small, hence we can choose $f$ such that $X_\lambda f$ is very small and $f\ge 1-\epsilon$. Hence $Y_{\rho,f,g}$ is of contact type. And for $\rho_1>\rho_2$, there exists such $f_1,g_1,f_2,g_2$ so that we have the domain between $Y_{\rho_1,f_1,g_1}$ and $Y_{\rho_2,f_2,g_2}$ is a Liouville cobordism. Using the Liouville vector field, we have the induced contact form on $Y_{\rho_1,f_1,g_1}$ is smaller than the induced contact form on  $Y_{\rho_2,f_2,g_2}$.
	
	On $Y^1_{\rho,f,g}$, by Proposition \ref{prop:cap}, the Reeb orbits of type \eqref{r1} of action bounded by $D_i$ have Conley-Zehnder indices greater than $N_i+2\lfloor \frac{D_i\rho}{2\pi} \rfloor+1$. Since $A<1$ in Proposition \ref{prop:cap}, Reeb orbits of type \eqref{r2} of action bounded by $D_i$ is of Morse-Bott type in $S^1$-families with generalized Conley-Zehnder indices greater than $N_i+\lfloor \frac{D_i\rho}{\pi} \rfloor+\frac{1}{2}$, since $g'\ge 1$. 	On $Y^1_{\rho,f,g}$, it is possible that there are periodic orbits not described by \eqref{r1} and \eqref{r2} in Proposition \ref{prop:cap}. That is those orbits with $\rho r^2$ in the domain where $g(x)=x$. Then we have $A=\frac{1}{2}$ and $B = \frac{\rho}{2}$. Then if we pick $\rho$ such that $|\gamma|\rho \notin 2\pi \N$ for all all Reeb orbits $\gamma$ on $(\partial V,\lambda_V)$ with period smaller than $D_i$. Then there are no Reeb orbits with period smaller than $2D_i$ on this domain.
	
	$Y^2_{\rho,f,g}$ is contactomorphic to $V\times S^1$ with contact form $\lambda + \frac{\rho}{f} \rd \theta$. Then by Proposition \ref{prop:reeb}, the Reeb vector is given by $\frac{\rho}{f}\partial_\theta+\widetilde{X_{\frac{\rho}{f}}}$. Therefore there are two types Reeb orbits, one is of form $(p, \phi_k(t))$ where $p$ is a critical point of $\frac{\rho}{f}$ and $\phi_k(t)$ is a reparametrization of $e^{ikt},t\in [0,2\pi]$. The other type is $\gamma$ such that $\pi\circ \gamma$ is a non-constant periodic orbit of $X_{\frac{\rho}{f}}$ for the projection $\pi:V\times S^1\to V$. When $p$ is a non-degenerate critical point of $\frac{\rho}{f}$, we have $\mu_{CZ}(p,\phi_k(t)) = n-\ind p + 2k$. If we choose $f$ such that critical points of $\frac{1}{f}$ have max index $\dim_M V$, then we have the degree of $(p,\phi_k(t))$ is at least $n-\dim_MV+2k+n+1-3=2n-\dim_MV+2k-2\ge 2n-\dim_M V$. Since $\dim_M V \le 2n-1$, those orbits have positive degree. For the other type of periodic orbits, periodic orbits of $X_{\frac{\rho}{f}}$ using symplectic form $\rd(\frac{\rho}{f}\lambda)$ is the same as periodic orbits of $X_{\frac{1}{f}}$ using symplectic form $\rd (\frac{1}{f}\lambda)$. Since we can choose $g$ and $f$ such that $\frac{1}{f}-\frac{1}{1-\epsilon}$ is $C^2$ small, therefore $\rd(\frac{1}{f}\lambda)$ is $C^1$ close to $\frac{1}{1-\epsilon}\rd \lambda$. This can be achieved by choosing $g$ such that $g'(1-\epsilon)$ very large and $g''(1-\epsilon)$ very small. Therefore by Lemma \ref{lemma:long}, we can assume all non-constant periodic orbits of $X_{\frac{\rho}{f}}$ are of periods greater than $D_i$. 
	
	Therefore we can pick big enough $\rho_i$ and $f_i,g_i$, such that all contractible Reeb orbits on $Y_{\rho_i,f_i,g_i}$ with action smaller than $D_i$ are either non-degenerate with positive degrees or Morse-Bott in $S^1$-families with generalized Conley-Zehnder indices greater than $4-n$ (or any fixed large number). By \cite[Lemma 2.3,2.4]{bourgeois2002morse}, there exist a very small perturbation of the contact form on $Y_{\rho_i,f_i,g_i}$, such that all contractible Reeb orbits with periods smaller than $D_i$ have positive degree. For all $j > i$, we have a similar construction by choosing $\rho_j>\rho_i$, such that the contact form on $Y_{\rho_j,f_j,g_j}$ is smaller than that on $Y_{\rho_i,f_i,g_i}$. This finishes the proof of Theorem for ADC case.
	
	In the TADC case, since we do not need to use arbitrarily large $\rho$ to lift the Conley-Zehnder index on $Y^1_{\rho,f,g}$. Therefore $Y_{\rho,f,g}$ can be chosen without shrinking to $V\times \{0\}$.
\end{proof}

\begin{remark}
	By \cite{cieliebak2002subcritical}, every subcritical domain splits into the product of a Weinstein domain with $\C$. Theorem \ref{thm:product} implies that all $2n$-dimensional subcritical domains are $(n-2)$-ADC, which is a special case of \cite[Corollary 4.1]{lazarev2016contact}.
\end{remark}

\begin{corollary} \label{cor:iso}
	Let $V$ be a Liouville domain, such that $c_1(V) = 0$. Let $Y=\partial(V\times \C)$. Then for any topologically simple exact filling $W$ of $Y$, we have a ring isomorphism $\phi:H^*(W)\to H^*(V\times \C)$, such that following commutes
	$$\xymatrix{
		H^*(W)\ar[d]^{\phi} \ar[r] & H^*(Y)\\
		H^*(V\times \C) \ar[ru] & 
		}$$
\end{corollary}
\begin{proof}
	First note that $Y$ is always connected and $c_1(Y) = 0$. Then by Corollary \ref{cor:B} and Theorem \ref{thm:product} and that $SH^*(V\times \C) =0$, we have a group morphism $\phi:H^*(W)\to H^*(V\times \C)$ such that the diagram commutes. Since the composition $H^*(V\times \C) \to H^*(Y) \to H^*(V\times S^1)$ is injective, we have $H^*(V\times \C) \to H^*(Y)$ is injective and is a ring map, and $H^*(W)\to H^*(Y)$ is a ring map, they force $\phi$ to be a ring isomorphism.
\end{proof}

If $V$ is $k$-TADC for some $k$, then $\partial(V\times \C)$ is TADC. Hence by Remark \ref{rmk:exact}, the exactness assumption in Corollary \ref{cor:iso} can be dropped. As another instant application of Corollary \ref{cor:iso}, we have the following.
\begin{corollary}\label{cor:weinstein}
	If $V^{2n}$ is a Liouville domain, such that $n\ge 2$ and $c_1(V) = 0$. If one of the following conditions hold, then $\partial(V\times \C)$ is not Weinstein fillable. 
	\begin{enumerate}
		\item\label{c1} There exists $k>n+1$, such that $H^k(V)\ne 0$.
		\item\label{c2} There exists $k<n$, such that $H^k(V\times \C)\to H^k(\partial(V\times \C))$ is not isomorphism.
	\end{enumerate}
\end{corollary}
\begin{corollary}\label{rmk:weinstein}
	If $Y$ is constructed by attaching subcritical (for 2-handles, conditions in \cite[Theorem 3.17]{lazarev2016contact} need to hold) or flexible handles onto $\partial(V\times \C)$ in Corollary \ref{cor:weinstein}, such that $c_1(Y)=0$. Then $Y$ can not be filled by a Weinstein domain. 
\end{corollary}
\begin{proof}
	In other words, $Y$ is the boundary of $(V\times \C) \cup W$, where $W$ is a flexible Weinstein cobordism from $\partial(V\times \C)$ to $Y$. Since $\dim W =2n+2\ge 6$, $c_1(Y)=0$ implies that $c_1(W)=0$. Hence $c_1(\partial(V\times \C))=0$. Since $H^*(V\times \C) \to H^*(\partial(V\times \C))$ is injective, we have $c_1(V\times \C)=c_1(V)=0$. Therefore by Theorem \ref{thm:product}, $\partial(V\times \C)$ is ADC. Then by \cite[Theorem 3.15,3.17,3.18]{lazarev2016contact}, $Y$ is ADC and $SH^*((V\times \C) \cup W) = 0$ by \cite{bourgeois2012effect}. Therefore $SH^*_+((V\times \C) \cup W) \to H^{*+1}((V\times \C) \cup W)$ is an isomorphism. If $H^k(V)\ne 0$ for $k>n+1$, then $H^k((V\times \C)\cup W) \ne 0$. Hence $Y$ is not Weinstein fillable by Corollary \ref{cor:B}. If $H^k(V\times \C)\to H^k(\partial(V\times \C))$ is not isomorphism for $k<n$, then it is not surjective since it is always injective. Assume $k$ is the smallest among all such integers. Then we have the following exact sequence
	$$\ldots \to H^k(V\times \C, \partial(V\times \C)) \to H^k(V\times \C) \hookrightarrow H^k(\partial(V\times \C)) \to H^{k+1}(V\times \C, \partial(V\times \C))\to \ldots.$$
	Since $k$ is the first failure of subjectivity, we have $ H^k(V\times \C, \partial(V\times \C))=0$ and $H^{k+1}(V\times \C, \partial(V\times \C))\ne 0$. Since we have $H^*((V\times \C)\cup W, W) = H^*(V\times \C , \partial(V\times \C))$, there is another long exact sequence
	\begin{align*}
	\ldots \to H^k(V\times \C, \partial(V\times \C)) \to & H^k((V\times \C)\cup W) \to H^k(W) \to \\ &
	 H^{k+1}(V\times \C, \partial(V\times \C))\to H^{k+1}((V\times \C)\cup W) \to H^{k+1}(W) \to \ldots
	\end{align*}
	Therefore either $H^k((V\times \C)\cup W) \to H^k(W)$ is not surjective or $H^{k+1}((V\times \C)\cup W) \to H^{k+1}(W)$ is not injective. Since $H^k(W)\to H^k(Y)$ is isomorphism and $H^{k+1}(W)\to H^{k+1}(Y)$ is injective, by $k<n$ and $W$ being Weinstein. Hence either $\delta_{\partial}:SH_+^{k-1}((V\times\C)\cup W)\to H^k(Y)$ is not surjective, or $\delta_{\partial}:SH_+^{k}((V\times\C)\cup W)\to H^{k+1}(Y)$ is not injective. Then by Corollary \ref{cor:B} and  Theorem \ref{thm:A},  $Y$ can not be filled by a Weinstein domain, otherwise $\delta_{\partial}$ should be an isomorphism on $SH^{k-1}_+$ and injective on $SH^{k}_+$.
\end{proof}

 Corollary \ref{rmk:weinstein} provides many exactly fillable but not Weinstein fillable  manifolds. Exactly fillable but not Weinstein fillable manifolds were found in \cite{bowden2012exactly} in dimension $3$ and in \cite{bowden2014topology} for higher dimensions. Our construction makes very few topological requirements and roughly shows that most Liouville but not Weinstein manifolds give rise to exactly fillable but not Weinstein fillable contact manifolds after a product with $\C$ and attaching subcritical or flexible handles.

\begin{remark}
	In the case \eqref{c1} of Corollary \ref{cor:weinstein} and Corollary \ref{rmk:weinstein}, the obstruction in Corollary \ref{cor:ob} does not vanish. Using Proposition \ref{prop:many}, we have many examples of not Weinstein fillable manifolds, whose symplectic cohomology is nonzero for a filling.
\end{remark}

\begin{example}\label{ex:exact}
Liouville but not Weinstein domains were first constructed in \cite{mcduff1991symplectic} in dimension $4$, higher dimensional examples were constructed in  \cite{geiges1994symplectic,massot2013weak}.  All such examples are of a diffeomorphism type $M\times [0,1]$, where the diffeomorphism type of $M$ is the following.
\begin{enumerate}
	\item $ST^*\Sigma_g$, where $\Sigma_g$ is a genus $g\ge 2$ surface \cite{mcduff1991symplectic}.
	\item\label{ex2} $T^{n}$ bundle over $T^{n-1}$ for any $n\ge 1$ \cite{geiges1994symplectic,massot2013weak}. In fact, such examples are ADC, as the contact boundaries are hypertight, meaning that there is a contact form with all Reeb orbits non-contractible.
\end{enumerate}
Those Liouville domains have vanishing first Chern class. Then by attaching Weinstein handles to them without changing the first Chern class, or taking products among them yield many Liouville domains $V^{2n}$ such that $c_1(V)=0$ and $H^{k}(V)\ne 0$ for some $k> n+1$. Therefore the boundary of $V\times \C$ admits no Weinstein filling by Corollary \ref{cor:weinstein}. One may keep attaching subcritical or flexible handles to it preserving the ADC property, then the boundary is not Weinstein fillable by Corollary \ref{rmk:weinstein}.
\end{example} 
 
Given an almost contact manifold, the existence of an almost Weinstein filling is purely homotopy theoretical, this was solved by Bowden-Crowley-Stipsicz in \cite{bowden2014topology}. The obstruction was phrased using bordism theory. Bowden-Crowley-Stipsicz used them to construct exactly fillable but not almost Weinstein fillable contact manifolds in all dimensions $\ge 5$. In the following, we show that in dimension $4k+3,k\ge 1$, our construction yields examples where the homotopy obstruction vanishes, while the obstruction in Corollary \ref{cor:ob} does not vanish. Hence the obstruction is symplectic in nature.

We first recall the criterion of almost Weinstein fillability from \cite{bowden2014topology}. An almost contact structure on a closed oriented $2n+1$ manifold $Y$ is a reduction of the structure group of $TY$ to $U(n)$. Then an almost contact structure defines a map $\zeta:Y\to BU(n)\to BU$. The $n$th Postnikov factorization of $\zeta$ consists of a space $B^{n-1}_{\zeta}$ and maps
$$Y \stackrel{\overline{\zeta}}{\to} B^{n-1}_{\zeta} \stackrel{\eta^{n-1}_{\zeta}}{\to}BU,$$
such that $\eta^{n-1}_{\zeta}$ is a fibration, $\zeta=\eta^{n-1}_{\zeta}\circ \overline{\zeta}$, the map $\overline{\zeta}$ is a $n$-equivalence, i.e. $\overline{\zeta}$ induces isomorphism on $\pi_j$ for $j<n$ and surjective map on $\pi_n$, the map $\eta^{n-1}_{\zeta}$ is a coequivalence, i.e. it induces isomorphism on $\pi_j$ for $j>n$ and injective map on $\pi_n$. Then the pair $(Y,\overline{\zeta})$ defines a bordism class $[Y,\overline{\zeta}]$ in $\Omega_{2n+1}(B^{n-1}_{\zeta};\eta^{n-1}_{\zeta})$, for the definition of this bordism group, see \cite[\S 2.1]{bowden2014topology}.
\begin{theorem}[{\cite[Theorem 1.2]{bowden2014topology}}]\label{thm:BCE}
	A closed almost contact manifold $Y$ of dimension $2n+1\ge 5$ is almost Weinstein fillable iff $[Y,\overline{\zeta}]=0\in \Omega_{2n+1}(B^{n-1}_{\zeta};\eta^{n-1}_{\zeta})$.
\end{theorem}

\begin{proof}[Proof of Theorem \ref{thm:I}]
By \cite{massot2013weak}, there exist an exact domain $V$ diffeomorphic to $M \times [0,1]$, such that $M$ is a $T^{n+2}$ bundle over $T^{n+1}$ for $n\ge 0$.  Moreover, the contact structure on each  component of $\partial V$ is a trivial complex bundle, since the contact structure is an invariant distribution on a Lie group. Hence the complex structure of $TV$ on $\partial V$ is trivial. Since $M \times\{0\} \to V$ is a homotopy equivalence, we have the complex structure on $TV$ is trivial. Let $Y'$ denote the contact boundary of $V\times \C^n$. Then by Theorem \ref{thm:product}, $Y'$ is ADC and $Y'$ is diffeomorphic to $M\times S^{2n}$.  Assume now $n\ge 1$. Let $S=\{*\}\times S^{2n}$ be a sphere on $Y'$. The $T(V\times \C^n)|_S$ is a trivial complex bundle, i.e. the structure map $S\to BU(2n+2)$ is  trivial. Let $\xi$ be the contact structure on $Y'$, since $BU(2n+1)\to BU(2n+2)$ induces isomorphism on $\pi_{2n}$, we have $\xi|_S$ is trivial $\C^{2n+1}$ bundle. Therefore there exists a homotopy  of monomorphisms $F_s:TS \to TY'$, such that $F_0$ is the inclusion and $F_1$ is an inclusion to $\xi$ and isotropic. Hence by the h-principle of isotropic embedding \cite[Theorem 7.11]{cieliebak2012stein}, $S$ is homotopic to to an isotropy sphere $S'$ with trivial $(TS')^\perp/ TS'$. Then we can attach a $2n+1$ subcritical handle to $S'$. The resulted contact manifold is $Y$ with filling $W:=V\times \C^n \cup H_{2n+1}$. 

By \cite[Theorem 3.15]{lazarev2016contact}, $Y$ is $0$-ADC. By \cite{cieliebak2002handle}, we have $SH^*(W) = 0$.  Since $H^{2n+3}(V) = \Z$, we have $Y$ does not admits Weinstein filling by Corollary \ref{rmk:weinstein}. We claim $Y$ admits almost Weinstein fillings. Since the structure map $V\times \C^n \to BU$ is trivial, the structure map $W\to BU$ is represented by a class $\pi_{2n+1}(BU)=0$. Therefore the structure map $W\to BU$ is trivial. Assume the $2n+1$th Postnikov factorization of the trivial structure map $\zeta:W \to BU$ is given by
$$W \stackrel{\overline{\zeta}}{\to} B^{2n} \stackrel{\eta^{2n}}{\to} BU.$$
We claim that $Y \hookrightarrow W  \stackrel{\overline{\zeta}}{\to} B^{2n} \stackrel{\eta^{2n}}{\to} BU$ gives the $2n+1$th of Postnikov factorization of the trivial structure map $Y \to BU$. Then $Y$ is almost Weinstein fillable by Theorem \ref{thm:BCE} once the claim is proven, as the $[W,\overline{\zeta}]$ is a bordism to $\emptyset$. By definition, it is sufficient to prove $Y\hookrightarrow W \stackrel{\overline{\zeta}}{\to}B^{2n}$ is a $2n+1$ equivalence. Since $W\to BU$ is trivial, we know that $\pi_{2n+1}(B^{2n})=0$. Therefore $Y\hookrightarrow W \stackrel{\overline{\zeta}}{\to}B^{2n}$ induces epimorphism on $\pi_{2n+1}$. Since $\overline{\zeta}$ is a $2n+1$ equivalence, it suffices to show that $\pi_*(Y)\to \pi_*(W)$ is an isomorphism for $*\le 2n$. We consider the universal cover $\widetilde{W}$ of $W$. Since $V$ is a $K(\pi,1)$, we have $\widetilde{W}$ is constructed by attaching $|\pi|$ copies $2n+1$-handles to the boundary of $\R^{2n+3}\times B^{2n+1}$, where $B^{2n+1}\subset \R^{2n+1}$ is the unit ball. Then the boundary $\widetilde{Y}:=\partial \widetilde{W}$ is connected. Since $\widetilde{W}$ can be viewed as attaching $|\pi|$ copies $2n+3$ handles and one $2n+1$-handle to $\widetilde{Y}$. Then we have $\pi_*(\widetilde{Y}) \to \pi_*(\widetilde{W})$ is isomorphism for $*\le 2n-1$. In particular $\pi_1(\widetilde{Y})=0$, i.e. $\widetilde{Y}$ is the universal cover of $Y$. Therefore it is sufficient to show prove $\pi_{2n}(\widetilde{Y}) \to \pi_{2n}(\widetilde{W})$ is an isomorphism. We will use the relative Hurewicz theorem to prove the claim. First,  we have the following commutative diagram,
$$
\xymatrix{
	H_{2n+1}(\widetilde{W}) \ar[r] \ar[d] & H_{2n+1}(\widetilde{W},\widetilde{Y})=\Z\ar[ld] \\
	H_{2n+1}(\widetilde{W},\widetilde{W}\backslash(\R^{2n+3}\times B^{2n+1})) &}
$$
By excision, we have $H_{2n+1}(\widetilde{W},\widetilde{W}\backslash(\R^{2n+3}\times B^{2n+1}))=H_{2n+1}(\R^{2n+3}\times B^{2n+1},\R^{2n+3}\times S^{2n})=\Z$, and the map from $H_{2n+1}(\widetilde{W},\widetilde{Y})$ is an isomorphism. Note that $H_{2n+1}(\widetilde{W})\to 	H_{2n+1}(\widetilde{W},\widetilde{W}\backslash(\R^{2n+3}\times B^{2n+1}))$ is surjective, since any $2n+1$ handle in $\widetilde{W}$ attached to $\R^{2n+3}\times S^{2n}$ along with $\{*\}\times B^{2n+1}$ gives a chain that is mapped to the generator of  $H_{2n+1}(\widetilde{W},\widetilde{W}\backslash(\R^{2n+3}\times B^{2n+1}))$. Hence 	$H_{2n+1}(\widetilde{W}) \to H_{2n+1}(\widetilde{W},\widetilde{Y})$ is surjective. By the homology long exact sequence for $\widetilde{W},\widetilde{Y}$, we have $H_{2n}(\widetilde{Y})\to H_{2n}(\widetilde{W})$ is an isomorphism. Note that $(\widetilde{W},\widetilde{Y})$ is $2n-1$ connected and $\widetilde{Y}$ is simply connected, then by the relative Hurewicz theorem, $\pi_{2n}(\widetilde{W},\widetilde{Y})=0$. That is $(\widetilde{W},\widetilde{Y})$ is $2n$ connected. Then the relative Hurewicz theorem implies that $\pi_{2n+1}(\widetilde{W},\widetilde{Y}) \to  H_{2n+1}(\widetilde{W},\widetilde{Y})=\Z$ is an isomorphism. As we have seen, there is a $S^{2n+1}\in \widetilde{W}$ that is mapped to the generator in $H_{2n+1}(\widetilde{W},\widetilde{Y})$. Therefore $\pi_{2n+1}(\widetilde{W}) \to \pi_{2n+1}(\widetilde{W},\widetilde{Y})$ is surjective, hence $\pi_{2n}(\widetilde{W})\to \pi_{2n}(\widetilde{Y})$  is an isomorphism. Since the obstruction in Corollary \ref{cor:ob} does not vanish for $Y$, then by Proposition \ref{prop:many}, there are infinity many exactly fillable, almost Weinstein fillable but not Weinstein fillable manifolds. 
\end{proof}
\begin{remark}
	In dimension $4n+1$ for $n\ge 2$, one may consider $V_1,V_2$ from Example \ref{ex:exact}, such that $\dim V_1+\dim V_2=2n+4$. In particular $V_1\times V_2 \simeq M_1\times M_2\times [0,1]^2$. Then we consider $Y$ to the boundary of $W$, which is constructed by attaching a $2n$-handle to $V_1\times V_2\times \C^{n-1}$. Most of the arguments in Theorem \ref{thm:I} go through, in particular, $Y$ is exactly fillable but not Weinstein fillable. The only issue is showing the $2n$th Postnikov factorization of $W$ gives a bordism to $\emptyset$. In fact, one can show that everything boils down to whether the class in $\pi_{2n}(BU)$ classifying the structure map $W\to BU$ is trivial.
\end{remark}

\begin{remark}
	In dimension $5$, manifold $V\times \C$ always admits Morse function of index at most $3$, hence Corollary \ref{cor:weinstein} can not give obstructions to Weinstein fillings. In dimension $3$, $V\times \C$ is always a subcritical Weinstein domain.
\end{remark}

\subsection{Product of ADC manifolds}
In the following, we will show the product of two ADC domains is again ADC. In particular, such construction provides many examples where Corollary \ref{cor:ob2} can be applied. Before stating the precise theorem, we first prove a proposition on the Reeb orbits on the boundary of general products. In the following, we define a function for a fixed $N \in \N$.
 $$g_N(x) = 2 - \frac{1}{N(2-x)}, \quad x \in [1,2-\frac{1}{N}].$$
 Then $g_N$ takes values in $[1,2-\frac{1}{N}]$ and $g^{-1}_N = g_N$. Let $V$ and $W$ be two Liouville domains with Liouville forms $\lambda_V,\lambda_W$ that are non-degenerate as contact forms, such that $c_1(V) = c_1(W) = 0$. Then we define $V\times_N W$ to be the following subset of $\widehat{V} \times \widehat{W}$, 
 $$V\times W \cup (\partial V \times \{r_V < f_W\})\times W \cup V\times (\partial W \times \{r_W < f_V\}) \cup \partial V \times \partial W \times \{(r_V,r_W)|r_V < g_N(r_W) \},$$
 where we have
 \begin{enumerate}
 	\item $f_V,f_W$ are Morse functions on $V$ and $W$, such that near the collar of $V$ and $W$ they are $g_N$ and $\frac{1}{f_V},\frac{1}{f_W}$ realize the Morse dimensions of $V,W$ respectively.
 	\item $2-\frac{1}{N} \le |f_V|,|f_W| \le 2 $ and $|\rd f_V|_{C^1},|\rd f_W|_{C^1} \le  \frac{1}{N}$.
 \end{enumerate}

 \begin{proposition}\label{prop:prod}
 	Assume $V,W$ as above and $\dim V,\dim W \ge 4$. For $D > 0$, for $N \in \N$ big enough, the the contractible Reeb orbits of $V\times_N W$ of period at most $D$ are of the following three types.
 	\begin{enumerate}
 		\item\label{o1} Non-degenerate orbit $(p,\gamma)$ for $p\in \cC(\frac{1}{f_V})$ and $\gamma \in \cP^{<D}(W)$ with index $ \frac{1}{2}\dim V-\ind p + \mu_{CZ}(\gamma)$.
 		\item\label{o2} Non-degenerate orbit $(\gamma,p)$ for $p\in \cC(\frac{1}{f_W})$ and $\gamma \in \cP^{<D}(V)$ with index $\frac{1}{2}\dim W-\ind p + \mu_{CZ}(\gamma)$.
 		\item\label{3} Morse-Bott orbits $(\gamma_1(t),\gamma_2(t+t_0))$ in an $S^1$ family for $\gamma_1,\gamma_2$ in $\cP^{<D}(V),\cP^{<D}(W)$ respectively up to reparametrization. The generalized Conley-Zehnder index is $\mu_{CZ}(\gamma_1)+\mu_{CZ}(\gamma_2) + \frac{1}{2}$.
 	\end{enumerate}
 \end{proposition}
 \begin{proof}
 	The boundary of $V\times_N W $ is given by three components $(\partial V \times \{r_V = f_W\})\times W \cup V\times (\partial W \times \{r_W = f_V\}) \cup \partial V \times \partial W \times \{(r_V,r_W)|r_V = g_N(r_W)\}$. Using Remark \ref{rmk:convex}, we have the Liouville vector of $\widehat{V}\times \widehat{W}$ is transverse to the boundary. Hence $\partial (V\times_N W)$ is indeed of contact type. By Proposition \ref{prop:reeb} and Lemma \ref{lemma:long}, for $N$ big enough, the Reeb orbits on $V\times (\partial W \times \{r_W = f_V\})$ and $(\partial V \times \{r_V = f_W\})\times W$ are those described in \eqref{o1} and \eqref{o2}. 
 	
 	It remains to study the Reeb orbits on $\partial V \times \partial W \times \{(r_V,r_W)|r_V = g_N(r_W)\}$. It can be either viewed as graph of function on $\widehat{W}$ or $\widehat{V}$. We choose to view it as graph on $\widehat{W}$. Then by Proposition \ref{prop:reeb}, the Reeb vector field is given by
 	$$\frac{1}{g_N(r_W)-g'_N(r_W)r_W} R_V + \frac{-g'_N(r_W)}{g_N(r_W)-g'_N(r_W)r_W} R_W.$$
 	Note that $g'_N(r_W) = \frac{r_V - 2}{2-r_W}$, we write the Reeb orbit in a more symmetric way as below,
 	\begin{equation}\label{eqn:reeb}
 	\frac{r_W-2}{2r_Vr_W-2r_V-2r_W} R_V+\frac{r_V-2}{2r_Vr_W-2r_V-2r_W} R_W.
 	\end{equation}
 	Here the coefficient varies in $[\frac{1}{2N},\frac{1}{2}]$. Therefore the Reeb orbit of \eqref{eqn:reeb} are in the form of $$(\gamma_1(\frac{r_W-2}{2r_Vr_W-2r_V-2r_W} t), \gamma_2(\frac{r_V-2}{2r_Vr_W-2r_V-2r_W} t)),$$ 
 	where $\gamma_1,\gamma_2$ are Reeb orbits of $\partial V$ and $\partial W$ with periods $D_1$ and $D_2$ such that $T:=D_1/\frac{r_W-2}{2r_Vr_W-2r_V-2r_W} = D_2/\frac{r_V-2}{2r_Vr_W-2r_V-2r_W} \ge 2D_1,2D_2$ is the period. By Remark \ref{rmk:contact}, the contact restructure is given by $\xi_V\oplus \xi_W \oplus \la \partial_{r_W}, -r_W R_V+r_V R_W \ra$. The linearized return map preserves the decomposition. Therefore the linearized return maps on $\xi_V$ and $\xi_W$ do not have $1$ as eigenvalue by non-degenerate assumptions on $V$ and $W$. By taking derivative of \eqref{eqn:reeb}, the linearized return map on the remaining part is given by 
 	\begin{equation}\label{eqn:return}
 	\left[ 
 	\begin{array}{cc}
 	1, & 0 \\
 	\frac{4-2r_V}{(2r_Vr_W-2r_V-2r_W)^2} T, & 1
 	\end{array}
 	\right]
 	\end{equation}
 	Let $\phi_T$ denote the linearized return map on the whole tangent space, then $\ker (\phi_T  - \Id)$ is spanned by the Reeb vector field \eqref{eqn:reeb} and $-r_W R_V + r_V R_W$, which is the $1$-eigenspace of \eqref{eqn:return}. Equivalently, $\ker (\phi_T  - \Id)$ is spanned by  $R_V$ and $R_W$, which is the tangent of $\Ima \gamma_1 \times \Ima \gamma_2$. Moreover $\rd (r_V \lambda_V + r_W \lambda_W)$ is rank zero on the tangent of $\Ima \gamma_1 \times \Ima \gamma_2$. Hence such Reeb orbits are of Morse-Bott type. To compute the generalized Conley-Zehnder index, since we have the decomposition of contact structure, it is sufficient to compute the contribution in the $\la \partial_{r_W}, -r_W R_V+r_V R_W \ra$ direction. Since $\dim V,\dim W \ge 4$, here the complex trivialization of the tangent bundles over the Reeb orbits splits into trivializations of $\xi_V$ and $\xi_W$ used in defining Conley-Zehnder index for $V$ and $W$ and a trivial complex bundle spanned by $\partial_{r_W}$ and $-r_W R_V + r_V R_W$.  The total generalized Conley-Zehnder index is the generalized Conley-Zehnder index of 
 	$$\left[ 
 	\begin{array}{cc}
 	1, & 0 \\
 	\frac{4-2r_V}{(2r_Vr_W-2r_V-2r_W)^2} t, & 1
 	\end{array}
 	\right],$$
 	which is $\frac{1}{2}$ plus $\mu_{CZ}(\gamma_1)+\mu_{CZ}(\gamma_2)$.
 \end{proof}

 \begin{theorem}\label{thm:prod}
 	Let $V,W$ be two Liouville domains of dimension $\ge 4$ respectively, such that $c_1(V) = c_1(W)=0$. Assume $V$ is $p$-ADC and $W$ is $q$-ADC, then we have $V\times W$ is $\min\{p+q+4,p+\dim W - \dim_M W,q+\dim V -\dim_MV\}$-ADC. In particular, if $V,W$ are both ADC, then $V\times W$ is ADC. If $V,W$ are both TADC, then $V\times W$ is TADC.
 \end{theorem}
 \begin{proof}
 	Assume the dimensions of $V$ and $W$ are $2n,2m$ respectively. Since $V,W$ are ADC, there exist positive functions $f_1>f_2>\ldots >f_k>\ldots$,  $g_1>g_2>\ldots >g_k>\ldots$ and real numbers $D_1<D_2<\ldots<\ldots D_k<\ldots$ converge to $\infty$, such that all orbits in $\cP^{<D_i}(f_i\lambda_V|_{\partial V})$ are non-degenerate and have Conley-Zehnder indices at least $p-n+4$ and all orbits in $\cP^{<D_i}(g_i\lambda_W|_{\partial W})$ are non-degenerate and have Conley-Zehnder indices at least $q-m+4$. Let $V_i:=\{r_V\le f_i\}$ and $W_i:=\{r_W\le g_i\}$, then for $N\in \N$ big enough $V_i\times_N W_i\subset \widehat{V}\times \widehat{W}$ is close to $\{r_V \le 2f_i\}\times \{r_W\le 2g_i\}$ and the Reeb orbits of period smaller than $D_i$ are decried by Proposition \ref{prop:prod}. Using \cite[Lemma 2.3,2.4]{bourgeois2002morse}, there exist a very small perturbation $(V\times W)_i$ of $V_i\times_N W_i$, such that all Reeb orbits in \eqref{3} of Proposition \ref{prop:prod} are perturbed into non-degenerate orbits with index $\mu_{CZ}(\gamma_1)+\mu_{CZ}(\gamma_2), \mu_{CZ}(\gamma_1)+\mu_{CZ}(\gamma_2)+1$. Then we have all elements in $\cP^{<D_i}((V\times W)_i)$ are non-degenerate with Conley-Zehnder indices at least $\min\{p+q-n-m+8,p-n+4-\dim_MW+m,q-m+4-\dim_MV+n\}$. Hence the degree is at least $\min\{p+q+5,p+2m-\dim_MW+1,q+2n-\dim_MV+1\}$. Since for $j>i$, we can choose $N$ in the construction big enough such that $(V\times W)_j\subset (V\times W)_i$, which implies that the contact form on $(V\times W)_j$ is smaller than that on $(V\times W)_i$. Hence $\partial(V\times W)$ is $\min\{p+q+4,p+\dim W - \dim_M W,q+\dim V -\dim_MV\}$-ADC. The TADC case is similar. $(V\times W)_i$ will not collapse, because $V_i,W_i$ do not collapse.  
 \end{proof}
 
 The following Theorem provides more examples of exactly fillable, but not Weinstein fillable contact manifolds using the obstruction in Corollary \ref{cor:ob2}.  
 \begin{corollary}
 	Let $V^{2n}$ be the Liouville manifold in \eqref{ex2} of Example \ref{ex:exact}, $W^{2m}$ be any ADC Liouville domain with dilations, see Example \ref{ex:ADCdilation}. If $n-1>m$, then $\partial(V\times W)$ can not be filled by Weinstein domains.
\end{corollary}
\begin{proof}
 By \cite{oancea2006kunneth,seidel2014disjoinable}, $\Delta(a\otimes b)=\pm \Delta(a)\otimes b \pm a\otimes \Delta(b)$ on $SH^*(V\times W) = SH^*(V)\otimes SH^*(W)$. Since those $V$ are hypertight, i.e. they allow a Reeb flow without contractible Reeb orbits. Therefore $V$ is automatically ADC. Since we only use contractible orbits to generate symplectic cohomology, we have $SH^*(V)=H^*(V)$. 	We may assume $SH^*(W)\ne 0$, otherwise Corollary \ref{cor:ob} can be applied to obtain the conclusion. Let $\alpha \in SH^{2n-1}(V)$ be a generator of  $H^{2n-1}(V)$ and $\beta \in SH^1(W)$ a dilation. Then we have $\Delta(\alpha\otimes \beta)=\pm \alpha \otimes 1 \in \Ima (H^{2n-1}(V\times W) \to SH^{2n-1}(V\times W))$. Similar to the proof of Proposition \ref{prop:dilation}, if we view $\alpha \otimes 1\in H^{2n-1}(V\times W)$, then it represents a class in the image of $\phi: \Ima(SH^*(V\times W)\to SH^*_+(V\times W))\cap \ker \Delta_+\to \coker \delta$. Since $H^*(V\times W)\to H^*(\partial(V\times W))$ is injective and $2n-1>n+m$. To show $\Ima \Delta_{\partial}$ contains a class in grading $2n-1$, it is enough to show that $\alpha\otimes 1$ is not in  $\Ima \delta$. Since $V$ is hypertight, the K{\"u}nneth formula implies that $SH^*_+(V\times W)=H^*(V)\otimes SH^*_+(W)$, and $\delta(\alpha \otimes \beta)=\pm \alpha\otimes \delta_{W}(\beta)$. Hence if $\alpha \otimes 1$ is in the image of $\delta$, then $1\in \Ima \delta_W$, which implies $SH^*(W)=0$, contradicting the assumption.
\end{proof}
\section{Obstructions to exact cobordisms}\label{s7}
A natural question in symplectic geometry is understanding the ``size" or ``complexity" of exact domains. One simple way of comparing complexity is by asking if one exact domain can be embedded into another. Due to the Viterbo transfer map, the vanishing of symplectic cohomology and existence of symplectic dilation are two levels of complexity, which are in fact the first two of the infinity many structures in \cite{higher}.  For example, one can not embed an exact domain $W$ with $SH^*(W)\ne 0$ to a flexible Weinstein domain. In particular,  flexible Weinstein domain does not contain closed exact Lagrangians. On the other hand, one can always embed the standard ball into any domain. Moreover, one can not embed an exact domain $W$ with no dilation into $T^*S^n$. In particular $T^*S^n$ contains no exact tori. 

The same question can be asked for contact manifolds, and the comparison is based on existence of symplectic cobordism. By putting different adjectives in front of the symplectic cobordism, we get several comparisons in different flavors. In this section, we will restrict to exact cobordisms. Similar to the discussion above, one can use the functoriality of contact invariants like contact homology and symplectic field theory \cite{eliashberg2000introduction} to study this problem. For example, an overtwisted contact contact manifold has vanishing contact homology \cite{yau2004vanishing}, hence there is no exact cobordism from a contact manifold with non-vanishing contact homology to an overtwisted one. In particular, there is no cobordism from $\emptyset$. A more general obstruction in a similar spirit using the full SFT was constructed in \cite{latschev2011algebraic}. However, unlike symplectic cohomology, such invariants are difficult to define \cite{hofer2017application,pardon2019contact} and notoriously hard to compute.  

Since we have shown that for ADC manifolds, the vanishing of symplectic cohomology and existence of symplectic dilation is independent of the filling, hence can be understood as contact invariants. In particular, we can use them to define obstructions to the existence of exact cobordisms to some ADC manifolds. Such obstructions using symplectic cohomology are relatively easy to define and compute.  
\begin{theorem}\label{thm:corb}
	Let $Y$ be an ADC contact manifold with a topologically simple exact filling $W$. Let $V$ be an exact domain with $c_1(V)=0$. If one of the following holds, then there is no exact cobordism $U$ from $\partial V$ to $Y$, such that $c_1(U)=0$ and $H^1(V)\oplus H^1(U) \mapsto H^1(\partial V), (a,b)\mapsto a|_{\partial V}-b|_{\partial V}$ is surjective and $\pi_1(Y) \to \pi_1(U\cup V)$ is injective.
	\begin{enumerate}
		\item\label{corb1} If $1\in \Ima \delta_{\partial}$ for $W$, and $1\notin \Ima \delta_{\partial}$ for $V$.
		\item\label{corb2} If $1\in \Ima \Delta_{\partial}$ for $W$, and $1\notin \Ima \Delta_{\partial}$ for $V$.
	\end{enumerate} 
\end{theorem}
\begin{proof}
	If there is a such cobordism. Since $H^1(V)\oplus H^1(U)\to H^1(\partial V)$ is surjective, we have $H^2(U\cup V) \to H^2(U)\oplus H^2(V)$ is injective. Since $c_1(V)=c_1(U)=0$, we have $c_1(U\cup V)=0$. Therefore $U\cup V$ is a topologically simple filling. Since $1\in \Ima \delta_{\partial}$  for $W$, then $SH^*(W)=0$. In particular, by Corollary \ref{cor:B}, we have $SH^*(U\cup V)=0$.  That $1\notin \Ima \delta_{\partial}$ for $V$ implies that $SH^*(V)\ne 0$, we have a contradiction by that the Viterbo transfer map preserves the unit, hence there is no such $U$. 
	
	When $1\in \Ima \Delta_{\partial}$ for $W$,  we know that  $\phi:\ker\Delta_+ \to \coker \delta$ contains $1$ in the image for $W$. Assume there is a such cobordism $U$. By Theorem \ref{thm:E}, we have that $1\in \Ima \Delta_{\partial},\Ima \phi$  for $U\cup V$. If  $1\notin \Ima \Delta_{\partial}$ for $V$, we know that $1\notin \Ima \phi$ for $V$. The Viterbo transfer map commutes with $\Delta_+$, $\delta$ and $\phi$ by Proposition \ref{prop:viterbo}, hence the Viterbo transfer map induces the following commutative diagram
	$$
	\xymatrix{
	\ker \Delta_+\subset SH^*_+(U\cup V) \ar[rr]^{\phi\hspace{1cm}} \ar[d]^{\text{Viterbo transfer}} && \coker(SH_+^{*-2}(U\cup V)\to H^{*-1}(U\cup V))\ar[d]^{\text{Viterbo transfer}}\\
	\ker \Delta_+\subset SH^*_+( V) \ar[rr]^{\phi\hspace{1cm}} && \coker(SH_+^{*-2}(V)\to H^{*-1}(V))
    }	
	$$
	Hence we have a contradiction. 
\end{proof}

\begin{proof}[Proof of Corollary \ref{cor:J}]
Assume otherwise, there is a Weinstein cobordism $U$, then $U\cup V$ is another Weinstein filling of $Y$. In particular, $\pi_1(Y)\to \pi_1(U\cup V)$ is an isomorphism and $c_1(U\cup V)=0$. Then we reach a contradiction by the same argument in the proof of Theorem \ref{thm:corb} for the first two conditions. The third one follows from Corollary \ref{cor:F}.
\end{proof}

As a direct application of Corollary \ref{cor:J}, if an ADC Weinstein domain $V$ of dimension $\ge 6$ admits a dilation and $c_1(V)=0$. Then not only $V$ does not contain $K(\pi,1)$ as closed exact Lagrangian, there is no Weinstein cobordism form $S^*K(\pi,1)$ to $\partial V$, since $T^*K(\pi,1)$ has no dilation.

On the other hand, the existence of almost Weinstein cobordism is purely homotopical and was studied in \cite{bowden2015topology}. In particular, Bowden-Crowley-Stipsicz \cite[Theorem 1.2]{bowden2015topology} showed that for dimension $2n-1\ge 5$, there exists an almost contact manifold $(M_{\max},\phi_{\max})$ such that for any almost contact manifold $(M,\phi)$ there is an almost Weinstein cobordism from $(M,\phi)$ to $(M_{\max},\phi_{\max})$. The maximal element also exists when restricted to the class of almost contact manifolds with vanishing first Chern class. In the latter case, when dimension is $5$ or $7$, the maximal element can be chosen as the standard $(S^{2n-1},\phi_{std})$. If we take the maximal element $(M_{\max},\phi_{\phi})$ for the class of contact manifolds with vanishing Chern class, since there is a Weinstein cobordism from $\emptyset$ to  $(M_{\max},\phi_{\max})$. In particular, in the homotopy class of $(M_{\max},\phi_{\max})$, there is a contact manifold $(M,\xi)$ admitting a flexible filling. Then for any contact manifold $Y$ with vanishing first Chern class and a Weinstein filling $V$ such that $SH^*(V)\ne 0$. There is an almost Weinstein cobordism from $Y$ to $M$, but there is no Weinstein cobordism by Corollary \ref{cor:J}.

There are also many examples with non-vanishing symplectic cohomology. For example one takes the Milnor fiber $W$ of the singularity $z_1^9+z_2^2+z_3^2+z_4^2=0$, then $\partial W$ as an almost contact manifold is $(S^5,\phi_{std})$. By \cite[Example 2.13]{seidel2014disjoinable}, we have $SH^*(W)\ne 0$ and admits a dilation and $\partial W$ is ADC by \cite[Lemma 4.2]{ustilovsky1999infinitely}.  Hence by Corollary \ref{cor:J}, for any Weinstein domain $V$ with $c_1(V)=0$ and no dilation. There is an almost Weinstein cobordism from $\partial V$ to $\partial W$, but there is no Weinstein cobordism from $\partial V$ to $\partial W$. In higher dimensions, let $W$ be an ADC Weinstein domain with non-vanishing symplectic cohomology and a symplectic dilation. Let $V$ be an exotic $\C^n$ with non-negative log Kodaira dimension, in particular $V$ has no symplectic dilations and the almost contact structure on $\partial V$ is the standard almost contact structure on $S^{2n-1}$. Then there is no Weinstein cobordism from $\partial V$ to $\partial W$, but there is an almost Weinstein cobordism from $\partial V=(S^{2n-1},\phi_{std})$ to $\partial W$ as $W$ is almost Weinstein fillable.
	
\section{Nonexact fillings}\label{s8}
In this section, we consider a strong filling $W$ of a contact manifold $Y$, such that $c_1(W) = 0$. Floer cohomology is still well-defined using the Novikov field $\Lambda = \{\sum_{i=1}^\infty q_i T^{\lambda_i}\st q_i \in \Q, \lambda_i\in \R, \lim \lambda_i = \infty\}$ with a $\Z$-grading, c.f. \cite{hofer1995floer}, similarly for symplectic cohomology \cite{ritter2014floer}. In particular sphere bubbles can be avoided by dimension reasons. In this section we show that positive symplectic cohomology is still defined to the extent that constructions in \S \ref{s2}, \S \ref{s3} can be generalized to strong fillings to finish the proof of Theorem \ref{thm:D}. 

\subsection{Positive symplectic cohomology}
In the nonexact case, the action functional \eqref{eqn:action} is not well-defined. In particular, there is no action separation of the symplectic cohomology into zero length part and positive length part. However, positive symplectic cohomology can still be defined due to the following lemma by Bourgeois-Oancea \cite[Proof of Proposition 5]{bourgeois2009exact}. We first recall it from \cite{cieliebak2018symplectic}. In this section, by strong filling of a contact manifold $(Y,\alpha)$, we mean a symplectic manifold $(W,\omega)$ with collar neighborhood of the boundary symplectomorphic to $(Y\times [1-\delta,1], \rd(r\alpha))$ for some $\delta>0$. The completion $\widehat{W}$ is defined to be $W\cup Y\times (1,\infty)$, with a symplectic form $\widehat{\omega}$ such that $\widehat{\omega}|_W=\omega,\widehat{\omega}|_{Y\times (1,\infty)}=\rd(r\alpha).$

\begin{lemma}{\cite[Lemma 2.3]{cieliebak2018symplectic}}\label{lemma:asm}
	Let $H=h(r)$ be a Hamiltonian on the symplectization $(Y\times \R_+,\rd(r\alpha))$. Let $u:\R_-\times S^1 \to Y\times \R_+$ be a finite energy solution to the Floer equation and $r_0:=\lim_{s\to -\infty} r\circ u(s,t)$. Assume $h''(r_0)>0$, $J$ is cylindrical convex on a neighborhood of $\overline{u(\R_-\times S^1)}$, then either there exists $(s_0,t_0)$ such that $r\circ u(s_0,t_0) > r_0$ or $r\circ u\equiv r_0$.
\end{lemma}

The combination of Lemma \ref{lemma:max} and Lemma \ref{lemma:asm} yields the following.
\begin{proposition}\label{prop:exclude}
Consider a Hamiltonian $H:\widehat{W}\to \R$, such that $H=0$ on $r\le r_0> 1$ and $H=h(r)$ for $r\ge r_0$ such that when $h'(r)\in \cS$ we have $h''(r)>0$, assume $J$ is cylindrical convex near all non-constant periodic orbits of $X_H$. Then there is no Floer solution $u$, such that $\displaystyle \lim_{s\to -\infty}u\in \cP^*(H)$ and $\displaystyle\lim_{s\to \infty}u\in W\cup Y\times (1,r_0)$.     
\end{proposition}

Motivated by Proposition \ref{prop:exclude}, instead of using the perturbed Hamiltonian $\bH$ in \S \ref{s2}. We will use the autonomous Hamiltonian before perturbation, denote the Hamiltonian by $\bF$. Then it has the following property,
\begin{enumerate}
	\item $\bF = 0$ on $W$,
	\item $\bF = h(r)$ on $\partial W \times (1,\infty)$ with $h''(r)>0$.
	\item The non-constant periodic orbits of $\bF$ are contained in levels $1<r_1<r_2<\ldots$ we will use $W^i$ to denote the domain inside $r=r^i$.
\end{enumerate}

Then by Proposition \ref{prop:exclude}, there are no Floer cylinders with negative end asymptotic to a non-constant orbit and positive end asymptotic to a point in $W$. In fact, the following holds.

\begin{proposition}\label{prop:increase}
	Using $\bF$ as the Hamiltonian, if there is a non-trivial Floer cylinder $u$ with the negative end asymptotic to an orbit inside $r=r_i$ and with the  positive end asymptotic to an orbit inside $r=r_j$, then $i<j$.
\end{proposition}
\begin{proof}
	If $i>j$, then the combination of Lemma \ref{lemma:max} and Lemma \ref{lemma:asm} yields a contradiction. If $i=j$, then by Lemma \ref{lemma:max} and Lemma \ref{lemma:asm}, we have $u$ is contained in $r=r_i$. Then the energy of $u$ must be zero, contradicting that $u$ is non-trivial.
\end{proof}

Using $\bF$ to define symplectic cohomology will result in a more Morse-Bott situation than $\bH$, we will use cascades to deal with degenerate $S^1$-family obits, such construction was studied in \cite{bourgeois2009exact,bourgeois2009symplectic} and more degenerate Morse-Bott cases were studied in \cite{bourgeois2002morse,diogo2019symplectic}. We will follow their constructions.  

\subsubsection{Notations and setups}
We first fix the following notations. 
\begin{enumerate}
	\item First note that every Reeb orbit of $R_{\alpha}$ corresponds to a $S^1$-family of periodic orbits of $X_{\bF}$. The set of non-constant (contractible) periodic orbits of $\bF$ has a decomposition into the union $\cup_{x\in \cP(\alpha)} S_x$, each $S_x$ can be viewed as an embedded circle in $\widehat{W}$. For each $S_x$, we fix a metric $g_x$ and a Morse function $f_x$ with a unique local maximum $\hat{x}$ and a unique local minimum $\check{x}$. With a little abuse of language, we denote $\cP^*(\bF):=\cup_x \{\hat{x},\check{x}\}$.
	\item For each $x\in \cP(\alpha)$, we fix a disk $u_x:\D \to Y$ extending one $\gamma\in S_x$ up to homotopy, which give extension to any element in $S_x$ by rotation. 
	\item $\Lambda := \{\sum_{i=1}^\infty q_i T^{\lambda_i}\st q_i \in \Q, \lambda_i\in \R, \lim \lambda_i = \infty\}$ is the Novikov field and $\Lambda_{0} := \{\sum_{i=1}^\infty q_i T^{\lambda_i}\st q_i \in \Q, \lambda_i\ge 0, \lim \lambda_i = \infty\}$ is the Novikov ring. Note that $\Lambda$ is the fraction field of the integral domain $\Lambda_{0}$. In particular $\Lambda$ is a flat $\Lambda_{0}$-module.
	\item We formally define $\cA_{\bF}(T^{a}\hat{x})=\cA_{\bF}(T^a\check{x})=a-\int_{S^1}\gamma^*\widehat{\lambda}+\int_{S^1}\bF\circ \gamma \rd t$, for $\gamma \in S_x$. And $\cA_{\bF}(T^ap)=a$ for $p\in \cC(f)$, where $f$ is an admissible Morse function on $W$.
\end{enumerate}

\subsubsection{Moduli spaces and cochain complexes}
In view of Proposition \ref{prop:exclude}, we need to modify the definition of admissible almost complex structures to the following. We first fix $\epsilon_i>0$ such that $r_i+\epsilon_i<r_{i+1}-\epsilon_{i+1}$ and $\epsilon_0>0$ such that $1+\epsilon_0<r_1-\epsilon_1$.
\begin{definition}
	A time-dependent almost complex structure $J:S^1\to \End(T\widehat{W})$ is admissible iff the following holds.
	\begin{enumerate} 
		  \item $J$ is compatible with $\widehat{\omega}$ on $\widehat{W}$.
		  \item $J$ is cylindrical convex on $\partial W \times (r_i-\epsilon_i,r_i+\epsilon_i)$ and on $(1-\epsilon_0,1+\epsilon_0)$.
		  \item $J$ is $S^1$-independent on $W$.
		  \end{enumerate}
	The class of admissible almost complex structure is again denoted by $\cJ(W)$
\end{definition}

Let $J$ be an admissible almost complex structure. First note that for any $t\in S^1$, any nontrivial closed $J_t$-holomorphic curve must be contained in $W^0$, where in particular $J_t$ does not depend $t$. This is because the curve can not be contained outside $W$ by exactness there and then we apply Lemma \ref{lemma:max} to $r=1-\epsilon_0$. We will be interested in the following uncompactified moduli spaces.
\begin{enumerate}
	\item\label{moduli1} $S_{A}(J):=\{u:\CP^1 \to \widehat{W}|  \overline{\partial}_Ju=0, [u]=A, u \text{ is simple}\}$ for $A\in H_2(W)$. It has an evaluation map $ev_0(u)=u(0)$.
	\item\label{moduli2} $M_{n,A}(J):=\{(u_1,\ldots,u_n) | u_i:\CP^1\to W, \overline{\partial}_Ju_i=0, \sum[u_i]=A, u_i \text{ is simple, and } u_i(\CP^1)\ne u_j(\CP^1), u_i(\infty)=u_{i+1}(0)\}$, it is equipped with two evaluation maps $ev_0((u_i))=u_1(0)$ and $ev_\infty((u_i))=u_n(\infty)$.
	\item\label{moduli3} $M_{S_x,S_y,A}(J):=\{u:\R\times S^1 \to \widehat{W}|\partial_su+J(\partial_t-X_{\bF})=0, \displaystyle \lim_{s\to -\infty} u \in S_x, \lim_{s\to +\infty} u \in S_y, [u\#(-u_y)\#u_x]=A \}$. It is equipped with evaluation maps $ev_-,ev_+$ to $S_x,S_y$ respectively
	\item \label{moduli4} $B_{S_x,A}(J):=\{u:\C \to \widehat{W}|\partial_su+J(\partial_t-X_{\bF})=0, u(0)\in W^0, \displaystyle \lim_{s\to +\infty} u \in S_x, [u\#(-u_x)]=A \}$. It is equipped with evaluations maps $ev_+,ev_0$ to $S_x,W$ respectively. 
\end{enumerate} 
The energy $E(u):=\frac{1}{2}\int_{\R\times S^1} ||\rd u-X_{\bF}\rd t||^2\rd s \wedge \rd t$ of $u\in M_{S_x,S_y,A}$ is given by $\omega(A)+\cA_{\bF}(S_x)-\cA_{\bF}(S_y)$. The energy for $u\in B_{s_x,A}$ is given by $\omega(A)-\cA_{\bF}(S_x)$.

To study the regularization of \eqref{moduli1}-\eqref{moduli4}, we need to consider the following universal moduli spaces. 
\begin{enumerate}
	\item Let $\bm{x}:=(x_{1}, \ldots, x_{n+1}\in \cP(\alpha)),\bm{A}:=(A_1,A_2, \ldots, A_n\in H_2(W))$ and $B\in H_2(W)$, we define $\cU^l_{\bm{x},\bm{A},B}$ be the set of $(u_1,\ldots, u_n,u,p,J)$ such that 
	\begin{enumerate}
		\item $J$ is a $C^l$ admissible almost complex structure; 
		\item $u_i\in M_{S_{x_i},S_{x_{i+1}},A_i}(J)$;
		\item $u\in S_B(J)$;
		\item $p$ is a point on domains of $u_i$.
	\end{enumerate}
	Then $\cU^l_{\bm{x},\bm{A},B}$ is equipped with an evaluation map $$EV:=ev_p\times ev_0(u) \times \prod_i ev_-(u_i)\times ev_+(u_i) \in \widehat{W}\times\widehat{W} \times S_{x_1}\times (S_{x_2})^2 \times \ldots \times (S_{x_n})^2\times S_{x_{n+1}}.$$
	\item  $\bm{x}:=(x_{1}, \ldots, x_{n}\in \cP(\alpha)),\bm{A}:=( A_1,A_2, \ldots, A_n\in H_2(W))$ and $B\in H_2(W)$, we define $\cU^l_{\bm{x},\bm{A},B,k}$ be the set of $(u_1,\ldots, u_n,u,J)$ such that 
	\begin{enumerate}
		\item $J$ is a $C^l$ admissible almost complex structure;
		\item $u_i\in M_{S_{x_{i-1}},S_{x_{i}},A_i}(J)$ if $i>1$;
		\item $u_1\in B_{S_{x_1},A_1}(J)$;
		\item $u\in M_{k,B}(J)$;
	\end{enumerate}
	Then we define $EV$ on $\cU^l_{\bm{x},\bm{A},B,k}$ by 
	\begin{align*}
	ev_0(u)\times ev_\infty(u)\times ev_0(u_1)\times ev_+(u_1) \times & \prod_{i>1} (ev_-(u_i)\times ev_+(u_i))  \\ & \in \widehat{W} \times \widehat{W} \times \widehat{W} \times(S_{x_1})^2 \times \ldots \times (S_{x_{n-1}})^2 \times S_{x_n}.
	\end{align*}
	\item Let $\bm{x}:=(x_{1}, \ldots, x_{n}\in \cP(\alpha)),\bm{A}=( A_1,A_2, \ldots, A_n\in H_2(W))$ and $B\in H_2(W)$, we define $\cV^l_{\bm{x},\bm{A},B}$ be the set of $(u_1,\ldots, u_n,u,p,J)$ such that
	\begin{enumerate}
		\item $J$ is a $C^l$ admissible almost complex structure; 
		\item $u_i\in M_{S_{x_{i-1}},S_{x_{i}},A_i}(J)$ for $i>1$;
		\item $u_1\in B_{S_{x_1},A_1}(J)$;
		\item $u\in S_B(J)$;
		\item $p$ is a point on domains of $u_i$.
	\end{enumerate}
    Then we define $EV$ on $\cV^l_{\bm{x},\bm{A},B}$ by 
    \begin{align*}
    ev_p\times ev_0(u) \times ev_0(u_1) \times ev_+(u_1)\times & \prod_{i>1} (ev_-(u_i)\times ev_+(u_i))\\ 
    & \in \widehat{W}\times\widehat{W} \times \widehat{W} \times (S_{x_1})^2 \times \ldots \times (S_{x_{n-1}})^2 \times S_{x_n}.
    \end{align*}
\end{enumerate}
The point of considering such moduli spaces is that they are cut out transversely, moreover the evaluation map is a submersion. This will allow us to perturb $J$ only without changing $(f,g)$ and $(f_x,g_x)$. The following proposition is derived from the same argument in \cite[\S 3.4]{mcduff2012j}.
\begin{proposition}\label{prop:transverseCY}
	The three universal moduli spaces considered above are cut out transversely as Banach manifolds and $EV$ are submersive.
\end{proposition}
\begin{proof}[Sketch of the proof]
	Note that by Proposition \ref{prop:increase}, for $\bm{x}=(x_1,\ldots,x_n)$ in the universal moduli spaces, the $r$-coordinate must be strictly decreasing. Therefore the objects we consider in the universal moduli spaces are somewhere invective (not just component-wise, but as a whole object). Then the claim follows from the proof of \cite[Proposition 3.4.2]{mcduff2012j} and the proof of \cite[Proposition 3.5]{bourgeois2009symplectic}\footnote{\cite{bourgeois2009symplectic} did not claim the evaluation maps to $S_x$ is submersive, but the proof of \cite[Propoistion 3.5]{bourgeois2009symplectic} implies the fact. Since in the universal moduli space, the surjectivity of the linearized operator holds without using the tangent of $S_x$. }.
\end{proof}
The cochain complex model we will use is the cascades construction \cite{bourgeois2009exact,bourgeois2009symplectic}, in particular the cochain complex will be generated by $\cC(f)$ and $\cP^*(\bF)$ as before using the Novikov field $\Lambda$. We recall the cascades moduli spaces first. Let $\overline{x}$ denote an element in $\{\hat{x},\check{x}\}$. 
\begin{enumerate}
	\item $M_{\overline{x},\overline{y},A,m}$ denotes the set of $m$-cascades, i.e. tuples
	$$(u_1,l_{1,2},u_2,\ldots, l_{m-1,m}, u_m),$$
	where 
	\begin{enumerate}
		\item $u_i\in M_{S_{x_i},S_{x_{i+1}},A_i}/\R$ with $x_1=x, x_{m+1}=y$, and $\sum A_i=A$;
		\item $l_{i,i+1}>0$ and $\phi_{x_{i+1},l_{i,i+1}}ev_+(u_i)=ev_-(u_{i+1})$, where $\phi_{x_{i+1},t}$ is the \textit{negative} gradient flow on $S_{x_{i+1}}$;
		\item $\displaystyle \lim_{t\to -\infty}\phi_{x_1,t} ev_-(u_1)=\overline{x}$ and $\displaystyle \lim_{t\to \infty}\phi_{x_{m+1},t} ev_+(u_m)=\overline{y}.$
	\end{enumerate}
	\item For $p\in \cC(f)$, $M_{p,x,A,m}$ is defined similarly, except the first curve $u_1$ is in $B_{S_{x_1},A_1}$ with $u_1(0)$ in the stable manifold of $p$. 
\end{enumerate}
For every $x\in \cP(\alpha)$, we can assign the gradings $|\hat{x}|:= n-\mu_{CZ}(x)$ and $|\check{x}|:=n - 1 -\mu_{CZ}(x)$.  As in \S \ref{s2}, we fix a Morse-Smale pair $(f,g)$ and $\cM_{p,q}$ is a Morse moduli spaces for $p,q\in \cC(f)$. Let we denote $\cM_{\overline{x},\overline{y},A}=\cup_{m\ge 0} M_{\overline{x},\overline{y},A,m}$ and $\cM_{p,\overline{x},A}=\cup_{m\ge 0}M_{p,\overline{x},A,m}$. From Proposition \ref{prop:transverseCY}, we have the following transversality and compactness result.
\begin{proposition}\label{prop:regularCY}
	There exist a second category set $\cJ_{reg}(f,g)$ such that $\cM_{\overline{x}/p,\overline{y},A}$ is compact smooth manifold of dimension $|\overline{x}/q|-|\overline{y}|-1$ whenever it is $\le 0$.
\end{proposition}
\begin{proof}
	On every $S_x$, there is submanifold $H_x:=\{(x,\phi_t(x))|t>0\}\subset S_x\times S_x$, where $\phi_t$ is the negative gradient flow on $S_x$. Let $S_p,U_p$ be the stable manifold and unstable manifold of $p\in \cC(f),\cC(f_x)$. Since $EV$ on universal moduli spaces are submersive, then the following space are Banach manifolds
	\begin{enumerate}
		\item\label{m1} $EV_{\cU^l_{\bm{x},\bm{A},B}}^{-1}(\Delta_{\widehat{W}}\times S_{\overline{x}_1}\times \prod_{i>1} H_{x_i}(\text{ or }\Delta_{S_{x_i}} )\times U_{\overline{x}_{n+1}})$;
		\item $EV_{\cU^l_{\bm{x},\bm{A},B,k}}^{-1}(S_p\times \Delta_{\widehat{W}} \times \prod_{i<n} H_{x_i}(\text{ or }\Delta_{S_{x_i}} )\times U_{\overline{x}_n})$, for $p\in \cC(f)$;
		\item\label{m3} $EV_{\cV^l_{\bm{x},\bm{A},B}}^{-1}(\Delta_{\widehat{W}}\times S_p \times \prod_{i<n} H_{x_i}(\text{ or }\Delta_{S_{x_i}} )\times U_{\overline{x}_n})$, for $p\in \cC(f)$.
	\end{enumerate}
	Then we can pick any $J$ in the regular value of the projections of above spaces to the space of $C^l$ admissible almost complex structure. Then transversality on $\cM_{*,*,A}$ is verified. It is sufficient to prove compactness. First by a dimension argument, it can not have a Morse breaking, or a fiber product breaking at $S_x$ when $|\overline{x}/p|-|\overline{y}|\le 1$, by the transversality of the preimage of $\Delta_{S_{x_i}}$  in \eqref{m1}-\eqref{m3} above.  For the general case in the compactification, we still need to consider is the following
	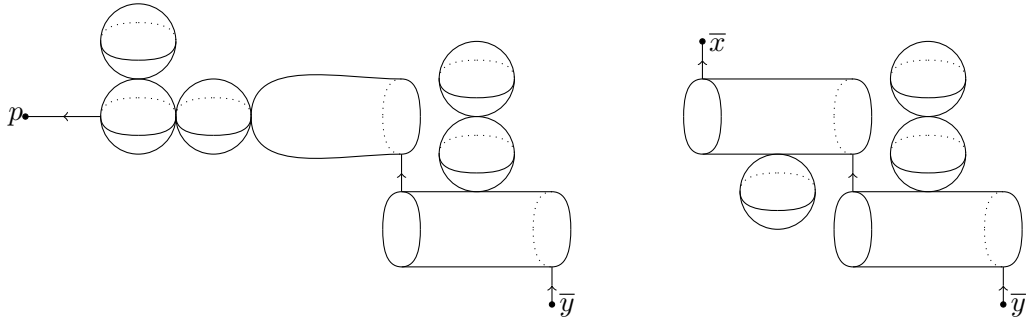
\begin{figure}[H]
		\begin{tikzpicture}
		\draw (0,0)--(0.5,0);
		\draw[<-] (0.5,0)--(1,0);
		\draw (0,0) circle[radius=1pt];
		\fill (0,0) circle[radius=1pt];
		\node at (-0.13,0) {$p$};
		
		\draw (1.5,0) circle [radius=0.5];
		\draw (1,0) to [out=270,in=180] (1.5,-0.25);
		\draw (1.5,-0.25) to [out=0,in=270] (2,0);
		\draw[dotted] (1,0) to [out=90,in=180] (1.5,0.25);
		\draw[dotted] (1.5,0.25) to [out=0,in=90] (2,0);
		
		\draw (2.5,0) circle [radius=0.5];
		\draw (2,0) to [out=270,in=180] (2.5,-0.25);
		\draw (2.5,-0.25) to [out=0,in=270] (3,0);
		\draw[dotted] (2,0) to [out=90,in=180] (2.5,0.25);
		\draw[dotted] (2.5,0.25) to [out=0,in=90] (3,0);
		
		\draw (1.5,1) circle [radius=0.5];
		\draw (1,1) to [out=270,in=180] (1.5,0.75);
		\draw (1.5,0.75) to [out=0,in=270] (2,1);
		\draw[dotted] (1,1) to [out=90,in=180] (1.5,1.25);
		\draw[dotted] (1.5,1.25) to [out=0,in=90] (2,1);
		
		\draw (3,0) to [out=90, in = 180] (5,0.5);
		\draw (3,0) to [out=270, in = 180] (5,-0.5);
		
		\draw[dotted] (5,0.5) to [out=180, in=90] (4.75, 0);
		\draw[dotted] (4.75, 0) to [out=270, in=180] (5,-0.5);
		\draw (5,0.5) to [out=0, in = 90] (5.25,0);
		\draw (5.25,0) to [out = 270, in = 0] (5,-0.5);
		
		\draw (5,-0.5)--(5,-0.75);
		\draw[<-] (5,-0.75)--(5,-1);
		
		\draw (5,-1) -- (7,-1);
		\draw (5,-2) -- (7,-2);
		
		\draw (5,-1) to [out=180, in=90] (4.75, -1.5);
		\draw (4.75, -1.5) to [out=270, in=180] (5,-2);
		\draw (5,-1) to [out=0, in = 90] (5.25,-1.5);
		\draw (5.25,-1.5) to [out = 270, in = 0] (5,-2);
		
		\draw[dotted] (7,-1) to [out=180, in=90] (6.75, -1.5);
		\draw[dotted] (6.75, -1.5) to [out=270, in=180] (7,-2);
		\draw (7,-1) to [out=0, in = 90] (7.25,-1.5);
		\draw (7.25,-1.5) to [out = 270, in = 0] (7,-2);
		
		\draw (7,-2)--(7,-2.25);
		\draw[<-] (7,-2.25)--(7,-2.5);
		
		\draw (7,-2.5) circle[radius=1pt];
		\fill (7,-2.5) circle[radius=1pt];
		\node at (7.2,-2.5) {$\overline{y}$};
		
		\draw (6,-0.5) circle [radius=0.5];
		\draw (5.5,-0.5) to [out=270,in=180] (6,-0.75);
		\draw (6,-0.75) to [out=0,in=270] (6.5,-0.5);
		\draw[dotted] (5.5,-0.5) to [out=90,in=180] (6,-0.25);
		\draw[dotted] (6,-0.25) to [out=0,in=90] (6.5,-0.5);
		
		\draw (6,0.5) circle [radius=0.5];
		\draw (5.5,0.5) to [out=270,in=180] (6,0.25);
		\draw (6,0.25) to [out=0,in=270] (6.5,0.5);
		\draw[dotted] (5.5,0.5) to [out=90,in=180] (6,0.75);
		\draw[dotted] (6,0.75) to [out=0,in=90] (6.5,0.5);
		
		\draw[->] (9,0.5)--(9,0.75);
		\draw (9,0.75)--(9,1);
		
		\draw (9,1) circle[radius=1pt];
		\fill (9,1) circle[radius=1pt];
		\node at (9.2,1) {$\overline{x}$};

		\draw (9,0.5) -- (11,0.5);
		\draw (9,-0.5) -- (11,-0.5);
		
		\draw (9,0.5) to [out=180, in=90] (8.75, 0);
		\draw (8.75, 0) to [out=270, in=180] (9,-0.5);
		\draw (9,0.5) to [out=0, in = 90] (9.25,0);
		\draw (9.25,0) to [out = 270, in = 0] (9,-0.5);
		
		\draw[dotted] (11,0.5) to [out=180, in=90] (10.75, 0);
		\draw[dotted] (10.75, 0) to [out=270, in=180] (11,-0.5);
		\draw (11,0.5) to [out=0, in = 90] (11.25,0);
		\draw (11.25,0) to [out = 270, in = 0] (11,-0.5);

		\draw (10,-1) circle [radius=0.5];
		\draw (9.5,-1) to [out=270,in=180] (10,-1.25);
		\draw (10,-1.25) to [out=0,in=270] (10.5,-1);
		\draw[dotted] (9.5,-1) to [out=90,in=180] (10,-0.75);
		\draw[dotted] (10,-0.75) to [out=0,in=90] (10.5,-1);
		
		\draw (11,-0.5)--(11,-0.75);
		\draw[<-] (11,-0.75)--(11,-1);

		\draw (11,-1) -- (13,-1);
		\draw (11,-2) -- (13,-2);
		
		\draw (11,-1) to [out=180, in=90] (10.75, -1.5);
		\draw (10.75, -1.5) to [out=270, in=180] (11,-2);
		\draw (11,-1) to [out=0, in = 90] (11.25,-1.5);
		\draw (11.25,-1.5) to [out = 270, in = 0] (11,-2);
		
		\draw[dotted] (13,-1) to [out=180, in=90] (12.75, -1.5);
		\draw[dotted] (12.75, -1.5) to [out=270, in=180] (13,-2);
		\draw (13,-1) to [out=0, in = 90] (13.25,-1.5);
		\draw (13.25,-1.5) to [out = 270, in = 0] (13,-2);
		
		\draw (13,-2)--(13,-2.25);
		\draw[<-] (13,-2.25)--(13,-2.5);
		
		\draw (13,-2.5) circle[radius=1pt];
		\fill (13,-2.5) circle[radius=1pt];
		\node at (13.2,-2.5) {$\overline{y}$};
		
		\draw (12,-0.5) circle [radius=0.5];
		\draw (11.5,-0.5) to [out=270,in=180] (12,-0.75);
		\draw (12,-0.75) to [out=0,in=270] (12.5,-0.5);
		\draw[dotted] (11.5,-0.5) to [out=90,in=180] (12,-0.25);
		\draw[dotted] (12,-0.25) to [out=0,in=90] (12.5,-0.5);
		
		\draw (12,0.5) circle [radius=0.5];
		\draw (11.5,0.5) to [out=270,in=180] (12,0.25);
		\draw (12,0.25) to [out=0,in=270] (12.5,0.5);
		\draw[dotted] (11.5,0.5) to [out=90,in=180] (12,0.75);
		\draw[dotted] (12,0.75) to [out=0,in=90] (12.5,0.5);
		\end{tikzpicture}	
		\caption{The general curves, the arrows stand for gradient flow.}\label{fig:general}
	\end{figure}
	Note that for the figure in the left, the link of spheres connecting Morse trajectory and Floer cylinder should be understood as Floer cylinders, hence they are $J$-curves without quotienting $S^1$. We will not emphasize the difference, since we will rule out such configuration altogether. Note that the existence of curves above in the compactification of $\cM_{*,*,A}$ will lead to existence of one of the following configurations.
	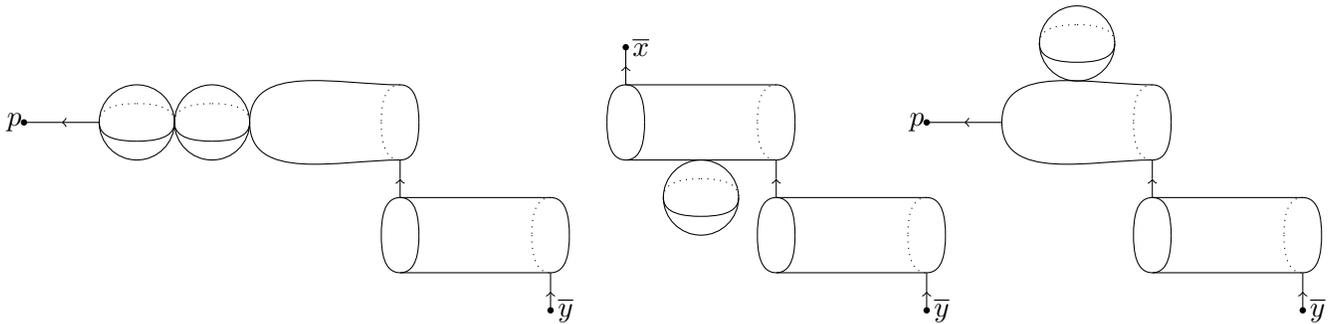
\begin{figure}[H]
	\begin{tikzpicture}
	\draw (1,0)--(1.5,0);
	\draw[<-] (1.5,0)--(2,0);
	\draw (1,0) circle[radius=1pt];
	\fill (1,0) circle[radius=1pt];
	\node at (0.87,0) {$p$};
	
	\draw (2.5,0) circle [radius=0.5];
	\draw (2,0) to [out=270,in=180] (2.5,-0.25);
	\draw (2.5,-0.25) to [out=0,in=270] (3,0);
	\draw[dotted] (2,0) to [out=90,in=180] (2.5,0.25);
	\draw[dotted] (2.5,0.25) to [out=0,in=90] (3,0);
	
	\draw (3.5,0) circle [radius=0.5];
	\draw (3,0) to [out=270,in=180] (3.5,-0.25);
	\draw (3.5,-0.25) to [out=0,in=270] (4,0);
	\draw[dotted] (3,0) to [out=90,in=180] (3.5,0.25);
	\draw[dotted] (3.5,0.25) to [out=0,in=90] (4,0);

	\draw (4,0) to [out=90, in = 180] (6,0.5);
	\draw (4,0) to [out=270, in = 180] (6,-0.5);
	
	\draw[dotted] (6,0.5) to [out=180, in=90] (5.75, 0);
	\draw[dotted] (5.75, 0) to [out=270, in=180] (6,-0.5);
	\draw (6,0.5) to [out=0, in = 90] (6.25,0);
	\draw (6.25,0) to [out = 270, in = 0] (6,-0.5);
	
	\draw (6,-0.5)--(6,-0.75);
	\draw[<-] (6,-0.75)--(6,-1);
	
	\draw (6,-1) -- (8,-1);
	\draw (6,-2) -- (8,-2);
	
	\draw (6,-1) to [out=180, in=90] (5.75, -1.5);
	\draw (5.75, -1.5) to [out=270, in=180] (6,-2);
	\draw (6,-1) to [out=0, in = 90] (6.25,-1.5);
	\draw (6.25,-1.5) to [out = 270, in = 0] (6,-2);
	
	\draw[dotted] (8,-1) to [out=180, in=90] (7.75, -1.5);
	\draw[dotted] (7.75, -1.5) to [out=270, in=180] (8,-2);
	\draw (8,-1) to [out=0, in = 90] (8.25,-1.5);
	\draw (8.25,-1.5) to [out = 270, in = 0] (8,-2);
	
	\draw (8,-2)--(8,-2.25);
	\draw[<-] (8,-2.25)--(8,-2.5);
	
	\draw (8,-2.5) circle[radius=1pt];
	\fill (8,-2.5) circle[radius=1pt];
	\node at (8.2,-2.5) {$\overline{y}$};
	\draw[->] (9,0.5)--(9,0.75);
	\draw (9,0.75)--(9,1);
	
	\draw (9,1) circle[radius=1pt];
	\fill (9,1) circle[radius=1pt];
	\node at (9.2,1) {$\overline{x}$};

	\draw (9,0.5) -- (11,0.5);
	\draw (9,-0.5) -- (11,-0.5);
	
	\draw (9,0.5) to [out=180, in=90] (8.75, 0);
	\draw (8.75, 0) to [out=270, in=180] (9,-0.5);
	\draw (9,0.5) to [out=0, in = 90] (9.25,0);
	\draw (9.25,0) to [out = 270, in = 0] (9,-0.5);
	
	\draw[dotted] (11,0.5) to [out=180, in=90] (10.75, 0);
	\draw[dotted] (10.75, 0) to [out=270, in=180] (11,-0.5);
	\draw (11,0.5) to [out=0, in = 90] (11.25,0);
	\draw (11.25,0) to [out = 270, in = 0] (11,-0.5);

	\draw (10,-1) circle [radius=0.5];
	\draw (9.5,-1) to [out=270,in=180] (10,-1.25);
	\draw (10,-1.25) to [out=0,in=270] (10.5,-1);
	\draw[dotted] (9.5,-1) to [out=90,in=180] (10,-0.75);
	\draw[dotted] (10,-0.75) to [out=0,in=90] (10.5,-1);
	
	\draw (11,-0.5)--(11,-0.75);
	\draw[<-] (11,-0.75)--(11,-1);

	\draw (11,-1) -- (13,-1);
	\draw (11,-2) -- (13,-2);
	
	\draw (11,-1) to [out=180, in=90] (10.75, -1.5);
	\draw (10.75, -1.5) to [out=270, in=180] (11,-2);
	\draw (11,-1) to [out=0, in = 90] (11.25,-1.5);
	\draw (11.25,-1.5) to [out = 270, in = 0] (11,-2);
	
	\draw[dotted] (13,-1) to [out=180, in=90] (12.75, -1.5);
	\draw[dotted] (12.75, -1.5) to [out=270, in=180] (13,-2);
	\draw (13,-1) to [out=0, in = 90] (13.25,-1.5);
	\draw (13.25,-1.5) to [out = 270, in = 0] (13,-2);
	
	\draw (13,-2)--(13,-2.25);
	\draw[<-] (13,-2.25)--(13,-2.5);
	
	\draw (13,-2.5) circle[radius=1pt];
	\fill (13,-2.5) circle[radius=1pt];
	\node at (13.2,-2.5) {$\overline{y}$};
	
	\draw (13,0)--(13.5,0);
	\draw[<-] (13.5,0)--(14,0);
	\draw (13,0) circle[radius=1pt];
	\fill (13,0) circle[radius=1pt];
	\node at (12.87,0) {$p$};
	
	\draw (14,0) to [out=90, in = 180] (16,0.5);
	\draw (14,0) to [out=270, in = 180] (16,-0.5);
	
	\draw (15,1.05) circle[radius=0.5];
	\draw (14.5,1.05) to [out=270,in=180] (15,0.8);
	\draw (15,0.8) to [out=0,in=270] (15.5,1.05);
	\draw[dotted] (14.5,1.05) to [out=90,in=180] (15,1.3);
	\draw[dotted] (15,1.3) to [out=0,in=90] (15.5,1.05);
	
	\draw[dotted] (16,0.5) to [out=180, in=90] (15.75, 0);
	\draw[dotted] (15.75, 0) to [out=270, in=180] (16,-0.5);
	\draw (16,0.5) to [out=0, in = 90] (16.25,0);
	\draw (16.25,0) to [out = 270, in = 0] (16,-0.5);
	
	\draw (16,-0.5)--(16,-0.75);
	\draw[<-] (16,-0.75)--(16,-1);
	
	\draw (16,-1) -- (18,-1);
	\draw (16,-2) -- (18,-2);
	
	\draw (16,-1) to [out=180, in=90] (15.75, -1.5);
	\draw (15.75, -1.5) to [out=270, in=180] (16,-2);
	\draw (16,-1) to [out=0, in = 90] (16.25,-1.5);
	\draw (16.25,-1.5) to [out = 270, in = 0] (16,-2);
	
	\draw[dotted] (18,-1) to [out=180, in=90] (17.75, -1.5);
	\draw[dotted] (17.75, -1.5) to [out=270, in=180] (18,-2);
	\draw (18,-1) to [out=0, in = 90] (18.25,-1.5);
	\draw (18.25,-1.5) to [out = 270, in = 0] (18,-2);
	
	\draw (18,-2)--(18,-2.25);
	\draw[<-] (18,-2.25)--(18,-2.5);
	
	\draw (18,-2.5) circle[radius=1pt];
	\fill (18,-2.5) circle[radius=1pt];
	\node at (18.2,-2.5) {$\overline{y}$};
	\end{tikzpicture}	
	\caption{Special components.}\label{fig:reduce}
\end{figure}
The single sphere bubble above is simple and the link of spheres in the left is in $M_{n,A}(J)$. To see that we can reduce to the above three case, if we have a curve in left of Figure \ref{fig:general}, then the bubble tree in the left can be reduced to a simple stable map by \cite[Proposition 6.1.2]{mcduff2012j}. We can pick out the link of spheres containing the two marked points and stabilize it by collapse unstable constant spheres. If the two marked points are on non-constant sphere(s), then we get the first case in Figure \ref{fig:reduce}. If the two marked points are on a constant sphere, then we get the third case in Figure \ref{fig:reduce}. The other cases in Figure \ref{fig:general} are similar.  

By construction, the moduli space in Figure \ref{fig:reduce} are cut out transversely. The expected dimensions after module the reparametrization action are
$|p|-|\overline{y}|-3-2(k-1),  |\overline{x}|-|\overline{y}|-3, |p|-|\overline{y}|-3$ respectively, where $k$ is the number of spheres. All of them are negative. Hence $\cM_{*,*,A}$ is compact when expected dimension  $\le 0$ (also holds for dimension $\le 1$ when adding Morse breaking and fiber product breaking). The proof of boosting $C^l$ almost complex structure to $C^\infty$ follows from the same argument in \cite[Theorem 5.1]{hofer1995floer}.
\end{proof}

Then for $J\in \cJ^{reg}(f,g)$, we can define cochain complexes $C(\bF,J,f),C_+(\bF,J),C_0(f)$ are free $\Lambda$-module generated by both $\cP^*(\bF)$ and $\cC(f)$, only $\cP^*(\bF)$, or only $C_0(f)$ respectively. And the differential is defined to be 
$$dy:= \sum_{\dim \cM_{x,y,A}=0}\#\cM_{x,y,A}T^{\omega(A)}x.$$
\begin{proposition}\label{prop:d}
	$d$ is a differential.
\end{proposition}
\begin{proof}
	Since the differential increases action $\cA_{\bF}$ and preserve the length filtration from $r_i$ by Proposition \ref{prop:increase}, it is sufficient to verify $d^2=0$ on orbits inside $W^i$ within a bounded action window. In such case, we have compactness and \cite[Theorem 3.7]{bourgeois2009symplectic} can be applied to verify $d^2=0$.  
\end{proof}
As usual, $C(\bF,J,f),C_+(\bF,J),C_0(f)$ form a short exact sequence, which induces a long exact sequence. We also have subcomplexes $C^{r_i}$, $C^{r_i}_+$ generated by orbits inside $W^i$. By definition, $C$, $C_+$ are direct limits of them. On the other hand, we have cochain complexes by action truncation. In particular, we will consider $C^{r_i,\ge 0}$ and $C^{r_i,\ge 0}_+$, i.e. the complex only with elements of non-negative action inside $W^i$. $C^{r_i,\ge 0}$ and $C^{r_i,\ge 0}_+$ are $\Lambda_0$ modules, and we have $C^{r_i,\ge 0}\otimes_{\Lambda_0}\Lambda = C^{r_i}$ and $C^{r_i,\ge 0}_+\otimes_{\Lambda_0}\Lambda = C^{r_i}_+$. Since $\Lambda$ is a flat $\Lambda_0$-module, the relation pass to cohomology. We may consider even smaller cochain complex $C^{r_i,[0,j]}_+$ of elements with action in $[0,j]$. On such complex, we have an energy control, hence neck-stretching can be applied.

On the other hand, like the exact case, symplectic cohomology can also be defined as a direct limit using non-degenerate Hamiltonian with finite slope when sphere bubbles can be avoided (e.g. $c_1=0$), see \cite{ritter2014floer}. 

\begin{proposition}\label{prop:cascadeCY}
	There is an isomorphism $SH^*(W;\Lambda)\to H^*(C(\bH,J,f))$ such that the following diagram commute
	$$\xymatrix{
		QH^*(W) \ar[r]\ar[d] & SH^*(W;\Lambda) \ar[d] \\
		H^*(C_0(f)) \ar[r] & H^*(C(\bF,J,f))
	}	
	$$
\end{proposition}
\begin{proof}[Sketch of the proof]
	Similar to Proposition \ref{prop:cascades}, we first perturb $\bF$ to $F$, such that on $W$, $F$ is a $C^2$-small time-independent Hamiltonian but $F=\bF$ on the cylindrical end. Then there is a cascade continuation map \cite[\S 2]{bourgeois2009exact} from the cochain complex of $F$ to the cochain complex of $\bF$ preserving the $r_i$ filtration. As in Proposition \ref{prop:cascades}, it is quasi-isomorphism as it induces isomorphism on the first page of the spectral sequence (which is convergent). Then similar to the discussion in \S \ref{sub:nondeg}, we first get $F_i$ with finite slope, which is the same as $F$ on $W^i$ but linear afterwards. Then cohomology of $C(F)$ can be write as direct limit of $H^*(C(F_i))$.  The last step is perturbing $F_i$ into non-degenerate Hamiltonians, then using the cascade continuation map, we have that the direct limit using non-degenerate Hamiltonians with finite slop is the same as $\varinjlim_{i} H^*(C(F_i))$. Since the inclusion from $C_0(f) \to C(\bF,J,f)$ can also be viewed as continuation map from a homotopy of zero slope truncation of $\bF$ to $\bF$ by Lemma \ref{lemma:max}. Then all of the construction above is compatible with continuation maps from the $0$-length part\footnote{Note that for a general non-degenerate Hamiltonian $H$ with finite slope, Proposition \ref{prop:exclude} may not apply, the $0$-length part may not be a subcomplex.}, this finishes the proof.
\end{proof}

From here, we give an ad hoc definition of $SH^*_+(W;\Lambda)$ motivated from Proposition \ref{prop:cascadeCY}, 
$$SH^*_+(W) = \coker(QH^*(W)\to SH^*(W;\Lambda))\oplus \ker (QH^{*+1}(W)\to SH^{*+1}(W;\Lambda)),$$
since $\Lambda$ is field, Proposition \ref{prop:cascadeCY} implies that there is a (non-canonical) isomorphism $SH^*_+(W) \to H^*(C_+(\bF,J))$. We do not claim this definition is good in a functorial way, but it is sufficient for our application.

\subsection{Proof of the independence}
Similarly, we can define moduli spaces $\cP_{p,\overline{x},A}$ and $\cH_{p,\overline{x},A}$ as in \S \ref{sub:shrink}. Combining with the proofs in Proposition \ref{prop:transverseCY} and Proposition \ref{prop:d}, \S \ref{sub:shrink} can be generalized to strong fillings with vanishing first Chern class with the same statements. Hence we will use the same terminology, in particular we have various regular sets of almost complex structures. The following Proposition follows from the same proof of Proposition \ref{prop:compute}.

\begin{proposition}\label{prop:naturalCY}
	Let $J^{i,j}\in \cJ^{r_i,[0,j]}_{reg,+}(\bF)\cap \cJ^{r_i,[0,j]}_{reg,P}(\bF,h,g_{\partial})$, then we have a commutative diagram
	$$\xymatrix{
		H^*(C^{r_1,[0,1]}_+(J^{1,1})) \ar[r] & H^*(C^{r_2,[0,1]}_+(J^{2,1})) \ar[r] & \ldots \ar[r] & H^{*+1}(\partial W) \times \Lambda^{[0,1]}\\
		H^*(C^{r_1,[0,2]}_+(J^{1,2})) \ar[r]\ar[u] & H^*(C^{r_2,[0,2]}_+(J^{2,2})) \ar[r]\ar[u] & \ldots  \ar[r] & H^{*+1}(\partial W) \times \Lambda^{[0,2]}\ar[u] \\		
		\vdots \ar[u] & \vdots \ar[u] & & \vdots \ar[u]
	}$$
	where vertical and horizontal arrows are continuation maps except those mapped to the last column and $\Lambda^{[0,j]}$ consists of elements in $\Lambda$ with action in $[0,j]$. The arrow to the last column is defined by counting $\cP_{*,*}$. Then $\displaystyle{\varinjlim_{i\to \infty}\varprojlim_{j\to \infty}}$ of the digram computes $SH^{\ge 0}_+(W) \to H^{*+1}(\partial W;\Lambda_0)$. 
\end{proposition}

Now let $Y$ be a TADC contact manifold with two strong fillings $W_1,W_2$, we assume $\alpha_i$ in Definition \ref{def:TADC} is represented by  a nested sequences $Y_i\in Y\times [1,R]$. Then we view $W_*\cup Y\times [1,R]$ as the new $W_*$, in particular $\bF$ is zero on this new $W_*$. Then we can stretch $Y_i$ in the same way as in \S \ref{sub:ind}.
\begin{proposition}\label{prop:ind3}
	With the setup above, there exist admissible $J^{1}_1,J^2_1,\ldots $ and $J^{1}_2,J^2_2,\ldots $ on $\widehat{W_1}$ and $\widehat{W_2}$ respectively and positive real numbers $\epsilon_{i,j},i,j\in \N_+$ such that the following holds.
	\begin{enumerate}
		\item\label{nsCY1} We have $\epsilon_{i,j}>\epsilon_{i,j+1}$. For $R<\epsilon_{i,j}$ and any $R'$, $NS_{i,R}(J^{i}_*), NS_{i+1,R'}(NS_{i,R}(J^{i}_*))\in \cJ^{r_i,[0,j]}_{reg,+}\cap \cJ^{r_i,[0,j]}_{reg,P}(h,g_{\partial})$. Such that all zero dimensional $\cM_{x,y,A}$ and $\cP_{p,x,A}$ are the same for both $W_1,W_2$ and contained outside $Y_i$ for $x,y\in C^{r_i}_+,p\in C(h)$ with action change at most $j$.  
		\item\label{nscY2} For every $i$, there exists $i_j\in \N_+$ such that $J^{i+1}_*=NS_{i,\frac{\epsilon_{i,i_j}}{2}}(J^{i}_*)$ on $W^i_*$. 
	\end{enumerate}
\end{proposition}
\begin{proof}
	The proof is exactly same as Proposition \ref{prop:ind}. We start with $J^1$ such that $NS_{1,0}(J^1)\in \cJ^{r_1}_{reg,SFT}(\bF,h,g_{\partial})$. The threshold $\epsilon_{i,j}$ will dependent on $j$, because we need compactness to apply neck-stretching. $\epsilon_{i,j+1}$ is smaller than $\epsilon_{i,j}$, as $\cJ^{r_i,[0,j+1]}_{reg}\subset \cJ^{r_i,[0,j]}_{reg}$.  To prove the second property, note that all curves in Figure \ref{fig:SFT} for $x,y\in C^{r_i}_+$ have a universal bound depending only on $i$. Therefore there exists $i_j$, such that $NS_{i+1,0}(NS_{i,\frac{\epsilon_{i,i_j}}{2}})(J^i)\in \cJ^{r_i}_{reg,SFT}(\bF, h,g_{\partial})$, since the related moduli spaces are contained outside $Y_i$ by the first property. Hence we can choose $J^{i+1}$ such that $NS_{i+1,0}(J^{i+1})\in \cJ^{r_{i+1}}_{reg,SFT}(\bF, h,g_{\partial})$ and $J^{i+1}=NS_{i,\frac{\epsilon_{i,i_j}}{2}}(J^i)$ on $W^i$ as in Proposition \ref{prop:ind}.
\end{proof}

\begin{theorem}\label{thm:CY}
Let $Y$ be a TADC manifold, then $SH^*_+(W;\Lambda)\to H^{*+1}(Y;\Lambda)$ is independent of the filling. 
\end{theorem}
\begin{proof}
	Using the almost complex structures in Proposition \ref{prop:ind3},  by Proposition \ref{prop:naturalCY}, we have that the part where $j\ge i_j$ of the following diagram computes $SH_+^{\ge 0}(W) \to H^{*+1}(Y;\Lambda_0)$,
	$$\xymatrix{
	H^*(C^{r_1,[0,1]}_+(NS_{1,\frac{\epsilon_{1,1}}{2}}J^{1}_*)) \ar[r] & H^*(C^{r_2,[0,1]}_+(NS_{2,\frac{\epsilon_{2,1}}{2}}J^{2}_*)) \ar[r] & \ldots \ar[r] & H^{*+1}(\partial W) \times \Lambda^{[0,1]}\\
	H^*(C^{r_1,[0,2]}_+(NS_{1,\frac{\epsilon_{1,2}}{2}}J^{1}_*)) \ar[r]\ar[u] & H^*(C^{r_2,[0,2]}_+(NS_{2,\frac{\epsilon_{2,2}}{2}}J^{2}_*)) \ar[r]\ar[u] & \ldots  \ar[r] & H^{*+1}(\partial W) \times \Lambda^{[0,2]}\ar[u] \\		
	\vdots \ar[u] & \vdots \ar[u] & & \vdots \ar[u]
    }$$	
	By the same argument in the proof of Theorem \ref{thm:A}, the horizontal continuation maps are inclusions for $j\ge i_j$.  The vertical continuation map can be decomposed into continuation maps 
	$C_+^{r_i,[0,j+1]}(NS_{i,\frac{\epsilon_{i,j+1}}{2}}J^i_*)\to C_+^{r_i,[0,j]}(NS_{i,\frac{\epsilon_{i,j+1}}{2}}J^i_*)$ and $C_+^{r_i,[0,j]}(NS_{i,\frac{\epsilon_{i,j+1}}{2}}J^i_*)\to C_+^{r_i,[0,j]}(NS_{i,\frac{\epsilon_{i,j}}{2}}J^i_*)$. The former map from the trivial homotopy of almost complex structure is the obvious quotient, the latter map is homotopic to identity by Lemma \ref{lemma:local}. Therefore the part where $j\ge i_j$ of the diagram is identified for both $W_1,W_2$. Hence $SH^{\ge 0}_+(W)\to H^{*+1}(Y;\Lambda_0)$ is independent of the filling. Since $\Lambda$ is a flat $\Lambda_0$ module, the claim follows from tensoring $\Lambda$.
\end{proof}

\begin{corollary}\label{cor:vanish}
	If $Y$ is a TADC contact manifold, then $SH^*(W;\Lambda)=0$ is a property independent of topologically simple strong fillings.
\end{corollary}
\begin{proof}
	Since $SH^*(W;\Lambda)=0$ is equivalent to $QH^0(W)\to SH^0(W;\Lambda)$ maps  $1$ to $0$ since it a unital ring map \cite[Corollary 14]{ritter2014floer}, which is equivalent to $1\in \Ima(SH^{-1}_+(W;\Lambda)\to H^0(W;\Lambda)\to H^0(Y;\Lambda))$, which by Theorem \ref{thm:CY} is independent of such fillings.
\end{proof}

\begin{proof}[Proof of Theorem \ref{thm:D}]
	By Theorem \ref{thm:CY} and Corollary \ref{cor:vanish}, the map $H^*(W;\Lambda)\to H^*(Y;\Lambda)$ is independent of the filling. Hence for any other topologically simple strong filling $W'$, we have $H^2(W';\Q) \to H^2(Y;\Q)$ is injective and $H^1(W';\Q)\to H^1(Y;\Q)$ is surjective. They imply that there exists one form $\beta$ on $W'$, such that $\omega = \rd \beta $ and $\beta=\alpha$ near $Y$. That is $W'$ is exact.
\end{proof}

\begin{remark}\label{rmk:exact}
	Another similar case is for a TADC $Y:=\partial(V\times \C)$, then by the same argument for any topologically simple strong filling $W$, have $H^2(W;\Q)\to H^2(Y;\Q)$ is surjective. Therefore the symplectic form $\omega$ must be exact. It may not be true that there is a primitive restricted to a contact form. But such exactness is enough to rule out all sphere bubbles, hence the invariance can be lifted to $\Z$ coefficient by Corollary \ref{cor:B}. 
\end{remark}

Theorem \ref{thm:CY} also holds for symplectic cohomology with local systems (using Novikov field over $\C$), as a combination of twisted version of Theorem \ref{thm:CY} and Theorem \ref{thm:local}, we have the following.

\begin{theorem}\label{thm:final}
	Assume $Q$ is in Theorem \ref{thm:local} and $\dim Q\ge 4$. Then any strong filling $W$ of $ST^*Q$ with vanishing first Chern class satisfies other topological conditions in Theorem \ref{thm:local} is exact.
\end{theorem}

\begin{remark}\label{rmk:SFT2}
	For general ADC manifolds, $Y_i$ may collapse to zero. Since non-exact filling $W$ does not contain $Y\times (0,1)$. Therefore we can only seek room in $Y\times (1,\infty)$. In particular, we need to accommodate the increasing of $\supp \bF$. However, in addition to the problem in Remark \ref{rmk:reverse}, it seems to be very difficult to arrange homotopies of Hamiltonians, so that integrated maximum principle can be applied to exclude contributions from the $0$-length generators to the positive length generators. Requiring $\supp \bF$ converges to $\infty$ indicates that it might be better to define the theory on the SFT level, as in SFT, the vanishing of contributions of the zero length generators to the positive length generators is automatic. In particular Theorem \ref{thm:CY} should generalize to ADC manifolds by considering the theory defined using SFT, see Remark \ref{rmk:SFT}.
\end{remark}

\begin{remark}
	For ADC but not TADC contact manifolds, it is still possible to get some information about the strong filling. For example, if we know that $1$ is killed by Hamiltonian orbits with action at least $-D$ in the exact filling. Then possibly after rescaling, we have a contact hypersurface $Y'$ inside the strong filling such that Reeb orbits with period smaller than $D$ are non-degenerate and have positive degree. Then the argument in Theorem \ref{thm:CY} implies that the symplectic cohomology of the strong filling with vanishing first Chern class is zero as we have found the primitive of $1$. In general, our argument proves that $\Ima(SH_+^*(W;\Lambda)\to H^{*+1}(Y;\Lambda))$ is an invariant, while the invariance of the map can not be obtained. 
\end{remark}
\appendix
\section{Orientations}\label{app}
\subsection{Coherent orientations in \S \ref{s2}, \S \ref{s3} and \S \ref{s8}}
Following \cite[\S 1.4]{abouzaid2013symplectic}, for every non-constant periodic orbit $x\in \cP^*(\bH)$, one can associate an orientation line $o_x$, which is the determinant line bundle of the following operator 
\begin{equation}\label{eqn:cap}
D_x:W^{1,p}(\C,\C^n)\to L^p(\C,\C^n), \quad X\mapsto \partial_sX + I(\partial_tX-B \cdot X),
\end{equation}
where $\C$ is equipped with the negative cylindrical end $\R\times S^1 \to \C, (s,t)\mapsto e^{-2\pi(s+ it) }$, $I$ is the complex structure on $\C^n$ and $B=\Psi'(t)\cdot \Psi(t)^{-1}$ when $s\ll 0$, here $\Psi(t)$ is the a path in $Sp(2n)$ determined by the linearization of the Hamiltonian flow of $X_{\bH}$ around $x$ and a trivialization of $x^*(TW)$. By \cite[Proposition 1.4.10]{abouzaid2013symplectic}, determinant bundles of $D_x$ using different choices of trivializations are conically isomorphic. Hence $o_x$ is well-defined. Moreover, we have $\ind D_x = |x|$.

On the other hand, if we equip $\C$ with the positive cylindrical end \cite[\S 1.4.3]{abouzaid2013symplectic}, then \eqref{eqn:cap} induces another determinant line $o^+_x$. Note that in this case, \eqref{eqn:cap} is the linearization of the equation of $B_x$ \eqref{eqn:B}, therefore we have a conical isomorphism $\la \partial_s \ra \otimes \det B_x = o^+_x$. By the gluing property of such determination bundles \cite[Lemma 1.4.5]{seidel2012symplectic}, we have a canonical isomorphism $o^+_x\otimes o_x = \det \C^n$. 

Let $u$ be a solution to the Floer equation with negative end asymptotic to $x\in \cP^*(\bH)$ and positive end asymptotic to $y\in \cP^*(\bH)$. After choosing a trivialization of $u^*TM$, the linearized operator of the Floer equation is in the form of 
\begin{equation}
D_u:W^{1,p}(\R\times S^1,\C^n)\to L^p(\R \times S^1,\C^n), \quad X\mapsto \partial_sX + I(\partial_tX-B \cdot X),
\end{equation}
where asymptotics of $B$ are determined by the linearization of the Hamiltonian flow of $X_{\bH}$ near two ends as before. Let $o_u$ denote the determinant line. When $\cM_{x,y}$ is a manifold, they form a continuous bundle $o_{x,y}$. Then with a regular $J$ in Proposition \ref{prop:transversality}, the gluing property of such bundles \cite[Lemma 1.4.5]{seidel2012symplectic} yields the following structures.
\begin{enumerate}
	\item Canonical isomorphisms $\rho_{x,y}:o_{x,y}\otimes o_y \to o_x$, $\rho_{x,y,z}:o_{x,y}\otimes o_{y,z}\to o_{x,z}$ on $\cM_{x,y}\times \cM_{y,z}\subset \cM_{x,z}$ such that $\rho_{x,z}\circ \rho_{x,y,z}=\rho_{x,y}\circ \rho_{y,z}$ on $o_{x,y}\otimes o_{y,z}\otimes o_z$. 
	\item\label{ori2} $o_{x,y} = \det (\langle \partial_s \rangle \oplus T\cM_{x,y})$ and the $\rho_{x,y,z}$ is induced by a map $(\langle \partial_{s_1} \rangle \oplus T\cM_{x,y})\oplus(\langle \partial_{s_2} \rangle \oplus T\cM_{y,z}) \to (\langle \partial_s \rangle \oplus T\cM_{x,z})$ with the property that $\partial_{s_1}+\partial_{s_2}$ is mapped to $\partial_{s}$ and $\partial_{s_2}-\partial_{s_1}$ is mapped to the out normal vector of $T\cM_{x,z}$, c.f. \cite[Lemma 1.5.7]{abouzaid2013symplectic}.
	\item\label{ori3} Canonical isomorphism $\rho^+_{x,y}:o_x^+\otimes o_{x,y} \to o^+_y$ over $\cB_x\times \cM_{x,y}\subset \cB_y$, which is induced by a map $\la \partial_{s_1} \ra \oplus T\cB_x \oplus \la \partial_{s_2} \ra \oplus T \cM_{x,y}\to \la \partial_s \ra \oplus T\cB_y$ with $\partial_{s_1}+\partial_{s_2}$ mapped to $\partial_s$ and $\partial_{s_2}-\partial_{s_1}$ mapped the out normal vector.
\end{enumerate}
If we fix orientations on $o_x$, then there are induced orientations on $o_{x,y}$ and $o^+_x$, which determine orientations of $\cM_{x,y}$ and $\cB_x$ by quotienting out the $\R $ factor from the left. 

Next we orient the Morse theory part. For a Morse function $f$, let $S_p$, $U_p$ denote the stable and unstable manifold of $\nabla_g f$. Then there is a conical isomorphism $T_p U_p\oplus T_pS_p = T_pW$. Moreover, we have $\la \partial_s \ra\oplus T\cM_{p,q} = T(S_p \cap U_q)$, the latter at $p$ has a natural isomorphism to $TS_p/TS_q$. Therefore if we fix orientations for every $S_p$, then there is an induced orientation on $\cM_{p,q}$. When $m$ is the unique local minimum, we orient $S_m$ such that the induced orientation $U_m$ coincide with orientation of $W$, this guarantees the identity is generated by $m$. Since $\cM_{p,y}$ is the fiber product $S_p \times_W \cB_y$, we orient $\cM_{p,y}$, such that the isomorphism $T\Delta_{W} \oplus T\cM_{p,y}\to TS_p \oplus T\cB_y$ preserves the orientation (it is actually twisted by $(-1)^{\dim S_p \times \dim W}$ for general fiber product). 
\begin{remark}
	Our convention is from the following consideration: if we view $f$ as a Hamiltonian, we can assign two line bundles $o_p,o^+_p$ as before. Then there are conical isomorphism $o_p=\det S_p$ and $o^+_p=\det U_p$, because $D_p$ is the linearization of an equation whose solution corresponds to the stable manifold $S_p$. Similarly for $o^+_p$. Then the gluing of determinant bundle gives an isomorphism $o^+_p\otimes o_p = \det \C^n$ corresponding to $T_pU_p\oplus T_pS_p=T_pW$. Moreover $o_{p,q}=\det(TS_p/TS_q)$ and the gluing map $\rho_{p,q}:o_{p,q}\otimes o_q\to o_p$ is induced from the obvious map. As for the orientation convention for fiber products, it is different from the one used in \cite{diogo2019symplectic} by a sign twisting for general fiber products, but they coincide in the special case considered here since $\dim W = 2n$. Our fiber product orientation rule also satisfies associativity. 
\end{remark}

Therefore we have oriented all $\cM_{*,*}$, the following proposition shows that orientations are coherent in the sense that they imply $d^2=0$.
\begin{proposition}\label{prop:ori}
	For $1$-dimensional $\cM_{*,*}$, with orientations above, we have $\partial \cM_{*,*}=\sum \cM_{*,*}\times \cM_{*,*}$.
\end{proposition}
\begin{proof}
	For $x,y,z\in \cP^*(\bH)$, by property \eqref{ori3} above, we have an orientation preserving map  $(\langle \partial_{s_1} \rangle \oplus T\cM_{x,y})\oplus(\langle \partial_{s_2} \rangle \oplus T\cM_{y,z}) \to (\langle \partial_s \rangle \oplus T\cM_{x,z})$ over $\cM_{x,y}\times \cM_{y,z}$ with the property that $\partial_{s_1}+\partial_{s_2}$ is mapped to $\partial_{s}$ and $\partial_{s_2}-\partial_{s_1}$ is mapped to the out normal vector of $T\cM_{x,z}$. Hence the product orientation $\cM_{x,y}\times \cM_{y,z}$ is the boundary orientation. The situation for $\cM_{p,q}\times \cM_{q,r}$ for $p,q,r \in \cC(f)$ is similar. Next we consider $\cM_{p,x}\times \cM_{x,y}$ for $p\in \cC(f),x,y\in \cP^*(\bH)$. Then by property \eqref{ori3} above, we have an orientation preserving map $\la \partial_{s_1} \ra \oplus T\cB_x \oplus \la \partial_{s_2} \ra \oplus T \cM_{x,y}\to \la \partial_s \ra \oplus T\cB_y$ with $\partial_{s_1}+\partial_{s_2}$ mapped to $\partial_s$ and $\partial_{s_2}-\partial_{s_1}$ mapped the out normal vector. Then by our fiber product orientation rule, we have the product orientation $\cM_{p,x}\times \cM_{x,y}$ is the boundary orientation.  The situation for $\cM_{p,q}\times \cM_{q,x}$ for $p,q\in \cC(f),x\in \cP^*(\bH)$ is the same.
\end{proof}

To orient other moduli spaces in \S \ref{s2}, \S \ref{s3}, we need to assign orientation to $S_p$ for every $p\in \cC(h)$. Then for $\cR_{p,q}$, we have an isomorphism $T\cR_{p,q}= TS_p/TS_q$ at the intersection point on $Y$. Therefore orientations of $S_p$ for $p\in \cC(f)\cup \cC(h)$ determines orientation of $\cR_{p,q}$. The orientation of $\cP_{p,y}$ is determined by the fiber product rule. For $\cH_{p,y}$, the orientation is determined by fiber product rule and the orientation on the Morse part is given by $\la \partial_l \ra\oplus S_p$ for $p \in \cC(h)$. Then by the same argument in Proposition \ref{prop:ori}, we have the following. 

\begin{proposition}\label{prop:ori1}
	For moduli spaces up to dimension $1$, we have the induced orientations above satisfying the following,
	\begin{enumerate}
		\item $\partial \cR_{*,*} = \sum \cM_{*,*}\times \cR_{*,*}-\sum \cR_{*,*}\times \cM_{*,*}$;
		\item $\partial \cP_{*,*}= \sum \cM_{*,*}\times \cP_{*,*}+\sum \cP_{*,*}\times \cM_{*,*}$;
		\item\label{o4} $\partial \cH_{*,*}= \sum \cM_{*,*}\times \cH_{*,*}+\sum \cH_{*,*}\times \cM_{*,*} + \sum \cR_{*,*}\times \cM_{*,*}-\sum \cP_{*,*}$.
	\end{enumerate} 
\end{proposition}

The orientations for the continuation map moduli spaces in Proposition \ref{prop:cascades} and \S \ref{ss:natural} is done similarly, but there is no $\R$ factor needed to be quotiented.

For the orientation in \S \ref{s8}, by \cite[Theorem 3.7]{bourgeois2009symplectic}, the moduli space encountered in \S \ref{s8} has a one-to-one correspondence to the moduli space using non-degenerate Hamiltonians nearby. Hence we can use the orientation from a perturbed Hamiltonian, then we have coherent orientations from Proposition \ref{prop:ori} and Proposition \ref{prop:ori1}. Alternative approaches to orient the cascades moduli space directly can be found in \cite[\S 4.4]{bourgeois2009symplectic}, but they require a twisting depending the number of cascades \cite[Proposition 3.9]{bourgeois2009symplectic}.

\subsection{Coherent orientations in \S \ref{s4}}
We follow the orientation convention for the BV operator in \cite{seidel2012symplectic}. If we consider a solution $(u,\theta)$ to the equation \eqref{eqn:BV1}, i.e. 
$$\partial_s u + J^{\theta}_{s,t}(\partial_t u - X_{\bH^\theta_{s,t}}) = 0, \quad \quad \lim_{s\to -\infty} u = x(\cdot + \theta), \lim_{s\to \infty} u = y, x\in \cP^*(\bH_-), y\in \cP^*(\bH_+).$$
If transversality for $\cM^{\Delta}_{x,y}$ holds, let $D_u$ denote the linearization of the equation above at $u$, then we have a short exact sequence
$$0\to T_{(u,\theta)}\cM^{\Delta}_{x,y}\to \ker D_u \oplus TS^1 \to \coker D_u \to 0.$$
Since $\det D_u = o_{x,y}$, the short exact sequence induces an isomorphism $\det T\cM^{\Delta}_{x,y}=o_{x,y}\otimes TS^1$. Hence given orientations of $o_x,o_y$ and $S^1$, we have induced orientation on $\cM^{\Delta}_{x,y}$.  Then we can similarly orient for $\cM^{\Delta}_{p,y}, p\in \cC(f),y\in \cP^*(\bH_+)$, $\cP^{\Delta}_{*,*}$ and $\cH^{\Delta}_{*,*}$ by tensoring $TS^1$ from the right. One last type is $\cT_{*,*}$ in the construction of homotopy in \S \ref{ssb:comp}, similar to the BV operator case, for every solution $(u,\theta, r)$ to \eqref{eqn:BVhomotopy}, there is a short exact sequence
$$0\to T_{(u,\theta,r)}\cT_{x,y} \to \ker D_u \oplus TS^1 \oplus T\R \to \coker D_u \to 0,$$
which yields an isomorphism $\det T\cT_{x,y}=o_{x,y}\otimes \det S^1 \otimes \det \R$ for $x \in \cP^*(\bH_-)$ and $y\in \cP^*(\bH_+)$. Hence $\cT_{x,y}$ is oriented, and the orientation for $\cT_{p,y}$ is similar. Then we have the following result with the same proof of Proposition \ref{prop:ori}.
\begin{proposition}
	For moduli spaces up to dimension $1$, the orientations above satisfy the following
	\begin{enumerate}
		\item $\partial \cM^{\Delta}_{*,*} = -\sum \cM^{\Delta}_{*,*}\times \cM_{*,*}-\sum \cM_{*,*}\times \cM^{\Delta}_{*,*}$;
		\item $\partial \cT_{*,*}= -\sum \cM_{*,*}\times \cT_{*,*}+\sum \cT_{*,*}\times \cM_{*,*} + \sum \cN_{*,*}\times \cM^{\Delta}_{*,*}-\sum \cM^{\Delta}_{*,*}\times \cN_{*,*}$;
		\item $\partial\cP^{\Delta}_{*,*} = -\sum \cM_{*,*}\times \cP^{\Delta}_{*,*}-\sum \cP^{\Delta}_{*,*}\times \cM_{*,*}-\sum \cP_{*,*}\times \cM^{\Delta}_{*,*}$;
		\item $\partial \cH^{\Delta}_{*,*} = \sum \cR_{*,*}\times \cM^{\Delta}_{*,*}-\cP^{\Delta}_{*,*} + \sum \cH_{*,*}\times \cM^{\Delta}_{*,*}+\sum \cM_{*,*}\times \cH^{\Delta}_{*,*}-\sum \cH^{\Delta}_{*,*}\times \cM_{*,*}$. 
	\end{enumerate} 
\end{proposition}
	
\bibliographystyle{plain} 
\bibliography{ref}

\end{document}